\documentclass[10pt]{article}

\usepackage{amsmath}    
\usepackage{graphicx}   
\usepackage{verbatim}   
\usepackage{color}      
\usepackage{hyperref}   

\usepackage{amssymb}
\usepackage{amsthm}
\usepackage{mathrsfs}

\newtheorem{Def}{Definition}[section]
\newtheorem{Eg}[Def]{Example}
\newtheorem{Prop}[Def]{Proposition}
\newtheorem{Lem}[Def]{Lemma}
\newtheorem{Thm}[Def]{Theorem}
\newtheorem{Cor}[Def]{Corollary}
\newtheorem{Ass}[Def]{Assumption}

\theoremstyle{definition}
\newtheorem{Rem}[Def]{Remark}

\newcommand{\p}{\mathbb{P}}
\newcommand{\e}{\mathbb{E}}
\newcommand{\real}{\mathbb{R}}
\newcommand{\n}{\mathbb{N}}

\newcommand{\z}{\mathbb{Z}}

\newcommand{\1}{{\bf 1}}

\newcommand{\rd}{\mathrm{d}}

\newcommand{\ksm}{K_\sigma}

\allowdisplaybreaks

\begin{document}

\title{A generalized Avikainen's estimate and its applications}
\author{
	Dai Taguchi
	\footnote{
	Research Institute for Interdisciplinary Science, department of mathematics,
	Okayama University
	3-1-1
	Tsushima-naka,
	Kita-ku
	Okayama
	700-8530,
	Japan,
	email~:~\texttt{dai.taguchi.dai@gmail.com}
	}
}
\maketitle
\begin{abstract}
	Avikainen provided in \cite{Av09} a sharp upper bound of the difference $\mathbb{E}[|g(X)-g(\widehat{X})|^{q}]$ by the moments of $|X-\widehat{X}|$ for any one-dimensional random variables $X$ with bounded density and $\widehat{X}$, and function of bounded variation $g$.
	In this article, we generalize this estimate to any one-dimensional random variable $X$ with H\"older continuous distribution function.
	As applications, we provide the rate of convergence for numerical schemes for solutions of one-dimensional stochastic differential equations (SDEs) driven by Brownian motion and symmetric $\alpha$-stable with $\alpha \in (1,2)$, fractional Brownian motion with drift and Hurst parameter $H \in (0,1/2)$, and stochastic heat equations (SHEs) with Dirichlet boundary conditions driven by space--time white noise, with irregular coefficients.
	We also consider a numerical scheme for maximum and integral type functionals of SDEs driven by Brownian motion with irregular coefficients and payoffs which are related to multilevel Monte Carlo method.
\end{abstract}



\section{Introduction}\label{sec_1}
In this article, inspired by \cite{Av09}, we consider a generalization of Avikainen's estimate and its application to error estimates for numerical approximations of one-dimensional stochastic differential and heat equations (SDEs, SHEs) with irregular coefficients.
To explain the goal of this article, we first recall the estimate of Avikainen proved in \cite{Av09}.
Let $X$ be a one-dimensional random variable with bounded density $p_{X}$ with respect to Lebesgue measure.
Then for any random variable $\widehat{X}$, function of bounded variation $g:\real \to \real$ and $p,q \in [1,\infty)$, it holds that
\begin{align}\label{Av_0}
	\e\left[
		\left|
			g(X)
			-
			g(\widehat{X})
		\right|^{q}
	\right]
	\leq
	3^{q+1}
	V(g)^{q}
	\left(
		\sup_{x \in \real}
		p_X(x)
	\right)^{\frac{p}{p+1}}
	\e\left[
		\left|
			X
			-
			\widehat{X}
		\right|^p
	\right]^{\frac{1}{p+1}},
\end{align}
where $V(g)$ is the total variation of $g$.
This estimate is optimal and can be applied to the numerical analysis for the payoff of the binary option in mathematical finance based on the Euler--Maruyama scheme (see, \cite{KP}) and multilevel Monte Carlo methods (see, \cite{Gi08}).
The idea of the proof is to use Skorokhod's representation of the random variable for embedding to the probability space $([0,1], \mathscr{B}([0,1]), \mathrm{Leb})$ and to use the trivial estimate
$
	\p(a<X\leq b)
	\leq
	\sup_{x \in \real}|p_{X}(x)|
	(b-a)
$, which means the distribution function of $X$, defined by $F_{X}(x):=\p(X \leq x)$, is Lipschitz continuous.
We remake that the existence of bounded density is equivalent to the distribution function is Lipschitz continuous (see, Remark \ref{Rem_0} below).
With this consideration in mind, the estimate \eqref{Av_0} can be generalized as follows (see, Theorem \ref{main_0} below).
Let $(S,\Sigma,\mu)$ be a measure space with $\mu(S)<\infty$ and $f$ be a one-dimensional measurable function on $(S,\Sigma,\mu)$ such that $F_{f}(x):=\mu(f \leq x)$ is $\alpha$-H\"older continuous with $\alpha \in (0,1]$.
Then for any one-dimensional measurable function $\widehat{f}$ on $(S,\Sigma,\mu)$, function of bounded variation $g:\real \to \real$, $p \in (0,\infty)$ and $q \in [1,\infty)$, it holds that
\begin{align}\label{Av_1}
	&\int_{S}
		\left|
			g\circ f(x)
			-
			g\circ \widehat{f}(x)
		\right|^{q}
	\mu(\rd x)\\
	&\leq
	3^{q+1}V(g)^{q}	
	\|F_{f}\|_{\alpha}^{\frac{p}{p+\alpha}}
	\mu(S)^{\frac{p}{p+\alpha}}
	\left(
		\int_{S}
			\left|
				f(x)
				-
				\widehat{f}(x)
			\right|^{p}
		\mu(\rd x)
	\right)^{\frac{\alpha}{p+\alpha}} \notag.
\end{align}
where $\|f\|_{\alpha}:=\sup_{x,y \in \real,~x \neq y} \frac{|f(x)-f(y)|}{|x-y|^{\alpha}}$, thus we do not need to assume existence of bounded density for the random variable $X$.

In order to apply the original Avikainen's estimate \eqref{Av_0}, we need to consider the upper bound of the density function of $X$ and we can prove this in various ways as follows.
It is well-known that if the characteristic function of the random variable $X$ is in $L^{1}(\real)$, then by using L\'evy's inversion formula (see, e.g. \cite{Wi91}, page 175), $X$ has a bounded and continuous density.
Also, Malliavin calculus is a powerful tool to study the existence and upper bound for the density of smooth random variables in the Malliavin sense.
Indeed, let $B=(B(t))_{t \geq 0}$ be a one-dimensional standard Brownian motion and let $X=(X(t))_{t \geq 0}$ be a one-dimensional It\^o process of the form $\rd X(t)=u(t,\omega) \rd B(t)$ for some adapted process $u=(u(t))_{t \geq 0}$.
Then if $u(t) \in \mathbb{D}^{2,2}$, uniformly positive and its Malliavin derivatives satisfy some moment condition, then for $t>0$, $X(t)$ admits the density with respect to Lebesgue measure and it is bounded above by $c/\sqrt{t}$ for some constant $c>0$, (see, e.g. Proposition 2.1.3 in \cite{Nu06} and \cite{CaFeNu98}) and if $u$ is uniformly bounded, then it satisfies the Gaussian upper bound (see, Corollary 2.1.1 in \cite{Nu06}).
On the other hand, as analytical approach, it is well-known that by using Levi's parametrix method (see, \cite{Fr64}), we can construct the fundamental solution of parabolic type partial differential equations (Kolmogorov equation) which is the density function of a solution of associated SDEs as conclusion of Feynman-Kac formula, and if the coefficients of SDE are bounded, measurable and diffusion coefficient is uniformly elliptic and H\"older continuous, then the density satisfies the Gaussian two sided bound (see, also \cite{Ku17,TaTa18} for path-dependent or unbounded drift, and \cite{LeMe10} for the Gaussian two sided bound for the density of the Euler--Maruyama scheme).
Moreover, the existence and Gaussian two sided bound for the fundamental solution of the parabolic equations in divergence form are proved by Aronson  \cite{Ar67}.
It is worth noting that the diffusion processes associated to the parabolic equation in divergence form, that is, its infinitesimal generator is given by $\frac{1}{2} \frac{\rd }{\rd x}\left(\sigma^{2} \frac{\rd }{\rd x}\right)$, is not semi-martingale in general, and for one-dimensional case, the distributional derivative of $\sigma$ is a signed Radon measure if and only if the process is semi-martingale and has SDE representation of the form
$
	X(t)
	=
	x_{0}
	+
	\int_{0}^{t}
		\sigma(X(s))
	\rd B(s)
	+
	\frac{1}{2}
	\int_{\real}
		\sigma^{-1}(x)
		L^{x}_{t}(X)
	\sigma'(\rd x)
$, where $L^{x}(X)=(L_{t}^{x}(X))_{t \geq 0}$ is the local time of $X$ at the level $x$, (see, Theorem 3.6 in \cite{Bach01}).

In general, it is not trivial that the existence and upper bound for the density of random variables or one-dimensional stochastic processes, especially, It\^o processes $Y=(Y(t))_{t \geq 0}$ driven by Brownian motion $B=(B(t))_{t \geq 0}$ of the form $\rd Y(t)=b(t,\omega) \rd t+\sigma(t,\omega)\rd B(t)$, $Y(0)=y_{0} \in \real$.
Indeed, Fabes and Kenig \cite{FaKe81} provided an example of diffusion coefficient $\sigma$ such that the law of $X(t)=x_{0}+\int_{0}^{t} \sigma(s,X(s))\rd B(s)$ is purely singular with respect to Lebesgue measure (see, also Theorem 27.19 in \cite{Sato} and Theorem A in \cite{NoSi06} for absolute continuity of L\'evy processes).

On the other hand, in section \ref{sec_example}, we provide two examples of stochastic processes with H\"older continuous distribution function which is related to multi-level Monte Carlo method (\cite{Gi08}).
The first example is one-dimensional SDEs driven by Brownian motion with path-dependent and linear growth drift coefficient.
The second example is a maximum of SDEs with irregular drift coefficient.
For the both examples, under more strong assumptions on the drift coefficient, it is shown that the existence and Gaussian type estimate for the density (see, \cite{TaTa18}, \cite{Nak16}) by using the parametrix method and Malliavin calculus, thus, in these cases, we can apply original Avikainen's estimate \eqref{Av_0}.

For It\^o process $Y$, in order to apply a generalized Avikainen's estimate \eqref{Av_1} we need to check the H\"older continuity of the distribution function, and we can prove this as the following sense: if the coefficients $b$ and $\sigma$ of $Y$ are bounded and $\sigma$ is uniformly elliptic, then by using local time argument or Krylov estimate, it can be shown that the map $x \mapsto \int_{0}^{T} \p(Y(s) \leq x) \rd s$ is Lipschitz continuous (see, Proposition \ref{main_1} below).
Therefore, we can apply the estimate \eqref{Av_1} by considering It\^o process $Y$ as a measurable function on the product measure space $([0,T] \times \Omega, \mathscr{B}([0,T]) \otimes \mathscr{F}, \mathrm{Leb} \otimes \p)$.
We will apply this fact in section \ref{sec_3} to the error estimate of the Euler--Maruyama (type) scheme of one-dimensional SDEs and SHEs with irregular coefficients.
For more preciously, in section \ref{sec_3}, we will consider the following six cases:
\begin{itemize}
	\item[(i)]
	Subsection \ref{sec_3_2} : SDEs driven by Brownian motion $B$ of the form $\rd X(t)=b(t,X(t)) \rd t+\sigma(t,X(t))\rd B(t)$ with bounded $2$-variation diffusion coefficient, 
	and singular SDEs of the form $
		X(t)
		=
		x_{0}
		+
		\int_{0}^{t}
			\sigma(X(s))
		\rd B(s)
		+
		\int_{\real}
			L_{t}^{a}(X)
		\nu (\rd a)
	$.
	
	\item[(ii)]
	Subsection \ref{sec_3_3} : SDEs driven by Brownian motion with super-linearly growing and irregular coefficients.
	
	\item[(iii)]
	Subsection \ref{sec_3_4} : approximation of integral type functionals of SDEs.
	
	\item[(iv)]
	Subsection \ref{sec_3_5} : SDEs driven by symmetric $\alpha$-stable $Z$ with $\alpha \in (1,2)$ of the form $\rd X(t)=\sigma(X(t-))\rd Z(t),$ with bounded $\alpha$-variation coefficient.
	
	\item[(v)]
	Subsection \ref{sec_3_6} : SDEs driven by fractional Brownian motion $B^{H}$ with $H \in (0,1/2)$ of the form $\rd X(t)=b(X(t)) \rd t+\rd B^{H}(t)$ with globally one-sided Lipschitz and irregular drift.
	
	\item[(vi)]
	Subsection \ref{sec_3_7} : stochastic heat equations driven by space--time white noise of the form $\frac{\partial}{\partial t}u(t,x)=\frac{\partial^{2}}{\partial x^{2}}u(t,x)+b(t,x,u(t,x))+\sigma(t,x,u(t,x))\frac{\partial^{2}}{\partial t \partial x}W(t,x)$, with Dirichlet boundary conditions $u(t,0)=u(t,1)=0$ for $t \in [0,T]$.
\end{itemize}

Case (i):
Le Gall \cite{LeGall} showed that the pathwise uniequness holds for SDEs with bounded $2$-variation diffusion coefficient (see also \cite{YaWa} for $1/2$-H\"older continuous diffusion coefficient and \cite{Nakao} for bounded variation diffusion coefficient).
Under H\"older continuous setting, Gy\"ongy and R\'asonyi,  and Yan \cite{Ya02} provided the rate of convergence of the Euler--Maruyama scheme $X^{(n)}$ by using Yamada and Watanabe approximation technique or It\^o--Tanaka formula.
Moreover, recently the rate of convergence for the Euler--Maruyama scheme with irregular drift coefficients have been widely studied \cite{BuDaGe19,LeSz17b,MeTa,NgTa16,NgTa1701,NgTa1702}, (see, also \cite{LeSz,LeSz17,MuYa19} for transformed schemes, \cite{EiKoKrLa19} for backward scheme, \cite{Zh19} for application of discrete Krylov estimate for the Euler--Maruyama scheme for mean field SDEs with discontinuous coefficients,  and  \cite{HeHeMu19} for lower error bound for strong approximations of SDEs with with non-Lipschitz coefficients).
One of the crucial role in the proof of these previous works is to estimate the integral
\begin{align}\label{In_0}
	\int_{0}^{T}
		\e\left[
			\left|
				g(X^{(n)}(s))
				-
				g(X^{(n)}(\eta_{n}(s)))
			\right|^{q}
		\right]
	\rd s,
\end{align}
where $\eta _{n}(s) := kT/n$ if $s \in \left[kT/n, (k+1)T/n \right)$, $g \in \{b, \sigma\}$ and $q \geq 1$.
In particular in \cite{NgTa16, NgTa1701} by using the Gaussian upper bound proved in \cite{LeMe10} for the density of the Euler--Maruyama scheme $X^{(n)}$ with H\"older continuous diffusion coefficient, 
\eqref{In_0} can be estimated above by some polynomial of $T/n$.
However, if the diffusion coefficient $\sigma$ is of bounded $2$-variation, then it is difficult to prove the Gaussian upper bound for the density of the Euler--Maruyama scheme, because in \cite{LeMe10}, H\"older continuity is the important in order to use the parametrix method.
On the other hand, by applying a generalized Avikainen's estimate \eqref{Av_1}, we can estimate \eqref{In_0} with $g \in \{b, \sigma\}$ above by some polynomial of $T/n$, and thus in subsection \ref{sec_3_2}, we provide the rate of convergence for the Euler--Maruyama scheme.
As application, we consider approximation scheme for singular SDEs with local time.

Case (ii):
It is known that there exist locally Lipschitz continuous and smooth coefficients such that the standard Euler--Maruyama scheme does not converge at any polynomial rate in $L^{p}$-norm, (see, \cite{HaHuJe15,JeMuYa16,Yar17}).
Moreover, if the coefficients of SDEs grow super-linearly, then the standard Euler--Maruyama scheme does not converge to a solution of the equation (see, Theorem 2.1 in \cite{HuJeKl11}).
Hence in order to approximate a solution of SDEs with super-linearly growing coefficients, several tamed Euler--Maruyama schemes are proposed.
Moreover, the rate of convergence under globally one-sided Lipschitz and locally Lipschitz continuous drift coefficient are provided (see, e.g., \cite{HuJeKl12,NgoLu17,Sa13,Sa16}).
Inspired by these previous works, we provide a rate of convergence for a tamed Euler--Maruyama scheme under globally one-sided Lipschitz and irregular coefficients as application of a generalized Avikainen's estimate \eqref{Av_1}.

Case (iii):
Kohatsu-Higa and Tanaka \cite{KoTa12} proved that the existence of a smooth density of some additive functionals of SDEs by using Malliavin calculus, and Kohatsu-Higa and Makhlouf \cite{KoMa13} extended this to some integral type functionals with the uniform H\"ormander condition, which is related to Asian type options in mathematical finance and the observation processes in filtering problems, and proved some Gaussian type two sided bound.
Hence in this case we can use original Avikainen's estimate \eqref{Av_0} to study some approximation of these functionals.
In subsection \ref{sec_3_4}, in order to approximation the expectation of some integral type functionals of SDEs with irregular coefficients, we will apply a generalized Avikainenain's estimate \eqref{Av_1}.

Case (iv):
Belfadli and Ouknine \cite{BeOu08} showed that the pathwise uniequness holds for SDE driven by symmetric $\alpha$-stable with bounded $\alpha$-variation coefficient (see also \cite{Ko82} for $1/\alpha$-H\"older continuous coefficient).
Under H\"older continuous setting, Tsuchiya and Hashimoto \cite{TsHa13} provided the rate of convergence by using Komatsu's approximation technique (see, \cite{Ko82}).
In subsection \ref{sec_3_5}, by using similar way as in subsection \ref{sec_3_2}, we also provide the rate of convergence for the Euler--Maruyama scheme.

Case (v).
Nualart and Ouknine \cite{NuOu02} showed that if $H<1/2$ and $b$ is of linear growth or if $H>1/2$ and $b$ is $\gamma$-H\"older continuous with $\gamma \in (1-1/(2H),1)$, then there exists a unique strong solution of SDEs (see, Theorem 3, 5, 8 in \cite{NuOu02}) by using Krylov type estimate.
In subsection \ref{sec_3_6}, inspired by this work, as application of Krylov type estimate for the Euler--Maruyama scheme, we prove H\"older continuity of the map $x \mapsto \int_{0}^{T} \p(X^{(n)}(s) \leq x) \rd s$, and then we apply a generalized Avikainen's estimate \eqref{Av_1} in order to estimate \eqref{In_0} with $g=b$ and we provide the rate of convergence for the Euler--Maruyama scheme.

Case (vi).
Bally, Gy\"ongy and Pardoux \cite{BaGyPa94} proved that the existence and uniqueness of the stochastic heat equations  with Dirichlet boundary conditions in case the drift coefficient is measurable and satisfies a one-sided linear growth condition, that is, $ub(t,x,u) \leq K(1+|u|^{2})$ for some $K>0$, and the diffusion coefficient is non-generate, linear growth and has a locally Lipschitz derivative, by using a Krylov type estimate based on Girsanov transform and some density estimate as applications of Malliavin calculus (see, also \cite{GyPa93a,GyPa93b}).
On the other hand, Gy\"ongy \cite{Gy98, Gy99} introduced numerical schemes for a solution of stochastic heat equations with Dirichlet boundary conditions and provided its strong rate of convergence under Lipschitz conditions on the coefficients.
In subsection \ref{sec_3_7}, inspired by these works, as application of Krylov type estimate for a solution of stochastic heat equations proved in \cite{BaGyPa94}, we prove the map $x \mapsto \int_{0}^{T}\int_{0}^{1} \p(u(s,y) \leq x) \rd y \rd s$ is $\alpha$-H\"older continuous with $\alpha \in (0,1/2)$, and then we apply a generalized Avikainen's estimate \eqref{Av_1} in order to provide the weak rate of convergence for the numerical scheme introduced in \cite{Gy99}.

\subsection*{Notations}

We give some basic notations and definitions used throughout this article.
For $x \in \real$, $[x]$ stands for the integer part of $x$.
We denote by $C(A;\real)$ the space of continuous functions $w:A \to \real$ for $A=[0,\infty)$ or $[0,T]$.
Let $k \in \n$ and $C_{b}^{k}(\real;\real)$ be the space of $\real$-valued functions from $\real$ such that for all $0 \leq \ell \leq k$, $\ell$-th order derivatives are bounded.
For bounded and measurable function $f:\real \to \real$, the supremum norm of $f$ is defined by $\|f\|_{\infty}:=\sup_{x \in \real}|f(x)|$.
For a $\alpha$-H\"older continuous function $f:\real \to \real$ with exponent $\alpha \in (0,1]$, we define $\|f\|_{\alpha}:=\sup_{x \neq y}|f(x)-f(y)|/|x-y|^{\alpha}$.
We denote the sign function by $\mbox{sgn}(x):=-\1_{(-\infty,0]}(x)+\1_{(0,\infty)}(x)$ for $x \in \real$, and the gamma function by $\Gamma(x):=\int_{0}^{\infty} t^{x-1}e^{-t} \rd t$ for $x \in (0,\infty)$.
For given a measurable function $f:\real \to \real$, we define the total variation function of $f$ by $T_f(x):=\sup \sum_{j=1}^N|f(x_j)-f(x_{j-1})|$, where the supremum is taken over $N$ and all partitions $-\infty < x_{0} < x_{1} < \cdots < x_{N}=x$.
We say that $f$ is \emph{function of bounded variation}, denoted by $f \in BV$, if $V(f):=\lim_{x \rightarrow \infty }T_f(x)$ is finite, and call $V(f)$ the \emph{total variation} of $f$.
We denote $L^x(Y)=(L_t^x(Y))_{t \geq 0}$ by the symmetric local time of continuous semi-martingale $Y$ at the level $x \in \real$. 
We define $\eta _{n}(t):=kT/n=:t_{k}^{(n)}$ if $ t \in \left[kT/n, (k+1)T/n \right)$.
$\delta_{a}$ denote the Dirac point mass measure at $a \in \real$.

\section{A generalized Avikainen's estimate}\label{sec_2}

Let $(S,\Sigma,\mu)$ be a measure space with $\mu(S)<\infty$.
For a measurable function $f$, we denote $F_{f}(x):=\mu(f \leq x)$.
Then we have the following generalized Avikainen's estimate.

\begin{Thm}\label{main_0}
	Let $f:S \to \real$ be a measurable function on $(S,\Sigma,\mu)$ such that $F_{f}$ is $\alpha$-H\"older continuous with $\alpha \in (0,1]$.
	Then for any measurable function $\widehat{f}:S \to \real$ on $(S,\Sigma,\mu)$, $g \in BV$, $p \in (0,\infty)$ and $q \in [1,\infty)$, it holds that
	\begin{align}
		&\int_{S}
			\left|
				g\circ f(x)
				-
				g\circ \widehat{f}(x)
			\right|^{q}
		\mu(\rd x) \notag\\
		&\leq
		3^{q+1}V(g)^{q}	
		\|F_{f}\|_{\alpha}^{\frac{p}{p+\alpha}}
		\mu(S)^{\frac{p}{p+\alpha}}
		\left(
			\int_{S}
				\left|
					f(x)
					-
					\widehat{f}(x)
				\right|^{p}
			\mu(\rd x)
		\right)^{\frac{\alpha}{p+\alpha}}
		\label{main_0_0}.
	\end{align}
	
\end{Thm}

\begin{Rem}\label{Rem_0}
	\begin{itemize}
		\item[(i)]
		Let $(S,\Sigma,\mu)=(\Omega,\mathscr{F},\p)$ be a probability space.
		Then the existence for a bounded density of $X$ is equivalent to the distribution function $F_{X}$ is Lipschitz continuous.
		Indeed, if $X$ has a bounded density, then the Lipschitz constant of $F_{X}$ is bounded above by the supremum of the density.
		Conversely, we suppose $F_{X}$ is Lipschitz continuous.
		Then $F_{X}$ is absolute continuous, thus there exists a density $p_{X}$.
		Moreover, for almost every $x\in \real$, $p_{X}(x) = F_{X}'(x)=\lim_{h \to 0} \frac{F_{X}(x+h)-F_{X}(h)}{h} \leq \|F_{X}\|_{1}$, thus $X$ has a bounded density.
		On the other hand, in Theorem \ref{main_0}, we do not need to assume existence of bounded density for the random variable $X$.
		
		\item[(ii)]
		Let $X$ be a $\real$-valued random variable.
		If $X$ has $L^{r}(\real)$-integrable density $p_{X}$ for $r \in (1,\infty)$, then by using H\"older's inequality, 
		$
		F_{X}(b)-F_{X}(a)
		\leq
		\|p_{X}\|_{L^{r}(\real)}
		(b-a)^{\frac{r-1}{r}}
		$,
		thus the distribution function $F_{X}$ is H\"older continuous.
		For example, let $Y=(Y(t))_{t \geq 0}$ be a $\real$-valued stochastic process with $Y(0)=y_{0}$.
		Suppose for $t>0$, $Y(t)$ admits a density $p_{t}(y_{0},\cdot)$ satisfying Gaussian upper bound, that is, there exist $C>0$ and $c>0$ such that for each $y \in \real$, $p_{t}(y_{0},y) \leq Ce^{-\frac{|y-y_{0}|^{2}}{2ct}}/\sqrt{2\pi ct}$.
		Then for each $p>1$, $|Y(t)|^{p}$ has a density $\{p_{t}(y_{0},y^{1/p})+p_{t}(y_{0},-y^{1/p})\}p^{-1}y^{-(1-\frac{1}{p})} \1_{(0,\infty)}(y)$ which is $L^{r}(\real)$-integrable for each $r \in (1,p/(p-1))$.
		
		\item[(iii)]
		Let $Y=(Y(t))_{t \geq 0}$ be a $\real$-valued c\`adl\`ag semi-martingale or fractional Brownian motion with drift.
		In general, it is not trivial to prove the existence or upper bound for the density function of $Y(t)$ for $t>0$.
		However, it can be shown that under some conditions for the coefficients of $Y$, then by using local time argument or Krylov type estimate, it can be shown that the map $x \mapsto \int_{0}^{T} \p(Y(s) \leq x) \rd s$ is H\"older continuous with exponent in $(0,1]$ (see, Proposition \ref{main_1}, \ref{Lem_super_3}, \ref{main_6}, \ref{main_8} below).
		Therefore, we can apply Theorem \ref{main_0} by considering the stochastic process $Y$ as measurable function on the product space $([0,T] \times \Omega, \mathscr{B}([0,T]) \otimes \mathscr{F}, \mathrm{Leb} \otimes \p))$.
		We will apply this fact in section \ref{sec_3} to the error analysis of approximations of SDEs.

	\end{itemize}
\end{Rem}

Since $\mu$ is a finite measure, without loss of generality, we can assume that $(S,\Sigma,\mu)$ is a probability space.
In this section, we fix a probability space $(\Omega,\mathscr{F},\p)$.

For proving Theorem \ref{main_0}, we first recall Skorokhod's representation of a random variable (for more details, see, e.g. \cite{Wi91}, section 3).
Let $X$ be a $\real$-valued random variable on $(\Omega,\mathscr{F}, \p)$.
We define a random variable $X^{*}$ on the probability space $([0,1], \mathscr{B}([0,1]), \mathrm{Leb})$ by
\begin{align*}
	X^{*}(s)
	:=
	\inf\{x \in \real ~;~F_{X}(x) \geq s\}.
\end{align*}
Then it holds that $X^{*}$ has the same distribution as $X$, that is, $\p(X \in F) = \mathrm{Leb}(X^{*} \in F)$ for any $F \in \mathscr{B}(\real)$ and satisfies
\begin{align}\label{dist_1}
	s \leq F_{X}(x)
	\quad \Leftrightarrow \quad
	X^{*}(s) \leq x.
\end{align}
Moreover, it holds that for any $s \in [0,1]$,  $F_{X}(X^{*}(s)-) \leq s \leq F_{X}(X^{*}(s))$.
Therefore, if $F_{X}$ is continuous, then we have
$
	s=F_{X}(X^{*}(s)),
$
which plays a crucial role in our argument.
We define $d_{X}:(0,1] \times \real \to [0, \infty)$ by
\begin{align*}
	d_{X}(\alpha,K)
	:=
	\inf_{s \in [0,1],~s\neq F_{X}(K)}
		\left\{
			\frac
				{|X^{*}(s)-K|^{\alpha}}
				{|s-F_{X}(K)|}
		\right\},~
	(\alpha,K) \in (0,1] \times \real,
\end{align*}
and $D_X(\alpha,K):=1/d_X(\alpha,K)$.
Then $D_{X}$ can be estimated above by the H\"older constant of the distribution function $F_{X}$.

\begin{Lem}\label{key_1}
	Let $X$ be a $\real$-valued random variable such that its distribution function is $\alpha$-H\"older continuous with $\alpha \in (0,1]$.
	Then for any $K \in \real$ it holds that $D_X(\alpha, K) \leq \|F_{X}\|_{\alpha}$.
\end{Lem}
\begin{proof}
	Since $F_{X}$ is $\alpha$-H\"older continuous, it holds that for any $s \in [0,1]$, $s = F_{X}(X^{*}(s))$.
	Hence we have
	$
		|s - F_{X}(K)|
		=
		\left|
			F_{X}(X^{*}(s))
			-
			F_{X}(K)
		\right|
		\leq
		\|F_{X}\|_{\alpha}
		\left|
			X^{*}(s)
			-
			K
		\right|^{\alpha},
	$
	which implies the statement.
\end{proof}

By using this lemma, we have the following auxiliary estimate.

\begin{Thm}\label{key_2}
	Let $X$ be a $\real$-valued random variable such that its distribution function is $\alpha$-H\"older continuous with $\alpha \in (0,1]$.
	Then for any $\real$-valued variable $\widehat{X}$, $K \in \real$, $p \in (0, \infty)$ and $q \in [1,\infty)$, it holds that
	\begin{align*}
		\e\left[
			\left|
				\1_{(-\infty,K]}(X)
				-
				\1_{(-\infty,K]}(\widehat{X})
			\right|^{q}
		\right]
		&\leq
		3
		\|F_{X}\|_{\alpha}^{\frac{p}{p+\alpha}}
		\e[|X - \widehat{X}|^p]^{\frac{\alpha}{p+\alpha}}.
	\end{align*}
\end{Thm}

\begin{Rem}\label{key_3}
	\begin{itemize}
		\item[(i)]
		Note that the estimate in Theorem \ref{key_2} shows that corresponding estimates for the functions $\1_{(-\infty,K)}$, $\1_{[K,+\infty)}$ and $\1_{(K,+\infty)}$ can be obtained by considering complements of the intervals in the indicator functions and the random variables $-X$, $-\widehat{X}$ and constant $-K$.
		
		\item[(ii)]
		Let $X=(X_{1},\ldots,X_{d})^{\top}$ be a $d$-dimensional random variable such that for each $i=1,\ldots,d$, the distribution function $F_{X_{i}}$ is $\alpha$-H\"older continuous with $\alpha \in (0,1]$.
		Let $g:\real^{d} \to \real$ be a measurable function of the form
		\begin{align*}
			g(x)
			:=
			\int_{\real^d}
				\prod_{i=1}^{d}
				h_{i}(z,x)
				\1_{(L_{i}(z),K_{i}(z)]}(x_{i})
			\nu(\rd z),
		\end{align*}
		where 
		$h_{i}:\real^{d} \times \real^{d} \to \real$ are bounded measurable functions such that $$
		\|h\|_{\beta}:=
		\max_{i=1,\ldots,d}
		\sup_{z \in \real^{d},x \neq y}
		\frac{|h_{i}(z,x)-h_{i}(z,y)|}{|x-y|^{\beta}}<\infty$$, for some $\beta \in (0,1]$, and $L_{i},K_{i}:\real^{d} \to [-\infty,\infty)$ are measurable functions with $L_{i}(z) \leq K_{i}(z)$,
		and $\nu$ is a singed measure $\nu$ of bounded variation on $(\real^{d}, \mathscr{B}(\real^{d}))$, that is, the total variation of $\nu$, denoted by $|\nu|(\real^{d})$ is finite (see, e.g. \cite{Ha13}, pp. 120--124).
		For example $\nu(\rd z)=\delta_{a}(\rd z)$ for some $a \in \real^{d}$.
		Then it follows from Theorem \ref{key_2} that for any $p \in (0,\infty)$ and $q \in [1,\infty)$, for multi-dimensional setting, the following trivial extension holds
		\begin{align*}
			&\e\left[
				\left|
					g(X)
					-
					g(\widehat{X})
				\right|^{q}
			\right]\\
			&\leq
			C_{1}
			\sum_{i=1}^{d}
			\int_{\real^{d}}
				\e\left[
					\left|
						h_{i}(z,X)
						\1_{(L_{i}(z),K_{i}(z)]}(X_{i})
						-
						h_{i}(z,\widehat{X})
						\1_{(L_{i}(z),K_{i}(z)]}(\widehat{X}_{i})
					\right|^{q}
				\right]
			\nu(\rd z)
			\\
			&\leq
			C_{2}
			\e[|X-\widehat{X}|^{q\beta}]
			+
			C_{2}
			\e[|X - \widehat{X}|^p]^{\frac{\alpha}{p+\alpha}},
		\end{align*}
		for some positive constants $C_{1}$ and $C_{2}$.
		Thus for multi-dimensional case we can consider some specific functions, and then it can be applied to multilevel Monte Carlo method.
	\end{itemize}

\end{Rem}

\begin{proof}[Proof of Theomre \ref{key_2}]
	The idea of the proof is based on \cite{Av09}.
	We first note that for any $q \in [1,\infty)$, 
	$
		|\1_{(-\infty,K]}(x)
		-
		\1_{(-\infty,K]}(y)|^{q}
		=
		|\1_{(-\infty,K]}(x)
		-
		\1_{(-\infty,K]}(y)|
	$, thus it is sufficient to prove the statement for $q=1$.
	We define $\varepsilon_{0}:=\e[|\1_{(-\infty,K]}(X)-\1_{(-\infty,K]}(\widehat{X})|]\in [0,1]$.
	If $\varepsilon_{0}=0$, then the statement is obvious, thus now we assume $\varepsilon_{0} \in (0,1]$.
	Then $\varepsilon_{0}$ can be decomposed to two terms $\varepsilon_{0} = \p(\Omega_{1})+\p(\Omega_{2})$, where $\Omega_{1}:=\{X \leq K, \widehat {X} > K\}$ and $\Omega_{2}:=\{X > K, \widehat{X} \leq K\}$.
	Note that
	$F_{X}(K)=\p(X \leq K)\geq \p(\Omega_{1})$, $F_{X}(K) + \p(\Omega_{2}) \leq \p(X \leq K) + \p(X > K )= 1$, $|X - \widehat{X}| \geq K-X$ on $\Omega_{1}$ and $|X - \widehat{X}| \geq X-K$ on $\Omega_{2}$.
	Therefore, we have
	\begin{align*}
	\e\left[\left|X - \widehat{X}\right|^p\right]
	&\geq
	\e\left[\left|X - \widehat{X}\right|^p\1_{\Omega_{1}}\right]
	+
	\e\left[\left|X - \widehat{X}\right|^p\1_{\Omega_{2}}\right]\\
	&\geq
	\e[|X - K|^p \1 _{\Omega_{1}}]
	+\e[|X - K|^p \1 _{\Omega_{2}}].
	\end{align*}
	Since $F_{X}$ is continuous and non-decreasing, there exists $c_{1} \in [-\infty,K]$ such that
	\begin{align}\label{key_2_0}
		\p(\Omega_{1})
		=
		\p(c_{1} < X \leq K)
		=
		\mathrm{Leb}(c_{1} < X^{*} \leq K).
	\end{align}
	Furthermore, we have
	\begin{align*}
		\e[|X - K|^p \1 _{\Omega_{1}}]
		\geq
		\e[|X - K|^p \1_{\{c_{1} < X \leq K\}}].
	\end{align*}
	Indeed, if $c_{1} \in (-\infty,K]$, then for any $A \subset \{X \leq K\}$ with $\p(A) = \p(\Omega_{1})$, we have
	\begin{align*}
		\e[|X - K|^p \1_{A\cap \{c_{1} < X \leq K\}^{c} }]
		&\geq
		|c_1 - K|^p 
		\p(A \cap \{c_{1} < X \leq K\}^{c})\\
		&=
		|c_1 - K|^p 
		\left\{
			\p(A)
			-
			\p(A\cap \{c_{1} < X \leq K\})
		\right \}\\
		&=
		|c_1 - K|^p 
		\left\{
			\p(c_{1} < X \leq K)
			-
			\p(A \cap \{c_{1} < X \leq K\})
		\right \}\\
		&=
		|c_1 - K|^p
		\p(A^{c} \cap \{c_{1} < X \leq K\})\\
		&\geq
		\e[|X-K|^p \1_{A^{c} \cap \{c_{1} < X \leq K\}}],
	\end{align*}
	and if $c_{1}=-\infty$, then since $\p(\Omega_{1})=\p(X \leq K)$, we have $\p(X\leq K, \widehat{X} \leq K)=0$, and thus we obtain $\e[|X - K|^p \1 _{\Omega_{1}}]=\e[|X - K|^p \1_{\{c_{1} < X \leq K\}}]$, thus we obtain \eqref{key_2_0}.
	Therefore, by using \eqref{key_2_0} and \eqref{dist_1}, we have
	\begin{align*}
	\e\left[\left|X - K\right|^p \1_{\Omega_{1}}\right]
	&\geq
		\e\left[
			\left|
				X
				-
				K
			\right|^{p}
			\1_{\{c_{1} < X \leq K\}}
		\right]\notag\\
	&=
	\int_{0}^{1}
		\left|
			X^{*}(s)
			-
			K
		\right|^{p}
		\1_{\{c_{1} < X^{*}(s) \leq K\}}
	\rd s\notag\\
	&=
	\int_{F_{X}(K)-\p(\Omega_{1})}^{F_{X}(K)}
		\left|
			X^{*}(s)
			-
			K
		\right|^{p}
	\rd s \notag\\
	&\geq
	|d_X(\alpha,K)|^{\frac{p}{\alpha}}
	\int_{F_{X}(K)-\p(\Omega_{1})}^{F_{X}(K)}
		|s-F_{X}(K)|^{\frac{p}{\alpha}}
	\rd s \notag \notag\\
	&=
	|d_X(\alpha,K)|^{\frac{p}{\alpha}}
	\frac{ \alpha}{p+\alpha}
	\p(\Omega_{1})^{\frac{p}{\alpha}+1}.
	\end{align*}
	By similar calculations, we have
	\begin{align*}
	\e[|X - K|^p \1 _{\Omega_{2}}]
	\geq
	|d_X(\alpha,K)|^{\frac{p}{\alpha}}
	\frac{ \alpha}{p+\alpha}
	\p(\Omega_{2})^{\frac{p}{\alpha}+1}.
	\end{align*}
	Therefore since $(a+b)^{\frac{p}{\alpha}+1}\leq 2^{\frac{p}{\alpha}}(a^{\frac{p}{\alpha}+1}+b^{\frac{p}{\alpha}+1})$, $a,b>0$, we obtain the following estimate
	\begin{align*}
		\e[|X - \widehat{X}|^p]
		&\geq
			|d_X(\alpha,K)|^{\frac{p}{\alpha}}
			\frac{ \alpha}{p+\alpha}
			\left\{
				\p(\Omega_{1})^{\frac{p}{\alpha}+1}
				+
				\p(\Omega_{2})^{\frac{p}{\alpha}+1}
			\right\}\\
		&\geq
		|d_X(\alpha,K)|^{\frac{p}{\alpha}}
		\frac{ \alpha}{p+\alpha}
		\frac{1}{2^{\frac{p}{\alpha}}}
		\varepsilon_{0}^{\frac{p}{\alpha}+1}.
	\end{align*}
	Therefore, by using Lemma \ref{key_1}, we have
	\begin{align*}
		\e\left[
			\left|
				\1_{(-\infty,K]}(X)
				-
				\1_{(-\infty,K]}(\widehat{X})
			\right|
		\right]
		&\leq
		2^{\frac{p}{p+\alpha}}
		\left(
			\frac{p+\alpha}{\alpha}
		\right)^{\frac{\alpha}{p+\alpha}}
		D_{X}(\alpha,K)^{\frac{p}{p+\alpha}}
		\e[|X - \widehat{X}|^p]^{\frac{\alpha}{p+\alpha}}\\
		&\leq
		2
		\left(
			\frac{p+\alpha}{2\alpha}
		\right)^{\frac{\alpha}{p+\alpha}}
		\|F_{X}\|_{\alpha}^{\frac{p}{p+\alpha}}
		\e[|X - \widehat{X}|^p]^{\frac{\alpha}{p+\alpha}}\\
		&\leq
		3
		\|F_{X}\|_{\alpha}^{\frac{p}{p+\alpha}}
		\e[|X - \widehat{X}|^p]^{\frac{\alpha}{p+\alpha}},
	\end{align*}
	where we use the fact that $\max_{x>0} x^{1/x}=e^{1/e}$ and $2e^{\frac{1}{2e}}<3$.
	This concludes the statement.
\end{proof}


\begin{proof}[Proof of Theorem \ref{main_0}]
	We first assume $g \in NBV$, that is, $g \in BV$ such that $g$ is left continuous and $\lim_{x \to -\infty} g(x)=0$.
	Then by Theorem 8.14 in \cite{Ru74}, there exists a unique signed measure $\nu$ on $\real$ such that $g(x)=\nu((-\infty,x))$ and $|\nu|((-\infty,x))=T_{g}(x)$, where $|\nu|$ is the total variation measure of $\nu$.
	Since $|\nu|(\real)=V(g)<\infty$, we have
	\begin{align*}
		g(x)
		=\int_{\real}
			\1_{(-\infty,x)}(z)
		\nu(\rd z)
		=\int_{\real}
			\1_{(z,+\infty)}(x)
		\nu(\rd z).
	\end{align*}
	Therefore, it follows from Remark \ref{key_3} (i) that
	\begin{align*}
		\e\left[
			\left|
				g(X)
				-
				g(\widehat{X})
			\right|^{q}
		\right]
		&\leq
		\e\left[
			\left|
				\int_{\real}
					\left\{
						\1_{(z,+\infty)}(X)
						-
						\1_{(z,+\infty)}(\widehat{X})
					\right\}
				\nu(\rd z)
			\right|^{q}
		\right]\\
		&\leq
		V(g)^{q-1}
		\int_{\real}
			\e\left[
				\left|
					\1_{(z,+\infty)}(X)
					-
					\1_{(z,+\infty)}(\widehat{X})
				\right|^{q}
				\right]
		|\nu|(\rd z)\\
		&\leq
		3
		V(g)^{q}
		\|F_{X}\|_{\alpha}^{\frac{p}{p+\alpha}}
		\e[|X - \widehat{X}|^p]^{\frac{\alpha}{p+\alpha}},
	\end{align*}
	which concludes the statement for $g \in NBV$.

	We now assume $g \in BV$.
	Then by Theorem 8.13 in \cite{Ru74}, $g$ can be decomposed by
	$
		g(x)
		=
		\widetilde{g}(x)
		+
		c
		+
		\Delta(x),
	$
	where $\widetilde{g} \in NBV$, $c \in \real$ and
	$
		\Delta(x)
		:=
		\sum_{j \in \n}
		\lambda_{j} \1_{\{a_{j}\}}(x),~
		\lambda_{j}:=g(a_{j})-\widetilde{g}(a_{j})-c,
	$
	and $\{a_{j}\}_{j \in \n}$ is a discontinuous points of $g$ (which are countable).
	We define $\widetilde{\nu}:=\sum_{j \in \n} \lambda_{j} \delta_{a_{j}}$.
	Then $V(\widetilde{g})<V(g)$, and since $g(a_{j}-)$ exists and $\widetilde{g}(a_{j})+c=g(a_{j}-)$,  we have $|\widetilde{\nu}|(\real) =\sum_{j \in \n} |g(a_{j})-g(a_{j}-)| \leq V(g)$ and
	$
		\Delta(x)
		=
		\int_{\real}
			\left\{
				\1_{(-\infty,z]}(x)
				-
				\1_{(-\infty,z)}(x)
			\right\}
		\widetilde{\nu}(\rd z).
	$
	Therefore, by similar way as NBV case, we have
	\begin{align*}
		\e\left[
			\left|
				g(X)
				-
				g(\widehat{X})
			\right|^{q}
		\right]
		&\leq
		3^{q+1}
		V(g)^{q}
		\|F_{X}\|_{\alpha}^{\frac{p}{p+\alpha}}
		\e[|X - \widehat{X}|^p]^{\frac{\alpha}{p+\alpha}}.
	\end{align*}
	This concludes the statement.
\end{proof}

\section{Examples}\label{sec_example}
In this section, we provide two examples of stochastic processes with H\"older continuous distribution function.
In this section, $B=(B(t))_{t \geq 0}$ is a standard  Brownian motion defined on a probability space $(\Omega, \mathscr{F},\p)$ with a filtration $(\mathscr{F}_{t})_{t \geq 0}$ satisfying the usual conditions.

\subsection{SDEs with path-dependent and linear growth drift}\label{sec_3_0}
Let $b:[0,\infty) \times C([0,\infty);\real) \to \real$ and diffusion coefficient $\sigma:[0,\infty) \times \real \to \real$ be measurable functions, and let $X^{x}=(X^{x}(t))_{t \geq 0}$ be a solution to the following path--dependent one-dimensional SDE
\begin{align}\label{SDE_0}
\rd X^{x}(t)
=
b(t,X^{x}) \rd t
+
\sigma(t,X^{x}(t))\rd B(t),
~X^{x}(0)=x \in \real,
~t \geq 0.
\end{align}
In this subsection, we consider H\"older continuity of the distribution function of a solution of SDE \eqref{SDE_0} with linear growth drift coefficient.
We need the following assumptions on the coefficients $b$ and $\sigma$.

\begin{Ass}\label{Ass_1}
	We suppose that the coefficients $b:[0,\infty) \times C([0,\infty);\real) \to \real$ and $\sigma:[0,\infty)\times \real\to \real$ satisfy the following conditions:
	\begin{itemize}
		\item[(i)]
		The drift coefficient $b$ is $\mathscr{B}([0,\infty)) \otimes \mathscr{B}(C([0,\infty);\real))/\mathscr{B}(\real)$-measurable and for each fixed $t>0$, the map $C([0,\infty);\real) \ni w\mapsto b(t,w) \in \real$ is $\mathscr{B}_{t}(C([0,\infty);\real))/\mathscr{B}(\real)$-measurable (see, Chapter IV, Definition 1.1 in \cite{IkWa}), and is of linear growth, that is, for each $T>0$, there exists $K(b,T)>0$ such that for any $(t,w) \in [0,T] \times C([0,T];\real^d)$,
		$
		|b(t,w)|
		\leq
		K(b,T)(1+\sup_{0 \leq s \leq t}|w(s)|).
		$
		
		\item[(ii)]
		$a:=\sigma^{2}$ is $\beta$-H\"older continuous in space and $\beta/2$-H\"older continuous in time with $\beta \in (0,1]$, that is,
		\begin{align*}
			\sup_{t \in [0,\infty), x \neq y}
			\frac{|a(t,x)-a(t,y)|}{|x-y|^{\beta}}
			+
			\sup_{x\in \real, t \neq s}
			\frac{|a(t,x)-a(s,x)|}{|t-s|^{\beta/2}}
			<\infty.
		\end{align*}
		
		\item[(iii)]
		The diffusion coefficient $\sigma$ is bounded and uniformly elliptic, that is, there exist $\underline{a}, \overline{a}>0$ such that for any $(t,x) \in [0,\infty) \times \real$, $\underline{a} \leq a(t,x) \leq \overline{a}$.
	\end{itemize}
\end{Ass}

\begin{Thm}\label{main_Hol_0}
	Suppose that Assumption \ref{Ass_1} holds.
	Let $(t,x) \in (0,\infty) \times \real$ be fixed.
	Then for each $\alpha \in (0,1)$, the distribution function of $X^{x}(t)$ is $\alpha$-H\"older continuous.
	More preciously, there exists $C_{\alpha}>0$ such that for any $z,z' \in \real$,
	\begin{align*}
		\left|
			F_{X^x(t)}(z)
			-
			F_{X^x(t)}(z')
		\right|
		\leq
		C_{\alpha}
		\left\{
			\frac{1}{t^{\frac{\alpha}{2}}}
			+
			\frac{1}{t^{\frac{1-\alpha}{2}}}
		\right\}
		|z-z'|^{\alpha}.
	\end{align*}
\end{Thm}
\begin{Rem}
	\begin{itemize}
		\item[(i)]
		Note that if the drift coefficient $b$ is path-dependent and of sub-linear growth, that is, for any $\delta,t>0$, there exists $K_t(\delta)>0$ such that $K_t(\delta)$ is increasing with respect to $t$ and for all $t>0$ and $w \in C([0,t];\real^d)$, $|b(t,w)|\leq \delta \sup_{0 \leq s \leq t}|w(s)|+K_t(\delta)$, then $X^{x}(t)$, $t>0$ admits a density function which satisfies the Gaussian two-sided bound (see, Theorem 3.4 in \cite{TaTa18} and also \cite{Ku17} for bounded drift), thus its distribution function is Lipschitz continuous.
		On the other hand, in the case of linear growth drift, it is difficult to estimate the upper bound of the density.
		
		\item[(ii)]
		As application of Theorem \ref{main_0}, we can use multilevel Monte Carlo method for irregular functional of SDEs with linear growth drift coefficient (e.g.  Lipschitz continuous function).
	\end{itemize}
\end{Rem}

\begin{proof}[Proof of Theorem \ref{main_0}]
	Let $T>0$ be fixed.
	We first recall that if the diffusion coefficient $\sigma$ satisfies the Assumption \ref{Ass_1} (ii) and (iii), then from Theorem 6.5.4 in \cite{Fr75}, there exists a fundamental solution $q(s,x;t,\cdot)$ of the following Kolmogorov backward equation:
	\begin{align}\label{pde_fund_0}
	(\partial_s + L_s)q(s,x;t,y)=0,
	~\lim_{s \uparrow t}\int_{\real^d} f(y) q(s,x;t,y) \rd y
	= f(x),~f \in C_b^{\infty}(\real^d;\real),
	\end{align}
	where $L_s$ is a differential operator defined by
	\begin{align*}
	L_sf(x)
	:=\frac{a(s,x)f''(x)}{2},
	\end{align*}
	(see page 149 in \cite{Fr75}).
	Then there exist $\widehat{C}>0$ and $\widehat{c}>0$ such that for any $(t,x,y) \in (s,T] \times \real \times \real$,
	\begin{align}\label{bound_qt}
		q(s,x;t,y)
		\leq
		\widehat{C}
		g_{\widehat{c}(t-s)}(x,y)
		~\text{and}~
		|\partial_{x} q(s,x;t,y)|
		\leq \frac{\widehat{C}}{(t-s)^{1/2}} g_{\widehat{c}(t-s)}(x,y),
	\end{align}
	where $g_{c}(x,y):=\frac{\exp(-\frac{|y-x|^{2}}{2c})}{\sqrt{2\pi c}}$, for $c>0$ (see, e.g. \cite{Fr64}, Theorem 9.4.2).
	Moreover, from Theorem 3.1 in \cite{TaTa18}, for any $(t,x) \in (0,\infty) \times \real^{d} $, $X^{x}(t)$ admits a density function with respect to Lebesgue measure, denoted by $p_{t}(x,\cdot)$, which satisfies the following representation
	\begin{align}\label{density_1}
		p_{t}(x,y)
		&=
		q(0,x;t,y)
		+
		\int_{0}^t{}
			\e\left[
				\partial_{x}
				q(s,X^{x}(s);t,y)
				b(s,X^{x})
			\right]
		\rd s,
		~\text{a.e.}~y \in \real.
	\end{align}
	
	By using this representation, we prove H\"older continuity of $X^{x}(t)$.
	Let $(t,x,y) \in (0,T] \times \real \times \real$ and $z,z' \in \real$.
	We first note that for $c>0$, by using H\"older's inequality with $1/p+1/q=1$, $p,q>1$, we have
	\begin{align}\label{main_00}
	\int_{z}^{z'}
		g_{ct}(x,y)
	\rd y
	&\leq
	\left(
		\int_{\real}
			\1_{(z,z']}(y)
		\rd y
	\right)^{1/p}
	\left(
		\int_{\real}
			g_{ct}(x,y)^{q}
		\rd y
	\right)^{1/q}
	\leq
	\frac{C_{q}|z-z'|^{1/p}}{t^{\frac{1}{2p}}},
	\end{align}
	for some $C_{q}=C_{q}(c)>0$.
	Therefore, by using the representation \eqref{density_1}, the estimates \eqref{bound_qt}, \eqref{main_00} and $\e\left[\sup_{0 \leq s \leq t} |X^{x}(s)|\right]<\infty$, we have
	\begin{align*}
		&\p(z <X^{x}(t) \leq z')
		\leq
		\int_{z}^{z'}
			q(0,x;t,y)
		\rd y
		+
		\int_{z}^{z'}
			\rd y
			\int_{0}^{t}
				\rd s
				\e\left[
					|\partial_{x} q(s,X^{x}(s);t,y)|
					|b(s,X^{x})|
				\right]
		\\&\leq
		\widehat{C}
		\int_{z}^{z'}
			g_{\widehat{c}t}(x,y)
		\rd y
		+
		\int_{0}^{t}
			\frac{\widehat{C}}{(t-s)^{1/2}}
			\e\left[
				\int_{z}^{z'}
					g_{\widehat{c}(t-s)}(X^{x}(s),y)
				\rd y
				|b(s,X^{x})|
			\right]
		\rd s\\
		&\leq
		\frac{C_{q}|z-z'|^{1/p}}{t^{\frac{1}{2p}}}
		+
		\int_{0}^{t}
			\frac
				{\widehat{C}C_{q}K(b,T)|z-z'|^{1/p}}
				{(t-s)^{\frac{1}{2}+\frac{1}{2p}}}
			\e\left[
				1+\sup_{0 \leq s \leq t}|X^{x}(s)|
			\right]
		\rd s
		\\&\leq
		C\left\{
			\frac{1}{t^{\frac{1}{2p}}}
			+
			\frac{1}{t^{\frac{1}{2}\left(1-\frac{1}{p}\right)}}
		\right\}
		|z-z'|^{1/p}
	\end{align*}
	for some $C>0$.
	By choosing $\alpha=1/p \in (0,1)$, we concludes the proof.
\end{proof}

\subsection{Maximum of SDEs}\label{sec_3_1}

Let us consider the following one-dimensional SDEs of the form
\begin{align}\label{SDE_max_1}
	\rd X(t)
	=
	b(t,X(t)) \rd t
	+
	\sigma(X(t))\rd B(t),
	~X(0)=x_{0} \in \real,
	~t \in [0,T].
\end{align}
In this subsection, we consider H\"older continuity of the distribution function of $\max_{0 \leq t \leq T}X(t)$.
As application we consider an approximation scheme for expectation $\e[g(\max_{0 \leq t \leq T} X(t))]$ for $g \in BV$, which includes the payoff of binary options with respect to the running maximum of $X$.

\begin{Prop}\label{Lem_max_0}
	Suppose that the drift coefficient $b$ is measurable and of sub-linear growth, and diffusion coefficient $\sigma \in C_{b}^{1}(\real;\real)$ and uniformly elliptic.
	Then, the distribution function of $\max_{0 \leq t \leq T} X(t)$ is $\alpha$-H\"older continuous for any $\alpha \in (0,1)$.
	
\end{Prop}

\begin{Rem}
	\begin{itemize}
		\item[(i)]
		Let $f:\real \to \real$ be a measurable function.
		Then $f$ satisfies that bounded on any compact subset of $\real$ and $|f(x)|=o(|x|)$ as $|x| \to \infty$ if and only if for any $\delta>0$, there exists a constant $K(\delta)>0$ such that $|f(x)|\leq \delta |x|+K(\delta)$.
		
		\item[(ii)]
		If $b, \sigma \in C^{2}_{b}(\real;\real)$ and $\sigma$ is uniformly elliptic, then by using Malliavin calculus approach, the law of $\max_{0 \leq t \leq T} X(t)$ and discrete time maximum $\max\{X(t_{1}),\ldots, X(t_{n})\}$ with $0 \leq t_{1} < \cdots < t_{n-1} < t_{n}=T$ are absolutely continuous with respect to Lebesgue measure and the densities satisfy some Gaussian type upper bound (see, Theorem 3 in \cite{Nak16}), thus in this setting, we can apply original Avikainen's estimate \eqref{Av_0}.
		
		\item[(iii)]
		It is well-known that the drift coefficient $b$ and $\sigma$ are Lipschitz continuous in space and $1/2$-H\"older continuous in time, then for any $p \geq 1$, $\mathrm{Err}_{p}(n)=C n^{-1/2}$ for some $C>0$ (see, \cite{KP}) and, recently, the strong rate of convergence under non-Lipschitz drift coefficient are studied (see, \cite{BaHuYu19,EiKoKrLa19,GyRa11,LeSz,LeSz17,LeSz17b,MuYa19,MeTa,NgTa16,NgTa1701}) and subsection \ref{sec_3_6}.
		Moreover, the estimate \eqref{max_0} can be applied to multilevel Monte Carlo methods (see, \cite{Gi08}).
	\end{itemize}
\end{Rem}


\begin{proof}
	For continuous stochastic process $A=(A(t))_{0 \leq t \leq T}$, we set $M_{A}(T):=\max_{0 \leq t \leq T} A(t)$.
	
	(Step 1).
	We frist apply Lamperti transform.
	Let $L(x):=\int_{0}^{x}\sigma(y)^{-1}\rd y$, then $L$ has an inverse function $L^{-1}$ and satisfies
	$L'(x) = \sigma(x)^{-1}$, $L''(x)=-\sigma'(x)\sigma(x)^{-2}$. 
	Moreover since $\sigma$ is bounded and uniformly elliptic, $L$ and $L^{-1}$ are Lipschitz continuous.
	Then by using It\^o's formula, $Y(t):=L(X(t))$ satisfies the following SDE of the form
	\begin{align*} 
		\rd Y(t)
		=
		b_{Y}(t,Y(t))\rd t
		+
		B(t),~
		Y(0)=L(x_{0})=:y_{0},
	\end{align*}
	where
	$
		b_{Y}(t,y)
		:=
		\frac
		{b(t,L^{-1}(y))}
		{\sigma\circ L^{-1} (y)}
		-
		\frac
		{\sigma'\circ L^{-1}(y)}
		{2}.
	$
	
	(Step 2).
	We apply Maruyama--Girsanov transform in order to remove the drift coefficient from $Y$.
	Let $p \in \real$, $Z(p,\cdot)=(Z(p,t))_{0 \leq t \leq T}$ is defined by
	\begin{align*}
		Z(p,t)
		&:=
		\exp\left(
			p
			\int_{0}^{t}
				b_{Y}(s,y_{0}+B(s))
			\rd B(s)
			-
			\frac{p^{2}}{2}
			\int_{0}^{t}
				|b_{Y}(s,y_{0}+B(s))|^2
			\rd s
		\right).
	\end{align*}
	Since $b_{Y}$ is also sub--linear growth, thus, $Z(-1,\cdot)$ is martingale and has any both positive and negative moments (see, Lemma 3.6 in \cite{TaTa18}).
	Thus by using Maruyama--Girsanov transform, it holds that for any bounded measurable functional $f \in C([0,T];\real)$,
	\begin{align}\label{Girsanov_0}
		\e[f(Y)]
		=\e[f(y_{0}+B) Z(-1,T)],
	\end{align}
	
	
	(Step 3).
	Since $L$ is strictly increasing, $M_{Y}(T)=L(M_{X}(T))$, and thus by using \eqref{Girsanov_0}, for any $a,b \in \real$ with $a<b$, we have
	\begin{align*}
		\p(a < M_{X}(T) \leq b)
		&=
		\p\left(
			L(a) < M_{Y}(T) \leq L(b)
		\right)\\
		&=
		\e\left[
			\1_{(L(a), L(b)]}
			(
				y_{0}
				+
				M_{B}(T)
			)
			Z(-1, T)
		\right].
	\end{align*}
	Therefore, by using H\"older's inequality with $1/p+1/p'=1$ for $p,p'>1$ and explicit representation of the density of maximum of Brownian motion (see, e.g. Problem 2.8.2 in \cite{KS}), we have
	\begin{align*}
	\p(a < M_{X}(T) \leq b)
	&
	\leq
	\e\left[
		Z(-p', T)^{1/p'}
	\right]
	\p(L(a) < y_{0}+M_{B}(T) \leq L(b))^{1/p}\\
	&=
	\e\left[
		Z(-p', T)^{1/p'}
	\right]
	\left(
		\int_{L(a)-y_{0}}^{L(b)-y_{0}}
			\sqrt{\frac{2}{\pi T}}
			\exp\left(
				-\frac{m^{2}}{2T}
			\right)
		\rd m
	\right)^{1/p}\\
	&\leq
	\e\left[
	Z(-p', T)^{1/p'}
	\right]
	\left(
		\frac{2}{\pi T}
	\right)^{1/(2p)}
	(L(b)-L(a))^{1/p}.
	\end{align*}
	Since $L$ is a Lipschitz continuous and $p>1$ is arbitrarily, we conclude that the distribution function of $M_{X}(T)$ is $\alpha$-H\"older continuous for any $\alpha \in (0,1)$.
	This concludes the proof of the statement.
\end{proof}

\begin{Thm}\label{main_max_0}
	Suppose that the drift coefficient $b$ is measurable and of sub-linear growth, that is, for any $\delta>0$, there exists $K_{t}(\delta)>0$ such that $K_{t}(\delta)$ is increasing with respect to $t$ and for all $t>0$ and $x \in \real$,
	$
		|b(t,x)|
		\leq
		\delta |x|+K_{t}(\delta)
	$, and diffusion coefficient $\sigma \in C_{b}^{1}(\real;\real)$ and uniformly elliptic.
	Let $X^{(n)}$ be the Euler--Maruyama scheme for SDE \eqref{SDE_max_1} defined by
	\begin{align*}
		\rd X^{(n)}(t)
		=
		b(\eta_{n}(t), X^{(n)}(\eta _n(t))) \rd t
		+
		\sigma(X^{(n)}(\eta _n(t))) \rd B(t),~
		X^{(n)}(0)=x_{0},~
		t \in [0,T],
	\end{align*}
	where $\eta _{n}(s) = kT/n=:t_{k}^{(n)}$ if $ s \in \left[kT/n, (k+1)T/n \right)$.
	Let $p \geq 1$ and suppose $\e[\max_{0 \leq t \leq T}|X(t)-X^{(n)}(t)|^{p}]^{1/p} \leq \mathrm{Err}_{p}(n)$.
	Then for any $g \in BV$, $q \in [1,\infty)$ and $\alpha \in (0,1)$, there exists a positive constant $C=C(g,b,\sigma,p,q,T,\alpha)>0$ such that for any $n \geq 2$,
	\begin{align}\label{max_0}
	&\e\left[
		\left|
			g\left(\max_{0 \leq t \leq T} X(t)\right)
			-
			g\left(\max_{0 \leq t \leq T} X^{(n)}(\eta_{n}(t))\right)
		\right|^{q}
	\right]
	\leq
	C
	\left\{
		\mathrm{Err}_{p}(n)^{\frac{p\alpha}{p+\alpha}}
		+
		\left(
			\frac{\log n}{n}
		\right)^{\frac{p\alpha}{2(p+\alpha)}}
	\right\}.
	\end{align}
\end{Thm}
\begin{proof}
	From Theorem \ref{main_0} and Proposition \ref{Lem_max_0}, it is suffices to estimate the moment of the difference between $\max_{0 \leq t \leq T} X(t)$ and $\max_{0 \leq t \leq T} X^{(n)}(\eta_{n}(t))$.
	By using triangle inequality, we have
	\begin{align*}
		&\e
		\left[
			\left|
				\max_{0 \leq t \leq T} X(t)
				-
				\max_{0 \leq t \leq T} X^{(n)}(\eta_{n}(t))
			\right|^p
		\right]\\
		&\leq
		2^{p-1}
		\e
		\left[
			\max_{0 \leq t \leq T}
				\left|
					X(t)
					-
					X^{(n)}(t)
			\right|^{p}
		\right]
		+
		2^{p-1}
		\e
		\left[
			\max_{0 \leq t \leq T}
				\left|
				X^{(n)}(t)
				-
				X^{(n)}(\eta_{n}(s))
			\right|^{p}
		\right].
	\end{align*}
	By L\'evy's modulus continuity of Brownian motion (see, e.g., Theorem 2.9.25 in \cite{KS}), since $b$ is sub--linear growth there exists $C>0$ such that
	\begin{align*}
	\e
	\left[
		\max_{0 \leq t \leq T}
		\left|
			X^{(n)}(t)
			-
			X^{(n)}(\eta_{n}(s))
		\right|^{p}
	\right]
	&\leq
	\frac{2^{p-1}\max_{0 \leq t \leq T} \e\left[|b(\eta_{n}(t),X^{(n)}(\eta_{n}(t)))|^{p} \right]T^{p}}{n^{p}}\\
	&\quad+
	2^{p-1}
	\|\sigma\|_{\infty}^{p}
	\e\left[
		\max_{0 \leq t \leq T}
		\left|
			B(t)
			-
			B(\eta_{n}(s))
		\right|^{p}
	\right]\\
	&\leq
	C
	\left(
		\frac{\log n}{n}
	\right)^{p/2},
	\end{align*}
	which concludes the proof.
\end{proof}

\section{Application to numerical schemes for SDEs and SHEs}\label{sec_3}
In this section, we apply a generalized Avikainen's estimate \eqref{main_0_0} to several problems on numerical analysis of SDEs and SHEs with irregular coefficients.

\subsection{SDEs with bounded $2$-variation diffusion coefficients}\label{sec_3_2}

In this subsection, we consider one-dimensional SDEs with bounded $2$-variation diffusion coefficient.
It is well-known that Yamada and Watanabe \cite{YaWa} proved that if the diffusion coefficient $\sigma$ is $\alpha$-H\"older continuous with exponent $\alpha \in [1/2,1]$, then the pathwise uniqueness holds.
Besides, Girsanov \cite{Gi62} and Barlow \cite{Ba82} provided some examples of $\alpha$-H\"older continuous function $\sigma$ with $\alpha \in (0,1/2)$ such that the pathwise uniqueness fails for SDE \eqref{SDE_1}, and thus the H\"older exponent $\alpha = 1/2$ is sharp.
On the other hand, Le Gall \cite{LeGall} proved that if the diffusion coefficients is bounded, uniformly positive and bounded $2$-variation, then the pathwise uniqueness holds.
Note that we need the uniformly positive condition.
Indeed, if $b=0$ and $\sigma=\mathrm{sgn}$ then the equation is called Tanka's equation, and in this case there is no strong solution.

The goal of this subsection is that, under above condition, we provide the rate of convergence for the Euler--Maruyama scheme (see Theorem \ref{main_3} and Theorem \ref{main_10}).

\subsubsection*{Case 1 : time independent coefficients}
Let us consider the following one-dimensional SDEs of the form
\begin{align}\label{SDE_1}
	\rd X(t)
	=
	b(X(t)) \rd t
	+
	\sigma(X(t))\rd B(t),
	~X(0)=x_{0} \in \real,
	~t \in [0,T],
\end{align}
where $B=(B(t))_{0\leq t \leq T}$ is a standard  Brownian motion defined on a probability space $(\Omega, \mathscr{F},\p)$ with a filtration $(\mathscr{F}_t)_{0\leq t \leq T}$ satisfying the usual conditions.
Since the solution of (\ref{SDE_1}) is rarely analytically tractable, one often approximates $X=(X(t))_{0 \leq t \leq T}$ by using the Euler--Maruyama scheme given by 
\begin{align*}
	\rd X^{(n)}(t)
	=
	b(X^{(n)}(\eta _n(t))) \rd t
	+
	\sigma(X^{(n)}(\eta _n(s))) \rd B(t),~
	X^{(n)}(0)=x_{0},~
	t \in [0,T],
\end{align*}
where $\eta _{n}(s) = kT/n=:t_{k}^{(n)}$ if $ s \in \left[kT/n, (k+1)T/n \right)$. 

If the diffusion coefficient $\sigma$ is of bounded $2$-variation, then we have the following rate of strong convergence for the Euler--Maruyama scheme.

\begin{Thm} \label{main_3}
	Suppose that coefficients $b$ and $\sigma$ are measurable, bounded and $\sigma$ is uniformly positive.
	Moreover, assume that there exist $\gamma \in (0,1]$, $f_{b} \in BV$ and bounded and strictly increasing function and $f_{\sigma}$ such that for any $x,y \in \real$,
	\begin{align*}
		|b(x)-b(y)|
		\leq
		|f_{b}(x)-f_{b}(y)|^{\gamma}
		\quad\text{and}\quad
		|\sigma(x)-\sigma(y)|^{2}
		\leq
		|f_{\sigma}(x)-f_{\sigma}(y)|.
	\end{align*}
	Then there exists a constant $C$ such that for any $n \geq 3$,
	\begin{align*}
		\sup_{0\leq t \leq T}
		\e\left[
			\left|
				X(t)
				-
				X^{(n)}(t)
			\right|
		\right]
		\leq
		\frac{Ce^{C\sqrt{\log \log n}}}{\log n}
	\end{align*}
	and if $ b \in L^{1}(\real)$, then there exists a constant $C$ such that for any $n \geq 2$,
	\begin{align*}
	\sup_{0\leq t \leq T}
	\e\left[
		\left|
			X(t)
			-
			X^{(n)}(t)
		\right|
	\right]
	\leq
	\frac{C}{\log n}.
	\end{align*}
\end{Thm}


\begin{Rem}
	By the structural Theorem (see, Theorem 3.1 in \cite{ChGa98}), the condition for $\sigma$ in Theorem \ref{main_3} is of bounded $2$-variation.
	Le Gall \cite{LeGall} showed that the pathwise uniequness holds for SDE \eqref{SDE_1} with bounded $2$-variation diffusion coefficient which includes discontinuous functions (see Remark \ref{Rem_app_SDE} (i) for applications of this type of equations).
	Note that Gy\"ongy and R\'asonyi \cite{GyRa11} proved the same error estimate in the case that the drift coefficient $b$ is Lipschitz continuous and the diffusion coefficient $\sigma$ is $1/2$-H\"older continuous.
\end{Rem}


Before proving Theorem \ref{main_3}, we consider the following one-dimensional It\^o process $Y=(Y(t))_{t \geq 0}$ defined by
\begin{align*}
	\rd Y(t)
	=
	b(t,\omega) \rd t
	+
	\sigma(t,\omega)
	\rd B(t),~
	Y(0)=y_{0} \in \real,~
	t \in [0,T],
\end{align*}
where $b,\sigma:[0,T] \times \Omega \to \real$ are progressively measurable stochastic processes.
Then we derive a key estimation for proving Theorem \ref{main_3}.

\begin{Prop}\label{main_1}
	The drift coefficient $b$ and diffusion coefficient $\sigma$ of $Y$ are uniformly bounded, and $a:=\sigma^{2}$ is uniformly positive, that is, there exists $\underline{a}>0$ such that $a(t,\omega) \geq \underline{a}$ for all $t \geq 0$ almost surely.
	Then for any $g \in BV$, $p \in (0,\infty)$ and $q \in [1,\infty)$, there exists a positive constant $C=C(g,b,\sigma,p,q)>0$ such that for any one-dimensional progressively measurable process $\widehat{Y}=(\widehat{Y}(t))_{t \geq 0}$, we have
	\begin{align*}
		\int_{0}^{T}
			\e\left[
				\left|
					g(Y(s))
					-
					g(\widehat{Y}(s))
				\right|^{q}
			\right]
		\rd s
		\leq
		C
		\left(
			\int_{0}^{T}
				\e
				\left[
					\left|
						Y(s)
						-
						\widehat{Y}(s)
					\right|^p
				\right]
			\rd s
		\right)^{\frac{1}{p+1}}.
	\end{align*}
	

\end{Prop}

For proving Proposition \ref{main_1}, we estimate a uniform $L^2$-bounded of the local time of $Y$.

\begin{Lem}\label{local_time}
	Suppose the coefficients $b$ and $\sigma$ of $Y$ satisfy the same conditions on Proposition \ref{main_1}.
	Then it holds that
	\begin{align*}
		\sup_{x \in \real}
		\e\left[\left|L_{T}^{x}(Y)\right|^{2}\right]
		&\leq
			12\|b\|_{\infty}^{2}T^{2}
			+
			6 \|\sigma\|_{\infty}^{2} T.
	\end{align*}
\end{Lem}
\begin{proof}
	By using the symmetric It\^o--Tanaka formula, we have
	\begin{align*}
	L_{T}^{x}(Y)
	&=
	|Y(T)-x|-|y_{0}-x|
	-
	\int_{0}^{T}
		\mathrm{sgn}(Y(s))
	\rd Y(s)\\
	&\leq
	|Y(T)-y_{0}|
	+
	2\int_{0}^{T}
			\left|
				b(s,\omega)
			\right|
		\rd s
	+
	\left|
		\int_{0}^{T}
			\mathrm{sgn}(Y(s))
			\sigma(s, \omega)
		\rd B(s)
	\right|.
	\end{align*}
	Since $b$ and $\sigma$ are bounded, it follows from the inequality $(a+b+c)^2\leq 3(a^2+b^2+c^2)$, $a,b,c \geq 0$ and the $L^2$-isometry that,
	\begin{align*}
		\sup_{x \in \real}
		\e\left[
			\left|
				L_{T}^{x}(Y)
			\right|^2
		\right]
		&\leq
		12\|b\|_{\infty}^{2}T^{2}
		+
		6
		\int_{0}^{T}
			\e\left[
				\left|
					\sigma(s,\omega)
				\right|^{2}
			\right]
		\rd s
		\leq
		12\|b\|_{\infty}^{2}T^{2}
		+
		6 \|\sigma\|_{\infty}^{2} T.
	\end{align*}
	This concludes the statement.
\end{proof}

\begin{proof}[Proof of Proposition \ref{main_1}]
	From Theorem \ref{main_0}, it is suffices to estimate $\int_{0}^{T} \p(a < Y(s) \leq b) \rd s$.
	Since the coefficients of $Y$ are bounded and $a=\sigma^{2}$ is uniformly positive, by using the occupation time formula and Lemma \ref{local_time}, we have
	\begin{align*}
		\int_{0}^{T}
			\p(a < Y(s) \leq b)
		\rd s
		&=
		\e\left[
			\int_{0}^{T}
				\1_{(a,b]}(Y(s))
			\rd s
		\right]\\
		&\leq
		\underline{a}
		\e\left[
			\int_{0}^{T}
				\1_{(a,b]}(Y(s))
			\rd \langle Y \rangle (s)
		\right]\\
		&=
		\underline{a}
		\e\left[
			\int_{\real}
				\1_{(a,b]}(x)
				L_{T}^{x}(Y)
			\rd x
		\right]\\
		&\leq
		\underline{a}
		\sqrt{
			12\|b\|_{\infty}^{2}T^{2}
			+
			6 \|\sigma\|_{\infty}^{2} T
		}
		(b-a),
	\end{align*}
	which concludes the statement.
\end{proof}

\begin{Rem}
	Note that since $b,\sigma$ are bounded, and $\sigma$ is uniformly elliptic, we can prove directly Proposition \ref{main_1} by using Krylov estimate (see, page 54, Theorem 4 in \cite{Kr80}).
\end{Rem}

Proposition \ref{main_1} shows the following error estimate for It\^o processes.

\begin{Prop}\label{main_2}
	Suppose the coefficients $b$ and $\sigma$ satisfies the same conditions on Proposition \ref{main_1}.
	Then for any $g \in BV$, $p \in (0,\infty)$ and $q \in [1,\infty)$, there exists $C=C(g,b,\sigma,p,q)>0$ such that for any $n \in \n$,
	\begin{align*}
		\int_{0}^{T}
			\e\left[
				\left|
					g(Y(s))
					-
					g(Y(\eta_{s}(s)))
				\right|^{q}
			\right]
		\rd s
		\leq
		C
		\left(\frac{T}{n}\right)^{\frac{p}{2(p+1)}}.
	\end{align*}
\end{Prop}
\begin{Rem}
	\begin{itemize}
		\item[(i)]
		Note that the estimate on Proposition \ref{main_2} is almost optimal.
		Indeed, since $p \in (0,\infty)$ is arbitrarily, for any $\varepsilon \in (0,1)$, we can choose $p$ as $\frac{p}{2(p+1)}=\frac{1-\varepsilon}{2}$.
		Moreover, there exist $Y$ and $g \in BV$ such that $\int_{0}^{T}
		\e[|g(Y(s))-g(Y(\eta_{s}(s)))|^{q}]\rd s \geq C' n^{-1/2}$ for some $C'>0$, (see Remark 3.6 in \cite{NgTa16}).
		
		\item[(ii)]
		In the paper \cite{KoMaNo14} the authors consider rate of convergence of
		\begin{align*}
			\e\left[
				\left|
					\int_{0}^{T}
						g(Y(s))
					\rd s
					-
					\int_{0}^{T}
						g(Y(\eta_{n}(s)))
					\rd s
				\right|^{q}
			\right]
		\end{align*}
		for irregular function $g$ and solution of one-dimensional SDEs $X$ with smooth coefficients, (see also \cite{NgoOga11} for a central limit theorem for occupation time of diffusion processes).
		
	\end{itemize}
\end{Rem}
\begin{proof}[Proof of Proposition \ref{main_2}]
	From Proposition \ref{main_1}, it suffices to estimate
	\begin{align}\label{main_2_1}
		\int_{0}^{T}
			\e
				\left[
					\left|
						Y(s)
						-
						Y(\eta_{n}(s))
					\right|^p
				\right]
		\rd s.
	\end{align}
	Since $b$ and $\sigma$ are bounded, thus by using Burkholder-Davis-Gundy's inequality and Jensen's inequality, \eqref{main_2_1} is estimated above by
	$
		2^{p-1}
		T
		\{
			\|b\|_{\infty}^{p}(T/n)^{p}
			+
			\|\sigma\|_{\infty}^{p}c_{p}(T/n)^{p/2}
		\},
	$
	where $c_{p}$ is the constant of Burkholder-Davis-Gundy's inequality.
	This concludes the proof.
\end{proof}

\begin{proof}[Proof of Theorem \ref{main_3}]
We will only present the detail proof for the case that $b \in L^1(\real)$.
The proof for the case $b \not \in L^1(\real)$ is based on the  localisation technique given in \cite{NgTa1701} and it will be omitted.

(Step 1).
In order to deal with non-Lipschitz drift coefficient $b$, we apply the method of removal drift.
We define the scale function $\varphi (x) := \int_0^x \exp(-2 \int_0^y \frac{b(z)}{\sigma^2(z)} \rd z ) \rd y$, which is well-defined since $\sigma^2$ is uniformly positive.
We define $Y(t):=\varphi(X(t))$ and $Y^{(n)}(t):=\varphi(X^{(n)}(t))$.
Note that $\varphi{''}$ exists and satisfies $\varphi{''}=-\frac{2 b \varphi{'}}{\sigma^2}$ almost everywhere, and $\varphi$ satisfies the ordinary differential equation $b(x)\varphi{'}(x) + \frac{1}{2} \sigma^2(x) \varphi{''}(x) = 0$.
Hence by using generalized It\^o's formula (see, e.g. \cite{KS} Problem 3.7.3, page 219), we have
\begin{align*}
	Y(t)
	&=
	\varphi(x_0)
	+ \int_{0}^{t} \varphi{'}(X(s)) \sigma(X(s)) \rd B(s),\\
	Y^{(n)}(t)
	&= \varphi(x_0)
	+
	\int_{0}^{t}
		\varphi{'}(X^{(n)}(s))
		\sigma(X^{(n)}(\eta_n(s)))
	\rd B(s)\\
	&\quad
	-
	\int_{0}^{t}
		\varphi{'}(X^{(n)}(s))
		\left\{
			b(X^{(n)}(s))
			-
			b(X^{(n)}(\eta_n(s)))
		\right\}
	\rd s\\
	&\quad
	+
	\int_{0}^{t}
		\frac
			{\varphi{'}(X^{(n)}(s))b(X^{(n)}(s))}
			{\sigma^2(X^{(n)}(s))}
		\left\{
			\sigma(X^{(n)}(s))^{2}
			-
			\sigma(X^{(n)}(\eta_n(s)))^{2}
		\right\}
	\rd s.
\end{align*}
We denote $K_{\sigma}:=\inf_{x \in \real}\sigma(x) \vee \|\sigma\|_{\infty}^{-1}>0$ and $C_0:= e^{2\ksm^2 \|b\|_{L^1(\real)}}$.
Then it is easy to verify that the scale function $\varphi$ satisfies the following three properties;
(a) for any $x \in \real$, $C_0^{-1} \leq \varphi{'}(x) \leq C_0$;
(b) for any $x \in \real$,  $|\varphi{''}(x)| \leq 2\|b\|_\infty \ksm^2 C_0$;
(c) for any $z,w \in \mathrm{Dom}(\varphi^{-1})$, $|\varphi^{-1}(z)-\varphi^{-1}(w)| \leq C_0 |z-w|$, (see, e.g., \cite{NgTa1701}).

(Step 2).
In order to deal with Le Gall's condition of the diffusion coefficient $\sigma$, we use  Yamada and Watanabe approximation technique (see \cite{GyRa11} or \cite{YaWa}).
For each $\delta \in (1,\infty)$ and $\varepsilon \in (0,1)$, we define a continuous function $\psi _{\delta, \varepsilon}: \real \to [0,\infty)$ with $\text{supp}\: \psi _{\delta, \varepsilon}  \subset [\varepsilon/\delta, \varepsilon]$ such that
$\int_{\varepsilon/\delta}^{\varepsilon} \psi _{\delta, \varepsilon}(z) \rd z
= 1$ and $0 \leq \psi _{\delta, \varepsilon}(z) \leq \frac{2}{z \log \delta}$, $z > 0$.
Since $\int_{\varepsilon/\delta}^{\varepsilon} \frac{2}{z \log \delta} \rd z=2$, there exists such a function $\psi_{\delta, \varepsilon}$.
We define a function $\phi_{\delta, \varepsilon} \in C^2(\real;\real)$ by $\phi_{\delta, \varepsilon}(x):=\int_0^{|x|}\int_0^y \psi _{\delta, \varepsilon}(z)\rd z \rd y$.
It is easy to verify that $\phi_{\delta, \varepsilon}$ has the following three properties;
(i) $|x| \leq \varepsilon + \phi_{\delta, \varepsilon}(x)$, for any $x \in \real$; 
(ii) $0 \leq |\phi{'}_{\delta, \varepsilon}(x)| \leq 1$, for any $x \in \real$; 
(iii) $\phi{''}_{\delta, \varepsilon}(\pm|x|)=\psi_{\delta, \varepsilon}(|x|)
\leq \frac{2}{|x|\log \delta}{\bf 1}_{[\varepsilon/\delta, \varepsilon]}(|x|)$ for any $x \in \real \setminus\{0\}$. 
From the property (c) of the scale function $\varphi$ and the property (i) of $\phi_{\delta, \varepsilon}$, for any $t \in [0,T]$, we have
\begin{align}\label{esti_X1}
	|X(t)-X^{(n)}(t)|
	\leq
	C_{0}
	|Y(t)-Y^{(n)}(t)|
	\leq
	C_0 \left\{
		\varepsilon
		+
		\phi_{\delta,\varepsilon}(Y(t)-Y^{(n)}(t))
	\right\}.
\end{align}
Since $\phi_{\delta, \varepsilon} \in C^{2}(\real;\real)$, by using It\^o's formula, $\phi_{\delta,\varepsilon}(Y(t)-Y^{(n)}(t))$ can be decomposed by the following four terms
\begin{align}\label{esti_X2}
	\phi_{\delta,\varepsilon}(Y(t)-Y^{(n)}(t))
	=
	M^{n,\delta,\varepsilon}(t)
	+
	I_{1}^{n,\delta,\varepsilon}(t)
	+
	I_{2}^{n,\delta,\varepsilon}(t)
	+
	J^{n,\delta,\varepsilon}(t),
\end{align}
where
\begin{align*}
	M^{n,\delta,\varepsilon}(t)
	&:=
	\int_{0}^{t}
		\phi{'}_{\delta,\varepsilon}(Y(s)-Y^{(n)}(s))\\
		&\quad
		\times
		\left\{
			\varphi{'}(X(s))\sigma(X(s))
			-
			\varphi{'}(X^{(n)}(s)) \sigma(X^{(n)}(\eta_n(s)))
		\right\}
	\rd B(s),\\
	I_{1}^{n,\delta,\varepsilon}(t)
	&:=
	\int_{0}^{t}
		\phi{'}_{\delta,\varepsilon}(Y(s)-Y^{(n)}(s))
		\varphi{'}(X^{(n)}(s))
		\left\{
			b(X^{(n)}(s))
			-
			b(X^{(n)}(\eta_n(s)))
		\right\}
	\rd s,\\
	I_{2}^{n,\delta,\varepsilon}(t)
	&:=
	-
	\int_{0}^{t}
		\phi{'}_{\delta,\varepsilon}(Y(s)-Y^{(n)}(s))
		\frac
			{\varphi{'}(X^{(n)}(s))b(X^{(n)}(s))}
			{\sigma^2(X^{(n)}(s))}\\
		&\quad\times
		\left\{
			\sigma(X^{(n)}(s))^{2}
			-
			\sigma(X^{(n)}(\eta_n(s)))^{2}
		\right\}
	\rd s,\\
	J^{n,\delta,\varepsilon}(t)
	&:=
	\frac{1}{2}
	\int_{0}^{t}
		\phi{''}_{\delta,\varepsilon}(Y(s)-Y^{(n)}(s))\\
		&\quad \times
		\left|
			\varphi{'}(X(s))
			\sigma(X(s))
			-
			\varphi{'}(X^{(n)}(s))
			\sigma(X^{(n)}(\eta_n(s)))
		\right|^{2}
	\rd s.
\end{align*}
Note that since $\phi{'}$, $\varphi{'}$ and $\sigma$ are bounded, $(M^{n,\delta,\varepsilon}(t))_{0 \leq t \leq T}$ is martingale, so expectation of $M^{n,\delta,\varepsilon}(t)$ equals to zero.

(Step 3).
To conclude the statement, we estimate the expectation of $I_{1}^{n,\delta,\varepsilon}(t)$, $I_{2}^{n,\delta,\varepsilon}(t)$ and $J^{n,\delta,\varepsilon}(t)$ by using Proposition \ref{main_2}.

We first consider the expectation of $I_{1}^{n,\delta,\varepsilon}(t)$ and $I_{2}^{n,\delta,\varepsilon}(t)$.
From the property (a) of the scale function $\varphi$ and the property (ii) of $\phi_{\delta, \varepsilon}$, by using Jensen's inequality and Proposition \ref{main_2} with $g=f_{b}$ and $q=1$, we have
\begin{align}\label{main_3_2}
	\e\left[
		\left|
			I_{1}^{n,\delta,\varepsilon}(t)
		\right|
	\right]
	&\leq
	C_{0}
	\int_{0}^{T}
		\e\left[
			\left|
				b(X^{(n)}(s))
				-
				b(X^{(n)}(\eta_n(s)))
			\right|
		\right]
	\rd s \notag\\
	&\leq
	C_{0}T^{1-\gamma}
	\left(
		\int_{0}^{T}
			\e\left[
				\left|
					f_{b}(X^{(n)}(s))
					-
					f_{b}(X^{(n)}(\eta_n(s)))
				\right|
			\right]
		\rd s
	\right)^{\gamma} \notag\\
	&\leq
	C_{0}T^{1-\gamma}
	C(f_{b},b,\sigma,p,1)^{\gamma}
	\left(\frac{T}{n}\right)^{\frac{p \gamma}{2(p+1)}},
\end{align}
and Proposition \ref{main_2} with $g=f_{\sigma}$ and $q=1$, we have
\begin{align}\label{main_3_3}
	\e\left[
		\left|
			I_{2}^{n,\delta,\varepsilon}(t)
		\right|
	\right]
	&\leq
	\frac{2C_{0}\|b\|_{\infty}\|\sigma\|_{\infty}}{\inf_{x \in \real} \sigma(x)}
	\int_{0}^{T}
		\e\left[
			\left|
				\sigma(X^{(n)}(s))
				-
				\sigma(X^{(n)}(\eta_n(s)))
			\right|
		\right]
	\rd s\notag\\
	&\leq
	\frac{2C_{0}\|b\|_{\infty}\|\sigma\|_{\infty}}{\inf_{x \in \real} \sigma(x)}T^{\frac{1}{2}}
	C(f_{\sigma},b,\sigma,p,1)^{\frac{1}{2}}
	\left(\frac{T}{n}\right)^{\frac{p}{4(p+1)}}.
\end{align}

Next, we consider the expectation of $J^{n,\delta,\varepsilon}(t)$.
From the property (iii) of $\phi_{\delta, \varepsilon}$, we have
\begin{align*}
	J^{n,\delta,\varepsilon}(t)
	&\leq
	\int_{0}^{T}
		\frac
			{\1_{[\varepsilon/\delta,\varepsilon]}(|Y(s)-Y^{(n)}(s)|)}
			{|Y(s)-Y^{(n)}(s)| \log \delta}
		\left|
			\varphi{'}(X(s))
			\sigma(X(s))
			-
			\varphi{'}(X^{(n)}(s))
			\sigma(X^{(n)}(\eta_n(s)))
		\right|^2
	\rd s\\
	&\leq
	3\{
		J_{1}^{n,\delta,\varepsilon}(T)
		+
		J_{2}^{n,\delta,\varepsilon}(T)
		+
		J_{3}^{n,\delta,\varepsilon}(T)
	\},
\end{align*}
where
\begin{align*}
	J_{1}^{n,\delta,\varepsilon}(t)
	&:=
	\int_{0}^{t}
		\frac
			{\1_{[\varepsilon/\delta,\varepsilon]}(|Y(s)-Y^{(n)}(s)|)}
			{|Y(s)-Y^{(n)}(s)| \log \delta}
		|\sigma(X(s))|^2
		\left|
			\varphi{'}(X(s))
			-
			\varphi{'}(X^{(n)}(s))
		\right|^{2}
	\rd s,\\
	J_{2}^{n,\delta,\varepsilon}(t)
	&:=
	\int_{0}^{t}
		\frac
			{\1_{[\varepsilon/\delta,\varepsilon]}(|Y(s)-Y^{(n)}(s)|)}
			{|Y(s)-Y^{(n)}(s)| \log \delta}
		|\varphi{'}(X^{(n)}(s))|^{2}
		\left|
			\sigma(X(s))
			-
			\sigma(X^{(n)}(s))
		\right|^{2}
	\rd s, \\
	J_{3}^{n,\delta,\varepsilon}(t)
	&:=
	\int_{0}^{t}
		\frac
			{\1_{[\varepsilon/\delta,\varepsilon]}(|Y(s)-Y^{(n)}(s)|)}
			{|Y(s)-Y^{(n)}(s)| \log \delta}
			|\varphi{'}(X^{(n)}(s))|^{2}
			\left|
				\sigma(X^{(n)}(s))
				-
				\sigma(X^{(n)}(\eta_n(s)))
			\right|^{2}
		\rd s.
\end{align*}

We first consider $J_{1}^{n,\delta,\varepsilon}(T)$.
From the property of the scale function (b), $\varphi{'}$ is Lipschitz continuous with Lipschitz constant $\|\varphi{''}\|_{\infty}$, thus we have
\begin{align}\label{esti_J1}
	J_{1}^{n,\delta,\varepsilon}(T)
	&\leq
	\frac{\ksm^2 \|\varphi{''}\|_{\infty}^2}{\log \delta}
	\int_{0}^{T}
		\frac
			{\1_{[\varepsilon/\delta,\varepsilon]}(|Y(s)-Y^{(n)}(s)|)}
			{|Y(s)-Y^{(n)}(s)|}
		\left|
			X(s)
			-
			X^{(n)}(s)
		\right|^2
	\rd s \notag\\
	&\leq
	\frac{\ksm^2 \|\varphi{''}\|_{\infty}^2 C_0^2}{\log \delta}
	\int_{0}^{T}
		\1_{[\varepsilon/\delta,\varepsilon]}(|Y(s)-Y^{(n)}(s)|)
		\left|
			Y(s)
			-
			Y^{(n)}(s)
		\right|
	\rd s \notag\\
	&\leq
	4\ksm^6 C_0^4 \|b\|_\infty^2 T
	\frac{\varepsilon}{\log \delta}.
\end{align}

Next we consider $J_{2}^{n,\delta,\varepsilon}(T)$.
This part is based on the argument in \cite{LeGall}.
By using the property of the scale function (a) and the assumption on $\sigma$, we have
\begin{align*}
	J_{2}^{n,\delta,\varepsilon}(T)
	\leq
	\frac{C^{3}_{0}}{\log \delta}
	\int_{0}^{T}
		\frac
			{\left|
				f_{\sigma}(X(s))
				-
				f_{\sigma}(X^{(n)}(s)) 
			\right|}
			{|X(s)-X^{(n)}(s)|}
		\1_{|X(s)- X^{(n)}(s)|\geq  \varepsilon/(C_{0}\delta)}
	\rd s.
\end{align*}
We consider approximation $f_{\sigma,\ell} \in C^1(\real)$ of $f_{\sigma}$ which is also strictly increasing function and satisfies $\|f_{\sigma, \ell}\|_{\infty} \leq \|f_{\sigma}\|_{\infty}$ and $f_{\sigma,\ell} \uparrow f_{\sigma}$ as $\ell \to \infty$ on $\real$.
Then by using Fatou's lemma and the mean value theorem, we have
\begin{align}\label{pr_1_5}
	J_{2}^{n,\delta,\varepsilon}(T)
	&\leq
	\frac{C^3_0}{\log \delta}
	\int_{0}^{T}
		\frac
			{|f_{\sigma}(X(s))-f_{\sigma}(X^{(n)}(s))|}
			{|X(s)-X^{(n)}(s)|}
		\1_{|X(s)-X^{(n)}(s)|> \varepsilon/(C_0\delta)}
	\rd s  \notag\\
	&\leq
	\liminf_{\ell \to \infty}
	\frac{C^3_0}{\log \delta}
	\int_{0}^{T}
		\frac
		{|f_{\sigma,\ell}(X(s))-f_{\sigma,\ell}(X^{(n)}(s))|}
		{|X(s)-X^{(n)}(s)|}
		\1_{|X(s)-X^{(n)}(s)|> \varepsilon/(C_0\delta)}
	\rd s \notag\\
	&\leq  \liminf_{\ell \to \infty}
	\frac{C_0^3}{\log \delta}	
	\int_{0}^{T}\rd s \int_{0}^{1} \rd \theta f{'}_{\sigma, \ell}(V_{s}^{(n)}(\theta)),
\end{align}
where $V_t^{(n)}(\theta):=(1-\theta)X(t)+\theta X^{(n)}(t)$.
Since $\sigma$ is uniformly positive, the quadratic variation of $V^{(n)}(\theta)$ satisfies
\begin{align*}
	\langle V^{(n)}(\theta) \rangle_t
	=
	\int_{0}^{t}
		\left\{
			(1-\theta)\sigma(X(s))
			+
			\theta
			\sigma(X^{(n)}(\eta_n(s)))
		\right\}^{2}
	\rd s
	\geq
	\left(\inf_{x \in \real} \sigma(x)\right)^{2} t.
\end{align*} 
Therefore, by using the occupation time formula, we have
\begin{align*}
	\int_{0}^{T} \rd s
		\int_{0}^{1} \rd \theta
			f{'}_{\sigma, \ell}(V_s^{(n)}(\theta))
	&\leq
	\left(\inf_{x \in \real} \sigma(x)\right)^{-2}
		\int_{0}^{1} \rd \theta
			\int_{0}^{T} \rd \langle V^{(n)}(\theta) \rangle_{s}
				f{'}_{\sigma, \ell}(V_{s}^{(n)}(\theta)) \notag\\
	&=
	\left(\inf_{x \in \real} \sigma(x)\right)^{-2}
		\int_{\real}\rd x
			f{'}_{\sigma, \ell}(x)
			\int_{0}^{1} \rd \theta
				L_{T}^{x}(V^{(n)}(\theta)),
\end{align*}
where $L_{t}^{x}(V^{(n)}(\theta))$ the symmetric local time of $V^{(n)}(\theta)$ up to time $t$ at the level $x \in \real$.
By using  Lemma \ref{local_time} and the estimate $\|f{'}_{\sigma, \ell}\|_{L^1(\real)} \leq 2 \|f_{\sigma, \ell}\|_{\infty} \leq 2 \|f_{\sigma} \|_{\infty}$ we have
\begin{align*}
	\e\left[
		\int_{0}^{T}\rd s
			\int_{0}^{1} \rd \theta
				f{'}_{\sigma, \ell}(V_{s}^{(n)}(\theta))
	\right]
	&\leq
	\left(\inf_{x \in \real} \sigma(x)\right)^{-2}
		\int_{\real} \rd x
			f{'}_{\sigma, \ell}(x)
			\int_{0}^{1} \rd \theta
				\e [L_{T}^{x}(V^{(n)}(\theta))] \\
	&\leq
	\left(\inf_{x \in \real} \sigma(x)\right)^{-2}
	\|f{'}_{\sigma, \ell}\|_{L^{1}(\real)}
	\sup_{\theta \in [0,1], x \in \real}
		\e [|L_T^x(V^{(n)}(\theta))|^2]^{1/2}\\
	&\leq
	2 \left(\inf_{x \in \real} \sigma(x)\right)^{-2}
	\|f_{\sigma}\|_{\infty} 
	\{ 12\|b\|_{\infty}^2T^2+6 \overline{\sigma}^2 T\}^{1/2}.
\end{align*}
By plugging this estimate to \eqref{pr_1_5} and using Fatou's lemma, we get the following estimate for the expectation of $J_{2}^{n,\delta,\varepsilon}(T)$;
\begin{align}\label{pr_1_7}
	\e[J_{2}^{n,\delta,\varepsilon}(T)]
	&\leq
	2C_0^3\underline{\sigma}^{-2} \|f_{\sigma}\|_{\infty} \{ 12\|b\|_{\infty}^2T^2+6 \overline{\sigma}^2 T\}^{1/2}
	\frac{1}{\log \delta}.
\end{align}

Finally, we consider $J_{3}^{n,\delta,\varepsilon}(T)$.
By using Proposition \ref{main_2} with $g=f_{\sigma}$ and $q=1$, we have
\begin{align}\label{eqnL8}
	\e[J_{3}^{n,\delta,\varepsilon}(T)]
	&\leq
	\frac{C_0^2 \delta}{\varepsilon \log \delta}
	\int_{0}^{T}
		\e\left[
			\left|
				f_{\sigma}(X^{(n)}(s))
				-
				f_{\sigma}(X^{(n)}(\eta_n(s)))
			\right|
		\right]
	\rd s \notag\\
	&\leq
	C_0^2 
	C(f_{\sigma},b,\sigma,p,q)
	\frac{\delta}{\varepsilon \log \delta}
	\left(\frac{T}{n}\right)^{\frac{p}{2(p+1)}}.
\end{align}

Since $\e [M^{n,\delta, \varepsilon}(t)] = 0$, it follows from \eqref{esti_X1}, \eqref{esti_X2}, \eqref{main_3_2},  \eqref{main_3_3}, \eqref{esti_J1}, \eqref{pr_1_7} and \eqref{eqnL8} that there exists a positive constant $C$ which does not depend on $n$ such that
\begin{align*}
	&\sup_{0\leq t \leq T}
	\e[|X(t)-X^{(n)}(t)|]\\
	&\leq
	C
	\left\{
		\varepsilon
		+
		\frac{1}{n^{\frac{p \gamma}{2(p+1)}}}
		+
		\frac{1}{n^{\frac{p}{4(p+1)}}}
		+
		\frac{\varepsilon}{\log \delta}
		+
		\frac{1}{\log \delta}
		+
		\frac{\delta}{\varepsilon \log \delta}
		\frac{1}{n^{\frac{p}{2(p+1)}}}
	\right\}.
\end{align*}
By choosing $\varepsilon = 1/\log n$ and $\delta = n^{\frac{p}{4(p+1)}}$, we conclude the proof.
\end{proof}

\subsubsection*{Case 2 : time dependent coefficients}
Let us consider the following one-dimensional SDEs of the form
\begin{align*}
	\rd X(t)
	=
	b(t,X(t)) \rd t
	+
	\sigma(t,X(t))\rd B(t),
	~X(0)=x_{0} \in \real,
	~t \in [0,T],
\end{align*}
and its Euler--Maruyama scheme of the form
\begin{align*}
	\rd X^{(n)}(t)
	=
	b(\eta _n(t), X^{(n)}(\eta _n(t))) \rd t
	+
	\sigma(\eta _n(t), X^{(n)}(\eta _n(s))) \rd B(t),~
	X^{(n)}(0)=x_{0},~
	t \in [0,T].
\end{align*}
Since the coefficients are time dependent, we cannot apply removal drift technique used in the proof of Theorem \ref{main_3}.
However, if the drift coefficient $b$ satisfies one-sided Lipschitz condition, additionally, then we have the rate of strong convergence for the Euler--Maruyama scheme.

\begin{Thm}\label{main_10}
	Suppose that coefficients $b$ and $\sigma$ are measurable, bounded and $\sigma$ is uniformly positive.
	Moreover, assume that there exist $\gamma \in (0,1]$, $f_{b} \in BV$ and bounded and strictly increasing function and $f_{\sigma}$ such that for any $x,y \in \real$ and $s,t \in [0,T]$,
	\begin{align*}
		|b(t,x)-b(t,y)|
		&\leq
		|f_{b}(x)-f_{b}(y)|^{\gamma}
		\quad\text{and}\quad
		|\sigma(t,x)-\sigma(t,y)|^{2}
		\leq
		|f_{\sigma}(x)-f_{\sigma}(y)|,
	\end{align*}
	and additionally, there exists $K\geq0$ and $\alpha \in (0,1]$ such that
	\begin{align*}
		(x-y)(b(t,x)-b(t,y))
		&\leq K|x-y|^{2}\\
		|b(t,x)-b(s,x)|
		+
		|\sigma(t,x)-\sigma(s,x)|^{2}
		&\leq
		K|t-s|^{\alpha}.
	\end{align*}
	Then there exists a constant $C$ such that for any $n \geq 2$,
	\begin{align*}
		\sup_{0\leq t \leq T}
		\e\left[
			\left|
				X(t)
				-
				X^{(n)}(t)
			\right|
		\right]
		\leq
		\frac{C}{\log n}.
	\end{align*}
\end{Thm}
\begin{proof}
	We again apply Yamada and Watanabe approximation technique which is used in the proof of Theorem \ref{main_3}.
	By using It\^o's formula, $\phi_{\delta,\varepsilon}(X(t)-X^{(n)}(t))$ can be decomposed by the following four terms
	\begin{align*}
		\phi_{\delta,\varepsilon}(X(t)-X^{(n)}(t))
		=
		M^{n,\delta,\varepsilon}(t)
		+
		I_{1}^{n,\delta,\varepsilon}(t)
		+
		I_{2}^{n,\delta,\varepsilon}(t)
		+
		J^{n,\delta,\varepsilon}(t),
	\end{align*}
	where
	\begin{align*}
		M^{n,\delta,\varepsilon}(t)
		&:=
		\int_{0}^{t}
			\phi^{'}_{\delta,\varepsilon}(X(s)-X^{(n)}(s))
			\left\{
				\sigma(s,X(s))
				-
				\sigma(\eta_{n}(s),X^{(n)}(\eta_n(s)))
			\right\}
		\rd B(s),\\
		I_{1}^{n,\delta,\varepsilon}(t)
		&:=
		\int_{0}^{t}
			\phi{'}_{\delta,\varepsilon}(X(s)-X^{(n)}(s))
			\left\{
				b(s,X(s))
				-
				b(s,X^{(n)}(s))
			\right\}
		\rd s,\\
		I_{2}^{n,\delta,\varepsilon}(t)
		&:=
		\int_{0}^{t}
			\phi^{'}_{\delta,\varepsilon}(X(s)-X^{(n)}(s))
			\left\{
				b(s,X^{(n)}(s))
				-
				b(\eta_{n}(s),X^{(n)}(\eta_n(s)))
			\right\}
		\rd s,\\
		J^{n,\delta,\varepsilon}(t)
		&:=
		\frac{1}{2}
		\int_{0}^{t}
			\phi^{''}_{\delta,\varepsilon}(X(s)-X^{(n)}(s))
			\left|
				\sigma(s,X(s))
				-
				\sigma(\eta_n(s),X^{(n)}(\eta_n(s)))
			\right|^{2}
		\rd s.
	\end{align*}
	To conclude the statement, it is sufficient to estimate $I_{1}^{n,\delta,\varepsilon}(t)$.
	The other terms can be estimated by the same way as the proof of Theorem \ref{main_3} or regularity of time variables on the coefficients.
	Since $\phi{'}_{\delta,\varepsilon}$ satisfies $\phi{'}_{\delta,\varepsilon}(x)/x>0$, $x \neq 0$, by using one-sided Lipschitz condition on $b$, we have $\phi{'}_{\delta,\varepsilon}(x-y)(b(s,x)-b(s,y)) \leq K|x-y|$, for any $x,y \in \real$ and $s \in [0,T]$.
	Therefore, it holds that
	\begin{align}\label{pr_0}
		\e\left[
			\left|
				I_{1}^{n,\delta,\varepsilon}(t)
			\right|
		\right]
		\leq
		K \int_{0}^{t}
			\e\left[
				\left|
					X(s)
					-
					X^{(n)}(s)
				\right|
			\right]
		\rd s.
	\end{align}
	Therefore, since $\e [M^{n,\delta, \varepsilon}(t)] = 0$, there exists a positive constant $C$ which do not depend on $n$ such that for any $t \in [0,T]$
	\begin{align*}
		&\e[|X(t)-X^{(n)}(t)|]
		\leq
		\varepsilon
		+
		K \int_{0}^{t}
			\e\left[
				\left|
					X(s)
					-
					X^{(n)}(s)
				\right|
			\right]
		\rd s\\
		&\quad+
		C\left\{
			\frac{1}{n^{\frac{p \gamma}{2(p+1)}}}
			+
			\frac{1}{n^{\frac{p}{4(p+1)}}}
			+
			\frac{1}{\log \delta}
			+
			\frac{\delta}{\varepsilon \log \delta}
			\frac{1}{n^{\frac{p}{2(p+1)}}}
			+
			\frac{1}{n^{\alpha}}
			+
			\frac{\delta}{\varepsilon \log \delta}
			\frac{1}{n^{\alpha}}
		\right\}
	\end{align*}
	By choosing $\varepsilon = 1/\log n$ and $\delta = n^{\frac{p}{4(p+1)} \wedge \frac{\alpha}{2}}$, we conclude the proof.
\end{proof}

\subsubsection*{Application to singular SDEs}
As application of Theorem \ref{main_3}, we consider the following one-dimensional SDEs with symmetric local time of the form
\begin{align}\label{SDE_local}
	X(t)
	=
	x_{0}
	+
	\int_{0}^{t}
		\sigma(X(s))
	\rd B(s)
	+
	\int_{\real}
		L_{t}^{a}(X)
	\nu (\rd a),
	~x_{0} \in \real,
	~t \in [0,T],
\end{align}
where $\nu$ is a signed measure on $\real$ satisfying $0 \leq |\nu(\{a\})|<1$ for any $a \in \real$.
We suppose that $\sigma$ is right continuous, bounded and uniformly positive.

\begin{Rem}\label{Rem_app_SDE}
	\begin{itemize}
		\item[(i)]
		Note that SDEs with discontinuous coefficients considered in subsection \ref{sec_3_2}, and singular SDEs \eqref{SDE_local} are applied in mathematical finance \cite{AkIm14,CoSa07,DeGoSc06,DeScGo04,GaSh17}, optimal control problems \cite{BSW,Lee94} and problems in multi-layered media \cite{Ra11}. 
		In particular, the equation \eqref{SDE_local} are related to (generalized) skew Brownian motions \cite{HaSh, Lejay07}.
		More preciously, let $\sigma=1$ and $\nu=(2\alpha-1) \delta_{0}$ with $|2\alpha-1| \in (0,1)$, (or more generally $\nu=\sum_{i=1}^{M} (2\alpha_{i}-1) \delta_{a_{i}}$ with $|2\alpha_i-1|<1$, $i=1,\ldots,M$), then the equation is called skew Brownian motion.
		Harrison and Shepp \cite{HaSh} proved that if $|2\alpha-1| \leq 1$ then there is a unique strong solution and if $|2\alpha-1|>1$ and $x_{0}=0$, there is no solution.
		Numerical schemes for this type of equations are studied (see, \cite{Et06,EtMa13,Fri16,KoTaZh16,LeMa,MaTa12}).
		
		\item[(ii)]
		Suppose that the distributional derivative of $\sigma$ is a signed Radon measure and $\nu(\rd a):=\frac{1}{2}\sigma^{-1}(x)\sigma'(\rd a)$ and $0<|\nu\{a\}|<1$.
		Then as mentioned in introduction, SDE \eqref{SDE_local} is a diffusion process associated to the parabolic equation in divergence form, that is, its infinitesimal generator is given by $\frac{1}{2} \frac{\rd }{\rd x}\left(\sigma^{2} \frac{\rd }{\rd x}\right)$ (see, Theorem 3.6 in \cite{Bach01}).
	\end{itemize}
\end{Rem}

Now we consider a transformation of the equation \eqref{SDE_local} in order to remove the drift part (see, for more details Proposition 2.2 in \cite{LeGall}).
We define a function $f_{\nu}$ by
\begin{align*}
	f_{\nu}(x)
	:=
	\exp(-2\nu_c((-\infty,x]))
	\prod_{-\infty < y \leq x} 
	\frac{1-\nu(\{y\})}{1+\nu(\{y\})},
\end{align*}
where $\nu_c$ is a continuous part of $\nu$.
Then $f_{\nu}$ is right continuous, non-increasing, $\lim_{x \to -\infty} f_{\nu}(x)=1$ and for any $x \in \real$, 
\begin{align}\label{bddUE_f}
	K_{\nu} \leq f_{\nu}(x) \leq 1,~K_{\nu} := \exp(-2\nu_c(\real)) \prod_{-\infty < y < \infty} \frac{1-\nu(\{y\})}{1+\nu(\{y\})}>0.
\end{align}
We set $F_{\nu}(x):=\int_0^x f_{\nu}(y)dy$.
Then it holds from \eqref{bddUE_f} that for any $x,y \in \real$ and $z,w \in \mathrm{Dom}(F_{\nu}^{-1})$,
\begin{align}\label{lip_F}
|F_{\nu}(x)-F_{\nu}(y)|
\leq |x-y| \text{ and }
|F_{\nu}^{-1}(z)-F_{\nu}^{-1}(w)|
\leq K_{\nu}^{-1}|z-w|.
\end{align}
We define the second derivative measure $f_{\nu}'(\rd a)$ associated with $f_{\nu}$, that is, it satisfies
\begin{align*}
	\int_{\real}
		g(a)
	f{'}_{\nu}(\rd a)
	=
	-
	\int_{\real}
		g'(a)
		D_{\ell}
		F_{\nu}(a)
	\rd a,~
\end{align*}
where $D_{\ell}f$ is the left derivative of $f$, and $g$ is a differentiable function with compact support on $\real$.
Then, it holds that (see Lemma 2.1 in \cite{LeGall}),
\begin{align}\label{lem_LeGall}
	f_{\nu}'(\rd a)
	+
	(f_{\nu}(a)
	+
	f_{\nu}(a-))\nu(\rd a)
	=0.
\end{align}
Define a stochastic process $Y$ by $Y(t):=F_{\nu}(X(t))$, then by using the symmetric It\^o--Tanaka formula and \eqref{lem_LeGall}, since $L_{t}^{a}(X)$ only increases when $X(t)=a$, we have
\begin{align*}
	Y(t)
	&=
	F_{\nu}(x_{0})
	+
	\frac{1}{2}
	\int_{0}^{t}
		\left\{
			f_{\nu}(X(s))+f_{\nu}(X(s)-)
		\right\}
		\sigma(X(s))
	\rd B(s) \notag \\
	&\quad
	+
	\frac{1}{2}
	\int_{\real}
		\int_{0}^{t}
			\left\{
				f_{\nu}(X(s))+f_{\nu}(X(s)-)
			\right\}
		\rd L_{s}^{a}(X)
	\nu(\rd a)
	+
	\frac{1}{2}
		\int_{\real} L_{t}^{a}(X)f_{\nu}'(\rd a) \notag \\
	&=
	F_{\nu}(y_{0})
	+
	\frac{1}{2}
	\int_{0}^{t}
		\left\{
			f_{\nu}(X(s))+f_{\nu}(X(s)-)
		\right\}
		\sigma(X(s))
	\rd B(s).
\end{align*}
Since $\sigma$ is uniformly positive, $\int_{0}^{T}\p(X(t) = x) \rd t=0$ for any $x \in \real$, thus $Y$ is a solution of SDE without drift term of the form
\begin{align*}
	\rd Y(t)
	=
	(f_{\nu} \cdot \sigma) \circ F_{\nu}^{-1} (Y(t))
	\rd B(t),~
	Y(0)=F_{\nu}(x_{0}).
\end{align*}
Note that $(f_{\nu} \cdot \sigma) \circ F_{\nu}^{-1}$ is of bounded variation and right continuous, bounded and uniformly positive, thus it satisfies assumptions of $\sigma$ in Theorem \ref{main_3}.
Thus since $F_{\nu}^{-1}$ is Lipschitz continuous (see \eqref{lip_F}), a solution of SDE \eqref{SDE_local} can be approximated as follows.

\begin{Cor}\label{Cor_3}
	Let $X$ be a solution of SDE \eqref{SDE_local}.
	Suppose that $\nu$ is a signed measure on $\real$ satisfying $0 \leq |\nu(\{a\})|<1$ for any $x \in \real$ and $\sigma$ is right continuous, bounded and uniformly positive.
	Then there exists $C>0$ such that for any $n \geq 2$,
	\begin{align*}
	\sup_{0\leq t \leq T}
	\e\left[
		\left|
			X(t)
			-
			F_{\nu}^{-1}(Y^{(n)}(t))
		\right|
	\right]
	\leq
	\frac{C}{\log n},
	\end{align*}
	where $Y^{(n)}$ is the Euler--Maruyama scheme of $Y$ defined by
	$
	\rd Y^{(n)}(t)
	=
	(f_{\nu} \cdot \sigma) \circ F_{\nu}^{-1} (Y^{(n)}(\eta_{n}(t)))
	\rd B(t)
	$,
	$Y^{(n)}(0)=F_{\nu}(x_{0})$.
\end{Cor}

\subsection{SDEs with super-linearly growing and irregular coefficients}\label{sec_3_3}

In this subsection, inspired by \cite{HuJeKl12,NgoLu17,Sa13,Sa16}, we consider a tamed Euler--Maruyama scheme, in order to approximate a solution of one--dimensional SDEs \eqref{SDE_1} 

\subsubsection*{Case 1 : super-linear growing diffusion coefficient}
We first consider that the coefficients $b:\real \to \real$ and $\sigma: \real \to \real$ satisfies the following conditions.

\begin{Ass}\label{Ass_super_0}
	We suppose that the coefficients $b:\real \to \real$ and $\sigma:\real \to \real$ are measurable and satisfy the following conditions:
	\begin{itemize}
		\item[(i)]
		(Khasminskii and one-sided Lipschitz condition)
		There exist $K>0$, $p_{0} >2$ and $p_{1} >2$ such that for each $x,y \in \real$,
		\begin{align*}
		2 x b(x)
		+(p_{0}-1)|\sigma(x)|^{2}
		&\leq
		K(1+|x|^{2}),\\
		2(x-y) (b(x)-b(y))
		+(p_{1}-1)
		|\sigma(x)-\sigma(y)|^{2}
		&\leq
		K|x-y|^{2}.
		\end{align*}
		
		\item[(ii)]
		(locally bounded $1/\gamma$-variation and polynomial growth)
		There exist $K>0$, $f_{b} \in BV$, $\gamma \in (0,1]$ and $\ell \in (0,\frac{p_{0}-2}{4}]$ such that for each $x,y \in \real$,
		\begin{align*}
			|b(x)-b(y)|
			&\leq
			K(1+|x|^{\ell}+|y|^{\ell}) |f_{b}(x)-f_{b}(y)|^{\gamma}\\
			|b(x)|
			&\leq
			K(1+|x|^{\ell+1}).
		\end{align*}
		
		\item[(iii)]
		The diffusion coefficient $\sigma$ is uniformly elliptic, that is, $\inf_{x \in \real} \sigma(x)^{2} >0$.
	\end{itemize}
\end{Ass}


\begin{Rem}
	\begin{itemize}
		\item[(i)]
		Note that under Assumption \ref{Ass_super_0} (i), the unique strong solution of SDE \eqref{SDE_1} can be constructed as a limit (in probability) of standard Euler--Maruyama scheme (see, \cite{GyKr96}).
		
		\item[(ii)]
		From Assumption \ref{Ass_super_0}, there exists $K_0>0$ such that for any $x,y \in \real$,
		\begin{align}\label{local_sigma_1}
			(p_{0}-1)|\sigma(x)|^2
			&\leq
			K(1+|x|^2)
			-2xb(x)
			\leq
			(p_{0}-1)K_0(1+|x|^{\ell+2}),
		\end{align}
		and
		\begin{align}\label{local_sigma_01}
			(p_{1}-1)|\sigma(x)-\sigma(y)|^2
			&\leq
			K|x-y|^2
			-2(x-y) (b(x)-b(y))\\
			&\leq
			(K+1)|x-y|^2
			+
			3K^{2}(1+|x|^{2\ell}+|y|^{2\ell}) |f_{b}(x)-f_{b}(y)|^{2\gamma} \notag.
		\end{align}
	\end{itemize}
\end{Rem}

Now inspired by \cite{HuJeKl12,Sa13,Sa16}, we consider a tamed Euler--Maruyama scheme for a solution of SDE \eqref{SDE_1}, defined by
\begin{align}\label{tamed_EM_0}
	\rd X^{(n)}(t)
	=
	b_{n}(X^{(n)}(\eta _n(t))) \rd t
	+
	\sigma_{n}(X^{(n)}(\eta _n(s))) \rd B(t),~
	X^{(n)}(0)=x_{0},~
	t \in [0,T],
\end{align}
where
\begin{align*}
	b_{n}(x)
	:=\frac{b(x)}{1+n^{-1/2} |x|^{\ell}}
	\quad\text{and}\quad
	\sigma_{n}(x)
	:=\frac{\sigma(x)}{1+n^{-1/4} |x|^{\ell/2}}.
\end{align*}
Then it holds that for any $x \in \real$ and $n \in \n$,
\begin{align}
	|b_{n}(x)|
	&\leq
	\{Kn^{1/2}(1+|x|)\} \wedge |b(x)|,
	\label{tamed_coef_0}\\
	|\sigma_{n}(x)|^2
	&\leq
	\{K_0n^{1/2}(1+|x|^2)\} \wedge |\sigma(x)|^2,
	\label{tamed_coef_01}
\end{align}
	and
\begin{align}
	|b(x)-b_{n}(x)|
	&=
	\frac{|b(x)| |x|^{\ell} n^{-1/2}}{1+n^{1/2} |x|^{\ell}}
	\leq
	K(1+|x|^{\ell+1})|x|^{\ell}n^{-1/2},
	\label{tamed_coef_1}\\
	|\sigma(x)-\sigma_{n}(x)|^2
	&=
	\frac{|\sigma(x)|^2 |x|^{\ell} n^{-1/2}}{(1+n^{-1/4} |x|^{\ell/2})^2}
	\leq 
	K_0(1+|x|^{\ell+2})|x|^{\ell} n^{-1/2}.
	\label{tamed_coef_2}
\end{align}

Under Assumption \ref{Ass_super_0}, we have the following rate of strong convergence for the tamed Euler--Maruyama scheme \eqref{tamed_EM_0}.

\begin{Thm}\label{main_4}
	Suppose that Assumption \ref{Ass_super_0} holds.
	Then for any $p \in [2,p_{0}/(2\ell+1)] \cap [2,p_{1})$, there exists $C>0$ such that
	\begin{align*}
		\sup_{0\leq t \leq T}
		\e[|X(t) - X^{(n)}(t)|^{p}]^{1/p}
		\leq
		Cn^{-(r(p) \wedge \frac{1}{4})}.
	\end{align*}
	where
	\begin{align*}
		r(p)
		:=
		\frac{\gamma (\ell+1)}{2p_{0}+\ell+2} \frac{p_{0}}{p(2\ell+1)}
		\in
		\left[
			\frac{8\gamma}{9p_{0}+6},
			\frac{\gamma p_{0}(p_{0}+6)}{4(p_{0}+1)}
		\right).
	\end{align*}
\end{Thm}

For proving Theorem \ref{main_4}, we first note that a solution of SDEs and the tamed Euler--Maruyama scheme \eqref{tamed_EM_0} have a moment.

\begin{Lem}\label{Lem_super_0}
	Under Assumption \ref{Ass_super_0}, it holds that for any $q \leq p_{0}$, there exists $C_{q}$ such that
	\begin{align*}
		\sup_{0\leq t \leq T}
		\e
		\left[
			\left|
				X(t)
			\right|^{q}
		\right]
		\vee
		\sup_{n\in\n}
		\sup_{0\leq t \leq T}
		\e
		\left[
			\left|
				X^{(n)}(t)
			\right|^{q}
		\right]
		\leq
		C_{q}.
	\end{align*}
\end{Lem}
\begin{proof}
	It follows from \eqref{tamed_coef_0} and \eqref{tamed_coef_01} that $b_{n}$ and $\sigma_{n}$ satisfies the conditions B-1, B-2 and B-3 in \cite{Sa16}.
	Thus from Lemma 2 in \cite{Sa16}, we conclude the proof.
\end{proof}

Using the above lemma, we have the following estimate.

\begin{Lem}\label{Lem_super_1}
	Suppose Assumption \ref{Ass_super_0} holds.
	Then there exists $C>0$ such that
	\begin{align*}
		\sup_{0\leq t \leq T}
		\e\left[
			\left|
				X^{(n)}(t)
				-
				X^{(n)}(\eta_{n}(t))
			\right|^{q}
		\right]
		\leq
		\left\{ \begin{array}{ll}
		\displaystyle
		Cn^{-q/4}
		&\text{ if } q \in [2,p_{0}],  \\
		\displaystyle
		Cn^{-q/2}
		&\text{ if } q \in [2,2p_{0}/(\ell+2)].
		\end{array}\right.
	\end{align*}
\end{Lem}
\begin{proof}
	By using, \eqref{local_sigma_1}, \eqref{tamed_coef_0} and \eqref{tamed_coef_01}, we have
	\begin{align*}
		&\left|
			X^{(n)}(t)
			-
			X^{(n)}(\eta_{n}(t))
		\right|^{q}\\
		&\leq
		2^{q-1}
		K^{q}n^{q/2}(1+|X^{(n)}(\eta_{n}(t))|)^{q}
		|t-\eta_{n}(t)|^{q}\\
		&\quad
		+
		\left\{ \begin{array}{ll}
		\displaystyle
		2^{q-1}
		K_{0}^{q/2}n^{q/4}(1+|X^{(n)}(\eta_{n}(t))|^{2})^{q/2}
		|B(t)-B(\eta_{n}(t))|^{q},\\
		\displaystyle
		2^{q-1}
		K_{0}^{q/2} (1+|X^{(n)}(\eta_{n}(t))|^{\ell+2})^{q/2}
		|B(t)-B(\eta_{n}(t))|^{q}.
		\end{array}\right.
	\end{align*}
	By Lemma \ref{Lem_super_0}, we conclude the proof.
\end{proof}

The generalized Avikainen's estimate shows the following error estimate for the tamed Euler--Maruyama scheme.

\begin{Prop}\label{Lem_super_3}
	Suppose Assumption \ref{Ass_super_0} holds and $p_{0} \geq \ell+2$.
	Then for any $g \in BV$, $T>0$ and $q \in [1,\infty)$ there exists $C=C(g,b,\sigma,p,T)>0$ such that
	\begin{align*}
		&\int_{0}^{T}
			\e\left[
				\left|
					g(X^{(n)}(s))
					-
					g(X^{(n)}(\eta_{s}(s)))
				\right|^{q}
			\right]
		\rd s
		\leq
		C
		n^{-\frac{p_{0}}{2p_{0}+\ell+2}}.
	\end{align*}
\end{Prop}
\begin{proof}
	We first prove
	\begin{align}\label{Lem_super_2}
		\sup_{n \in \n,~x \in \real}
		\e[|L_{T}^{x}(X^{(n)})|^2]
		&
		<
		\infty.
	\end{align}
	By using the symmetric It\^o--Tanaka formula, we have
	\begin{align*}
		&L_{T}^{x}(X^{(n)})
		\leq
		|X^{(n)}(T)-x_{0}|
		+
		2\int_{0}^{T}
		\left|
		b_{n}(X^{(n)}(\eta_{n}(s)))
		\right|
		\rd s\\
		&\quad
		+
		\left|
			\int_{0}^{T}
				\left\{
					\1_{(x,\infty)}(X^{(n)}(s))-\1_{(-\infty,x)}(X^{(n)}(s))
				\right\}
				\sigma_{n}(X^{(n)}(\eta_{n}(s)))
			\rd B(s)
		\right|.
	\end{align*}
	By using \eqref{tamed_coef_0}, \eqref{tamed_coef_01} and Assumption \ref{Ass_super_0} (ii) and \eqref{local_sigma_1}, it follows from inequality $(a+b+c)^2\leq 3(a^2+b^2+c^2)$  and the $L^2$-isometry that,
	\begin{align*}
		\sup_{n \in \n,~x \in \real}
		\e\left[
			\left|
				L_{T}^{x}(X)
			\right|^2
		\right]
		&\leq
		12K
		\int_{0}^{T}
			\e\left[
				(1+|X^{(n)}(\eta_{n}(s))|^{\ell})
			\right]
		\rd s\\
		&\quad+
		6K_{0}
		\int_{0}^{T}
			\e\left[
				(1+|X^{(n)}(\eta_{n}(s))|^{\ell+2})
			\right]
		\rd s.
	\end{align*}
	Hence by using Lemma \ref{Lem_super_0} with $q=\ell, \ell+2$ we conclude \eqref{Lem_super_2}.
	
	Therefore, by using the same way as the proof of Proposition \ref{main_1}, and applying Lemma \ref{Lem_super_1} with $q=2p_{0}/(\ell+2)$, we conclude the proof.
\end{proof}

\begin{proof}[Proof of Theorem \ref{main_4}]
	We define $\chi_{n}(t):=X(t)-X^{(n)}(t)$, $\beta_{n}(t):=b(X(t))-b_{n}(X^{(n)}(\eta_{n}(t)))$ and $\alpha_{n}(t):=\sigma(X(t))-\sigma_{n}(X^{(n)}(\eta_{n}(t)))$.
	Then by It\^o's formula, for any $p \geq 2$,
	\begin{align}\label{pr_main_4_1}
		|\chi_{n}(t)|^{p}
		&\leq
		\frac{p}{2}
		\int_{0}^{t}
			|\chi_{n}(s)|^{p-2}
			\left\{
				2
				\chi_{n}(s)
				\beta_{n}(s)
				+
				(p-1)
				|\alpha_{n}(s)|^{2}
			\right\}
		\rd s \notag\\
		&\quad+
		p
		\int_{0}^{t}
			|\chi_{n}(s)|^{p-2}
			\chi_{n}(s)
			\alpha_{n}(s)
		\rd B(s).
	\end{align}

	Now we estimate the integrand of the first part of \eqref{pr_main_4_1}.
	Since for any $\varepsilon,a,b>0$, $(a+b)^{2} \leq (1+\varepsilon)a^{2}+(1+1/\varepsilon)b^{2}$ thus by choosing $\varepsilon>0$ as $(1+\varepsilon)(p-1) \leq p_{1}-1$, if follows from Assumption \ref{Ass_super_0} (i) and Young's inequality $ab \leq a^{2}/2+b^{2}/2$ for $a,b>0$ that
	\begin{align*}
		&
		2\chi_{n}(s)
		\beta_{n}(s)
		+
		(p-1)|\alpha_{n}(s)|^{2}
		\\
		&\leq
		2\chi_{n}(s)
		\{
			b(X(s))
			-
			b(X^{(n)}(s))
		\}
		+
		(1+\varepsilon)(p-1)
		\left|
			\sigma(X(s))
			-
			\sigma(X^{(n)}(s))
		\right|^{2}\\
		&
		\quad
		+
		2\chi_{n}(s)
		\{
			b(X^{(n)}(s))
			-
			b_{n}(X^{(n)}(\eta_{n}(s)))
		\}\\
		&
		\quad
		+
		(1+1/\varepsilon)(p-1)
		\left|
			\sigma(X^{(n)}(s))
			-
			\sigma_{n}(X^{(n)}(\eta_{n}(s)))
		\right|^{2}\\
		&\leq
		(2K+1)
		|\chi_{n}(s)|^{2}
		+
		\left|
			b(X^{(n)}(s))
			-
			b_{n}(X^{(n)}(\eta_{n}(s)))
		\right|^{2}\\
		&\quad
		+
		(1+1/\varepsilon)(p-1)
		\left|
			\sigma(X^{(n)}(s))
			-
			\sigma_{n}(X^{(n)}(\eta_{n}(s)))
		\right|^{2}.
	\end{align*}
	Applying Young's inequality $ab\leq a^{q}/q+b^{q'}/q'$ for $a,b>0$ and $q,q'>1$ with $1/q+1/q'=1$, it follows from Assumption \ref{Ass_super_0} (ii), \eqref{local_sigma_01}, \eqref{tamed_coef_1} and \eqref{tamed_coef_2}  that
	\begin{align}\label{pr_main_4_2}
		&|\chi_{n}(s)|^{p-2}
		\left\{
			2\chi_{n}(s)
			\beta_{n}(s)
			+
			(p-1)|\alpha_{n}(s)|^{2}
		\right\}  \notag\\
		&\leq
		C|\chi_{n}(s)|^{p}
		+
		C
		\left|
			b(X^{(n)}(s))
			-
			b_{n}(X^{(n)}(\eta_{n}(s)))
		\right|^{p}\notag\\
		&\quad
		+
		C
		\left|
			\sigma(X^{(n)}(s))
			-
			\sigma_{n}(X^{(n)}(\eta_{n}(s)))
		\right|^{p} \notag\\
		&
		\leq
		C|\chi_{n}(s)|^{p} \notag\\
		&\quad+
		C(1+|X^{(n)}(s)|^{p\ell}
		+
		|X^{(n)}(\eta_{n}(s))|^{p\ell})
		\left|
			f_{b}(X^{(n)}(s))
			-
			f_{b}(X^{(n)}(\eta_{n}(s)))
		\right|^{p\gamma} \notag\\
		&\quad
		+
		C(1+|X^{(n)}(\eta_{n}(s))|^{p(2\ell+1)})\{n^{-p/2}+n^{-p/4}\}.
	\end{align}
	for some $C>0$.
	
	Let $\tau_{N}^{(n)}:=\inf\{t >0~;~|X(t)| \geq N\} \wedge \inf\{t >0~;~|X^{(n)}(t)| \geq N\}$
	Then by Lemma \ref{Lem_super_0}, $\tau_{N}^{(n)} \to \infty$ as $N \to \infty$ a.s.
	By \eqref{pr_main_4_1}, \eqref{pr_main_4_2},  H\"older's inequality, Jensen's inequality, Lemma \ref{Lem_super_0} with $q=2(2\ell+1)$ and Proposition \ref{Lem_super_3} with $q=p(2\ell+1)/(\ell+1)$, we have
	\begin{align*}
		\e\left[
			|\chi_{n}(t\wedge \tau_{N}^{(n)})|^{p}
		\right]
		&\leq
		C\int_{0}^{t}
			\e\left[
				\left|
					\chi_{n}(s \wedge \tau_{N}^{(n)})
				\right|^{p}
			\right]
		\rd s
		+
		Cn^{-p/4}
		\\
		&\quad
		+
		C\left(
			1
			+
			\sup_{0\leq t \leq T}
				\e\left[
					\left|
						X^{(n)}(s)
					\right|^{p(2\ell+1)}
				\right]^{\frac{\ell}{2\ell+1}}
		\right)\\
		&\quad\quad \times
		\left(
			\int_{0}^{T}
				\e\left[
					\left|
						f_{b}(X^{(n)}(s))
						-
						f_{b}(X^{(n)}(\eta_{n}(s)))
					\right|^{\frac{p(2\ell+1)}{\ell+1}}
				\right]
			\rd s
		\right)^{\frac{\gamma(\ell+1)}{2\ell+1}}\\
		&\leq
		C\int_{0}^{t}
		\e\left[
			\left|\chi_{n}(s\wedge \tau_{N}^{(n)})\right|^{p}
		\right]
		\rd s
		+
		Cn^{-p/4}
		+
		Cn^{-\frac{2\gamma (\ell+1)}{p_{0}+\ell+2} \frac{p_{0}}{2\ell+1}}.
	\end{align*}
	Finally, applying Gronwall's inequality and taking the limit $N \to \infty$, we conclude the proof.
\end{proof}


\subsubsection*{Case 2 : H\"older continuous diffusion coefficient}

In Assumption \ref{Ass_super_0}, we consider Khasminskii and one-sided Lipschitz condition for the coefficients.
We next consider that the coefficients $b:\real \to \real$ and $\sigma: \real \to \real$ satisfies the following conditions.

\begin{Ass}\label{Ass_super_1}
	We suppose that the coefficients $b:\real \to \real$ and $\sigma:\real \to \real$ are measurable and satisfy the following conditions:
	\begin{itemize}
		\item[(i)']
			There exist $K>0$ and $\alpha \in [1/2,1]$ such that for each $x,y \in \real$,
			\begin{align*}
				2 x b(x)
				&\leq
				K(1+|x|^{2}),~
				(x-y) (b(x)-b(y))
				\leq
				K|x-y|^{2}\\
				|\sigma(x)-\sigma(y)|
				&\leq
				K |x-y|^{\alpha}.
			\end{align*}
		\item[(ii)']
		There exist $K>0$, $f_{b} \in BV$, $\gamma \in (0,1]$ and $\ell \in [0,\infty)$ such that for each $x,y \in \real$,
		\begin{align*}
			|b(x)-b(y)|
			&\leq
			K(1+|x|^{\ell}+|y|^{\ell}) |f_{b}(x)-f_{b}(y)|^{\gamma}\\
			|b(x)|
			&\leq
			K(1+|x|^{\ell+1}).
		\end{align*}
		
		\item[(iii)]
		The diffusion coefficient $\sigma$ is uniformly elliptic.
	\end{itemize}
\end{Ass}

\begin{Eg}
	Let $b:\real \to \real$ be decreasing and polynomial growth.
	Then $b$ satisfies Assumption \ref{Ass_super_1} with $f_{b}=b$ and $\gamma=1$.
\end{Eg}

Since the diffusion coefficient is of linear growth from H\"older continuity, and the constant $\ell \in [0,\infty)$, thus as similar way as subsection \ref{sec_3_3}, we consider a tamed Euler--Maruyama scheme for a solution of SDE \eqref{SDE_1}, defined by
\begin{align}\label{tamed_EM_1}
	\rd X^{(n)}(t)
	=
	b_{n}(X^{(n)}(\eta _n(t))) \rd t
	+
	\sigma(X^{(n)}(\eta _n(s))) \rd B(t),~
	X^{(n)}(0)=x_{0},~
	t \in [0,T],
\end{align}
where
\begin{align*}
	b_{n}(x)
	:=
	\left\{ \begin{array}{ll}
	\displaystyle b(x), &\text{ if } ,  \ell=0\\
	\displaystyle \frac{b(x)}{1+n^{-1/2} |x|^{\ell}}, &\text{ if }, \ell \in (0,\infty),
	\end{array}\right.
\end{align*}
that is, if $b$ is of linear growth, $X^{(n)}$ is the standard Euler--Maruyama scheme.

\begin{Rem}\label{Rem_super_0}
	Note that under Assumption \ref{Ass_super_1}, for any $p >0$, there exists $C_{p}>0$ such that
	\begin{align*}
		\sup_{0\leq t \leq T}
		\e
		\left[
			\left|
				X(t)
			\right|^{p}
		\right]
		\vee
		\sup_{n\in\n}
		\sup_{0\leq t \leq T}
		\e
		\left[
			\left|
				X^{(n)}(t)
			\right|^{p}
		\right]
		&\leq
		C_{p}\\
		\sup_{0\leq t \leq T}
		\e\left[
			\left|
				X^{(n)}(t)
				-
				X^{(n)}(\eta_{n}(t))
			\right|^{p}
		\right]
		&\leq
		C_{p}
		n^{-p/2},
	\end{align*}
	(see, Lemma 3.1 and Lemma 3.3 in \cite{Sa13}).
	Moreover, since $\sigma$ is uniformly elliptic, by using the same was as the proof of Proposition \ref{main_1}, it holds that for any $g \in BV$, $p\in(0,\infty)$ and $q \in [1,\infty)$ there exists $C=C(g,b,\sigma,p)>0$ such that
	\begin{align*}
		&\int_{0}^{T}
		\e\left[
			\left|
				g(X^{(n)}(s))
				-
				g(X^{(n)}(\eta_{s}(s)))
			\right|^{q}
		\right]
		\rd s
		\leq
		C
		n^{-\frac{p}{2(p+1)}}.
	\end{align*}
\end{Rem}

Under Assumption \ref{Ass_super_1}, we have the following rate of strong convergence for the tamed Euler--Maruyama scheme \eqref{tamed_EM_1}.

\begin{Thm}\label{main_4_2}
	Suppose that Assumption \ref{Ass_super_0} holds.
	Then for any $\rho \in (0,1)$, there exists $C>0$ such that
	\begin{align*}
		\sup_{0\leq t \leq T}
		\e[|X(t) - X^{(n)}(t)|]
		&\leq
		\left\{ \begin{array}{ll}
			\displaystyle
				C (\log n)^{-1},
				&\text{ if } \alpha=1/2,  \\
			\displaystyle
				C n^{-r(\alpha,\gamma,\rho,1)},
				&\text{ if } \alpha \in (1/2,1],
		\end{array}\right.\\
		\e\left[
			\sup_{0\leq t \leq T}
			\left|
				X(t)
				-
				X^{(n)}(t)
			\right|
		\right]
		&\leq
		\left\{ \begin{array}{ll}
			\displaystyle
				C (\log n)^{-1/2},
				&\text{ if } \alpha=1/2,  \\
			\displaystyle
				C n^{-r(\alpha,\gamma,\rho,1)(2\alpha-1)},
				&\text{ if } \alpha \in (1/2,1],
		\end{array}\right.
	\end{align*}
	and for any $p \geq 2$, there exists $C_{p}>0$ such that
	\begin{align*}
		\e\left[
			\sup_{0\leq t \leq T}
			\left|
				X(t)
				-
				X^{(n)}(t)
			\right|^{p}
		\right]
		&\leq
		\left\{ \begin{array}{ll}
			\displaystyle
				C_{p}(\log n)^{-1},
				&\text{ if } \alpha=1/2,  \\
			\displaystyle
				C_{p} n^{-r(\alpha,\gamma,\rho,1)},
				&\text{ if } \alpha \in (1/2,1),\\
			\displaystyle
				C_{p} n^{-r(1,\gamma,\rho,p)},
				&\text{ if } \alpha =1,
		\end{array}\right.
	\end{align*}
	where $r(\alpha,\gamma,\rho,p):=\min\{ \gamma(1-\rho)/2, p(2\alpha-1)/2\}$.
\end{Thm}

\begin{Rem}
	Note that it is proved in Theorem 2.10 and Proposition 2.5 of \cite{NgTa16} that if $\ell=0$ and $b$ and $\sigma$ are bounded and the number of discontinuous points of $b$ are countable, then by using Gaussian upper bound for the density of the Euler--Maruyama scheme, it holds that
	\begin{align*}
		\sup_{0\leq t \leq T}
		\e[|X(t) - X^{(n)}(t)|]
		\leq
		\left\{ \begin{array}{ll}
		\displaystyle
		C (\log n)^{-1},
		&\text{ if } \alpha=1/2,  \\
		\displaystyle
		C n^{-\alpha/2},
		&\text{ if } \alpha \in (1/2,1],
		\end{array}\right.
	\end{align*}
	(see, also Theorem 2.11 and 2.12 in \cite{NgTa16} for $L^{p}$-sup estimates) and proved in Theorem 1.1 of \cite{HuJeKl12} and Corollary 2.3 of \cite{Sa13} that if $\ell >0$, $f_{b}(x)=x$, $\gamma=1$ (that is, $b$ is locally Lipschitz continuous) and $\sigma$ is globally Lipschitz continuous (in this setting, we do not need to assume uniformly elliptic condition on $\sigma$), then for any $p\in(0,\infty)$, it holds that
	\begin{align*}
		\e\left[
			\sup_{0\leq t \leq T}
			\left|
				X(t)
				-
				X^{(n)}(t)
			\right|^{p}
		\right]^{1/p}
		\leq
		Cn^{-1/2}.
	\end{align*}
	Therefore, the statements of Theorem \ref{main_4_2} are generalization of these error estimates for unbounded and irregular coefficients.
\end{Rem}

\begin{proof}[Proof of Theorem \ref{main_4_2}]
	In order to deal with H\"older continuity of $\sigma$, we apply Yamada and Watanabe approximation technique which is used in the proof of Theorem \ref{main_3}.
	By using It\^o's formula, $\phi_{\delta,\varepsilon}(X(t)-X^{(n)}(t))$ can be decomposed by the following five terms
	\begin{align*}
		\phi_{\delta,\varepsilon}(X(t)-X^{(n)}(t))
		=
		M^{n,\delta,\varepsilon}(t)
		+
		I_{1}^{n,\delta,\varepsilon}(t)
		+
		I_{2}^{n,\delta,\varepsilon}(t)
		+
		I_{3}^{n,\delta,\varepsilon}(t)
		+
		J^{n,\delta,\varepsilon}(t),
	\end{align*}
	where
	\begin{align*}
		M^{n,\delta,\varepsilon}(t)
		&:=
		\int_{0}^{t}
			\phi{'}_{\delta,\varepsilon}(X(s)-X^{(n)}(s))
			\left\{
				\sigma(X(s))
				-
				\sigma(X^{(n)}(\eta_n(s)))
			\right\}
		\rd B(s),\\
		I_{1}^{n,\delta,\varepsilon}(t)
		&:=
		\int_{0}^{t}
			\phi{'}_{\delta,\varepsilon}(X(s)-X^{(n)}(s))
			\left\{
				b(X(s))
				-
				b(X^{(n)}(s))
			\right\}
		\rd s,\\
		I_{2}^{n,\delta,\varepsilon}(t)
		&:=
		\int_{0}^{t}
			\phi{'}_{\delta,\varepsilon}(X(s)-X^{(n)}(s))
			\left\{
				b(X^{(n)}(s))
				-
				b(X^{(n)}(\eta_n(s)))
			\right\}
		\rd s,\\
		I_{3}^{n,\delta,\varepsilon}(t)
		&:=
		\int_{0}^{t}
			\phi{'}_{\delta,\varepsilon}(X(s)-X^{(n)}(s))
			\left\{
				b(X^{(n)}(\eta_n(s)))
				-
				b_{n}(X^{(n)}(\eta_n(s)))
			\right\}
		\rd s,\\
		J^{n,\delta,\varepsilon}(t)
		&:=
		\frac{1}{2}
		\int_{0}^{t}
			\phi^{''}_{\delta,\varepsilon}(X(s)-X^{(n)}(s))
			\left|
				\sigma(X(s))
				-
				\sigma(X^{(n)}(\eta_n(s)))
			\right|^{2}
		\rd s.
	\end{align*}
	Note that if $\ell=0$ then $I_{3}^{n,\delta,\varepsilon}(t)=0$, and since $\phi{'}$ is bounded and $\sigma$ is of linear growth, $(M^{n,\delta,\varepsilon}(t))_{0 \leq t \leq T}$ is martingale, so expectation of $M^{n,\delta,\varepsilon}(t)$ equals to zero.
	
	To conclude the statement, we estimate the expectation of $I_{k}^{n,\delta,\varepsilon}(t)$, $k=1,2,3$ and $J^{n,\delta,\varepsilon}(t)$.
	
	We first consider the expectation of $I_{k}^{n,\delta,\varepsilon}(t)$, $k=1,2,3$.
	Since the drift coefficient $b$ is one-sided Lipschitz,  by the same way as the estimation of \eqref{pr_0}, it holds that
	\begin{align}\label{pr_421}
		\e\left[
			\left|
				I_{1}^{n,\delta,\varepsilon}(t)
			\right|
		\right]
		\leq
		K \int_{0}^{t}
			\e\left[
				\left|
					X(s)
					-
					X^{(n)}(s)
				\right|
			\right]
		\rd s.
	\end{align}
	By using property (ii) of $\phi_{\delta, \varepsilon}$, Assumption \ref{Ass_super_1} on $b$ and Remark \ref{Rem_super_0}, we have for any $q > 1$,
	\begin{align}\label{pr_422}
		&\e\left[
			\left|
				I_{2}^{n,\delta,\varepsilon}(t)
			\right|
		\right]\notag\\
		&\leq
		\int_{0}^{T}
			\e\Big[
				(1+|X^{(n)}(s)|^{\ell}+|X^{(n)}(\eta_n(s))|^{\ell}) 
				\left|
					f_{b}(X^{(n)}(s))
					-
					f_{b}(X^{(n)}(\eta_{n}(s)))
				\right|^{\gamma}
			\Big]
		\rd s\notag\\
		&\leq
		C
		\left(
			1
			+
			\sup_{0\leq t \leq T}
			\e\left[
				\left|
					|X^{(n)}(t)|^{\frac{\ell q}{q-1}}
				\right|
			\right]^{\frac{q-1}{q}}
		\right) 
		\left(
			\int_{0}^{T}
			\e\left[
			\left|
			f_{b}(X^{(n)}(s))
			-
			f_{b}(X^{(n)}(\eta_{n}(s)))
			\right|^{q}
			\right]
			\rd s
		\right)^{\frac{\gamma}{q}}\notag\\
		&\leq
		C(f_{b},b,\sigma,p,T)
		n^{-\frac{\gamma p}{2q(p+1)}}.
	\end{align}
	Finally, since $2\ell+1 \leq p_{0}$, by using \eqref{tamed_coef_1}, there exists $C>0$ such that
	\begin{align}\label{pr_423}
		\e\left[
			\left|
				I_{3}^{n,\delta,\varepsilon}(t)
			\right|
		\right]
		&\leq
		Cn^{-1/2}.
	\end{align}
	
	Next, we consider the expectation of $J^{n,\delta,\varepsilon}(t)$.
	From the property (iii) of $\phi_{\delta, \varepsilon}$, we have
	\begin{align*}
		J^{n,\delta,\varepsilon}(t)
		&\leq
		\int_{0}^{T}
			\frac
			{\1_{[\varepsilon/\delta,\varepsilon]}(|X(s)-X^{(n)}(s)|)}
			{|X(s)-X^{(n)}(s)| \log \delta}
			\left|
				\sigma(X(s))
				-
				\sigma(X^{(n)}(\eta_n(s)))
			\right|^{2}
		\rd s\\
		&\leq
		2
		\{
			J_{1}^{n,\delta,\varepsilon}(T)
			+
			J_{2}^{n,\delta,\varepsilon}(T)
		\},
	\end{align*}
	where
	\begin{align*}
		J_{1}^{n,\delta,\varepsilon}(t)
		&:=
		\int_{0}^{t}
			\frac
				{\1_{[\varepsilon/\delta,\varepsilon]}(|X(s)-X^{(n)}(s)|)}
				{|X(s)-X^{(n)}(s)| \log \delta}
			\left|
				\sigma(X(s))
				-
				\sigma(X^{(n)}(s))
			\right|^{2}
		\rd s, \\
		J_{2}^{n,\delta,\varepsilon}(t)
		&:=
		\int_{0}^{t}
		\frac
		{\1_{[\varepsilon/\delta,\varepsilon]}(|X(s)-X^{(n)}(s)|)}
		{|X(s)-X^{(n)}(s)| \log \delta}
		\left|
			\sigma(X^{(n)}(s))
			-
			\sigma(X^{(n)}(\eta_n(s)))
		\right|^{2}
		\rd s.
	\end{align*}
	
	We first consider $J_{1}^{n,\delta,\varepsilon}(T)$.
	Since $\sigma$ is $\alpha$-H\"older continuous, we have
	\begin{align}\label{esti_J1_super}
		J_{1}^{n,\delta,\varepsilon}(T)
		&\leq
		\frac{K}{\log \delta}
		\int_{0}^{T}
			\frac
			{\1_{[\varepsilon/\delta,\varepsilon]}(|X(s)-X^{(n)}(s)|)}
			{|X(s)-X^{(n)}(s)|}
			\left|
				X(s)
				-
				X^{(n)}(s)
			\right|^{2\alpha}
		\rd s\notag\\
		&\leq
		\frac{K \varepsilon^{2\alpha-1}}{\log \delta}.
	\end{align}
	Next we consider $J_{2}^{n,\delta,\varepsilon}(T)$.
	By using the property (iii) of $\phi_{\delta, \varepsilon}$, H\"older continuity of $\sigma$ and Remark \ref{Rem_super_0}, we have
	\begin{align}\label{esti_J2_super}
		\e\left[
			\left|
				J_{2}^{n,\delta,\varepsilon}(t)
			\right|
		\right]
		&\leq
		\frac
			{K \delta}
			{\varepsilon \log \delta}
		\int_{0}^{t}
			\e\left[
				\left|
					X^{(n)}(s)
					-
					X^{(n)}(\eta_n(s))
				\right|^{2\alpha}
			\right]
		\rd s
		\leq
		C
		\frac
		{\delta}
		{\varepsilon \log \delta}
		\frac{1}{n^{\alpha}},
	\end{align}
	for some $C>0$.
	
	Since $\e [M^{n,\delta, \varepsilon}(t)] = 0$, it follows from \eqref{pr_421}, \eqref{pr_422}, \eqref{pr_423},  \eqref{esti_J1_super} and \eqref{esti_J2_super} that there exists a positive constant $C$ which do not depend on $n$ such that
	\begin{align*}
		\e\left[\left|X(t)-X^{(n)}(t)\right|\right]
		&\leq
		\varepsilon
		+
		K
		\int_{0}^{t}
			\e\left[
				\left|
					X(s)
					-
					X^{(n)}(s)
				\right|
			\right]
		\rd s\\
		&\quad+
		C
		\left\{
			\frac{1}{n^{\frac{\gamma p}{2q(p+1)}}}
			+
			\frac{1}{n^{1/2}}
			+
			\frac{\varepsilon^{2\alpha-1}}{\log \delta}
			+
			\frac
			{\delta}
			{\varepsilon \log \delta}
			\frac{1}{n^{\alpha}}
		\right\}.
	\end{align*}
	For $\alpha=1/2$ we choose $\varepsilon = 1/\log n$ and $\delta = n^{\alpha/2}$, and for $\alpha \in (1/2,1]$ we choose $\varepsilon:=n^{-1/2}$ and $\delta=2$, then by using Gronwall's inequality, since $p>0$, $q>1$ are arbitrarily, we conclude the proof of the first statement.
	
	Now consider the proof of the second statement.
	Let $V^{(n)}(t):=\sup_{0\leq s \leq t} |X(s)-X^{(n)}(s)|$.
	Then from the above computations, it holds that for any $t \in [0,T]$,
	\begin{align*}
		\e\left[V^{(n)}(t)\right]
		&
		\leq
		\varepsilon
		+
		K
		\int_{0}^{t}
			\e\left[
				V^{(n)}(s)
			\right]
		\rd s
		+
		\e\left[
			\sup_{0\leq s \leq t}
			\left|
				M^{n,\delta,\varepsilon}(s)
			\right|^{p}
		\right]\\
		&\quad+
		C
		\left\{
			\frac{1}{n^{\frac{\gamma p}{2q(p+1)}}}
			+
			\frac{1}{n^{1/2}}
			+
			\frac{\varepsilon^{2\alpha-1}}{\log \delta}
			+
			\frac
			{\delta}
			{\varepsilon \log \delta}
			\frac{1}{n^{\alpha}}
		\right\}.
	\end{align*}
	By using Burkholder-Davis-Gundy's inequality, Young's inequality, Jensen's inequality and using the first statement, we have there exist $C_{1}, C_{2}>0$ such that
	\begin{align*}
		\e\left[
			\sup_{0\leq s \leq t}
			\left|
				M^{n,\delta,\varepsilon}_{s}
			\right|
		\right]
		&\leq
		C_{1}
		\e\left[
			\left(
				\int_{0}^{t}
					\left|
						X(s)
						-
						X^{(n)}(s)
					\right|^{2\alpha}
				\rd s
			\right)^{1/2}
		\right]
		+
		C_{1}n^{-\alpha/2}\\
		&\leq
		\left\{ \begin{array}{ll}
			\displaystyle
				C_{2} (\log n)^{-1/2}
				&\text{ if } ,  \alpha=1/2\\
			\displaystyle
				\frac{1}{2}\e\left[V^{(n)}(t)\right]
				+
				C_{2}
				n^{-r(\alpha,\gamma,\rho,1)(2\alpha-1)}
				+
				C_{1}n^{-\alpha/2}
				&\text{ if } , \alpha \in (1/2,1].
		\end{array}\right.
	\end{align*}
	This concludes the second estimate.
	
	Finally, we consider the proof of the third statement.
	By the similar way as the proof of the first estimate, it holds that for any $p \geq 2$, there exists $C_{3}>0$ such that
	\begin{align*}
		&\e\left[V^{(n)}(t)^{p}\right]
		\leq
		C_{3}
		\left\{
			\varepsilon^{p}
			+
			\e\left[
				\left(
					\int_{0}^{t}
						V^{(n)}(s)
					\rd s
				\right)^{p}
			\right]
			+
			\e\left[
				\sup_{0\leq s \leq t}
				\left|
					M^{n,\delta,\varepsilon}(s)
				\right|^{p}
			\right]
		\right\}
		\\
		&\quad+
		C_{3}
		\left\{
			\frac{1}{n^{p/2}}
			+
			\left(
				\frac{\varepsilon^{2\alpha-1}}{\log \delta}
			\right)^{p}
			+
			\left(
				\frac
				{\delta}
				{\varepsilon \log \delta}
			\right)^{p}
			\frac{1}{n^{p\alpha}}
		\right\}\\
		&\quad+
		C_{3}
		\int_{0}^{T}
			\e\left[
				(1+|X^{(n)}(s)|^{p\ell}+|X^{(n)}(\eta_n(s))|^{p\ell})
				\left|
					f_{b}(X^{(n)}(s))
					-
					f_{b}(X^{(n)}(\eta_{n}(s)))
				\right|^{p\gamma}
			\right]
		\rd s.
	\end{align*}
	By using Burkholder-Davis-Gundy's inequality, there exists $C_{4}>0$ such that
	\begin{align*}
		\e\left[
			\sup_{0\leq s \leq t}
				\left|
					M^{n,\delta,\varepsilon}(s)
				\right|^{p}
		\right]
		&\leq
		C_{4}
		\e\left[
			\left(
				\int_{0}^{t}
					\left|
						X(s)
						-
						X^{(n)}(s)
					\right|^{2\alpha}
				\rd s
				\right)^{p/2}
		\right]
		+
		C_{4}n^{-\alpha p/2}
	\end{align*}
	and by using H\"older's inequality and Remark \ref{Rem_super_0}, for any $\widehat{p}>0$ and $q>1$,
	\begin{align*}
		&C_{3}
		\int_{0}^{T}
			\e\left[
				(1+|X^{(n)}(s)|^{p\ell}+|X^{(n)}(\eta_n(s))|^{p\ell})
				\left|
					f_{b}(X^{(n)}(s))
					-
					f_{b}(X^{(n)}(\eta_{n}(s)))
				\right|^{p\gamma}
			\right]
		\rd s\\
		&\leq
		C_{4}
		n^{-\frac{\widehat{p} \gamma }{2q(\widehat{p}+1)}}.
	\end{align*}
	
	For $\alpha=1$, then we  can use Groanwall's inequality, we obtain
	\begin{align*}
		\e\left[V^{(n)}(t)^{p}\right]
		\leq
		C_{5}
		\left\{
			\varepsilon^{p}
			+
			\frac{1}{n^{p/2}}
			+
			\left(
				\frac{\varepsilon}{\log \delta}
			\right)^{p}
			+
			\left(
				\frac
				{\delta}
				{\varepsilon \log \delta}
			\right)^{p}
			\frac{1}{n^{p}}
			+
			n^{-p/2}
			+
			n^{-\frac{\widehat{p} \gamma }{2q(\widehat{p}+1)}}
		\right\}.
	\end{align*}
	for some $C_{5}$.
	By choosing $\varepsilon=n^{-1/2}$ and $\delta=2$, we conclude the statement for $\alpha=1$.
	
	For $\alpha \in (1/2,1)$, by choosing $\varepsilon=n^{-1/2}$ and $\delta =2$, we have
	\begin{align*}
		\e\left[V^{(n)}(t)^{p}\right]
		&\leq
		C_{6}
		\left\{
			n^{-\frac{p(2\alpha-1)}{2}}
			+
			n^{-\frac{\widehat{p} \gamma }{2q(\widehat{p}+1)}}
			+
			\e\left[
				\left(
					\int_{0}^{t}
						V^{(n)}(s)
					\rd s
				\right)^{p}
			\right]
		\right\}\\
		&\quad+
			C_{6}
			\e\left[
				\left(
					\int_{0}^{t}
						\left|
							X(s)
							-
							X^{(n)}(s)
						\right|^{2\alpha}
					\rd s
				\right)^{p/2}
			\right],
	\end{align*}
	for some $C_{6}$.
	By using Lemma 3.2 (ii) in \cite{GyRa11} with $\rho=2\alpha$, $q=2$, and $\delta=C_{6}n^{-\frac{p(2\alpha-1)}{2}}
	+
	n^{-\frac{\widehat{p} \gamma }{2q(\widehat{p}+1)}}$, we obtain
	\begin{align*}
		\e\left[V^{(n)}(t)^{p}\right]
		\leq
		C_{7}
		\left\{
			n^{-p(2\alpha-1)/2}
			+
			n^{-\frac{\widehat{p} \gamma }{2q(\widehat{p}+1)}}
			+
			\e\left[
				\int_{0}^{t}
					\left|
						X(s)
						-
						X^{(n)}(s)
					\right|
				\rd s
			\right]
		\right\},
	\end{align*}
	which concludes the statement for $\alpha \in (1/2,1)$.
	
	For $\alpha =1/2$, by choosing $\varepsilon=(\log n)^{-1}$ and $\delta =n^{-1/3}$, we have
	\begin{align*}
		\e\left[V^{(n)}(t)^{p}\right]
		\leq
		C_{8}
		\Bigg\{
			(\log n)^{-p}
			+
			\e\left[
				\left(
					\int_{0}^{t}
						V^{(n)}(s)
					\rd s
				\right)^{p}
			\right]\\
			+
			\e\left[
				\left(
					\int_{0}^{t}
						\left|
							X(s)
							-
							X^{(n)}(s)
						\right|
					\rd s
				\right)^{p/2}
			\right]
		\Bigg\},
	\end{align*}
	for some $C_{8}>0$.
	By using Lemma 3.2 (ii) in \cite{GyRa11} with $\rho=1$, $q=2$, and $\delta=C_{8}(\log n)^{-p}$, we obtain
	\begin{align*}
		\e\left[V^{(n)}(t)^{p}\right]
		\leq
		C_{9}
		\left\{
			(\log n)^{-p}
			+
			\e\left[
				\int_{0}^{t}
					\left|
						X(s)
						-
						X^{(n)}(s)
					\right|
				\rd s
			\right]
		\right\},
	\end{align*}
	which concludes the statement for $\alpha =1/2$.
\end{proof}

\subsection{Approximation of integral type functionals of SDEs}\label{sec_3_4}

Let us consider the following (decoupled) system of two-dimensional SDEs of the form
\begin{align}\label{SDE_4}
	\begin{split}
	\rd X(t)
	&=
	b(t,X(t)) \rd t
	+
	\sigma(t,X(t))\rd B(t),\\
	\rd Y(t)
	&=
	\mu(t,X(t),Y(t)) \rd t
	+
	\rho_{1}(t,X(t),Y(t))\rd B(t)
	+
	\rho_{2}(t,X(t),Y(t))\rd W(t),\\
	(X(0),Y(0))^{\top}
	&=(x_{0},y_{0})^{\top}
	\in \real^{2},
	~t \in [0,T],
	\end{split}
\end{align}
where $(B,W)^{\top}$ is a two-dimensional standard Brownian motion, and consider its Euler--Maruyama scheme of the form
\begin{align*}
	\begin{split}
	\rd X^{(n)}(t)
	&=
	b(\eta _n(t), X^{(n)}(\eta _n(t))) \rd t
	+
	\sigma(\eta _n(t), X^{(n)}(\eta _n(s))) \rd B(t),~,\\
	\rd Y^{(n)}(t)
	&=
	\mu(\eta _n(t),X^{(n)}(\eta _n(t)),Y^{(n)}(\eta _n(t))) \rd t\\
	&\quad+
	\rho_{1}(\eta _n(t),X^{(n)}(\eta _n(t)),Y^{(n)}(\eta _n(t)))\rd B(t)\\
	&\quad
	+
	\rho_{2}(\eta _n(t),X^{(n)}(\eta _n(t)),Y^{(n)}(\eta _n(t)))\rd W(t),\\
	(X^{(n)}(0),Y^{(n)}(0))^{\top}
	&=(x_{0},y_{0})^{\top}
	\in \real^{2},
	~t \in [0,T].
	\end{split}
\end{align*}
The process $Y$ in \eqref{SDE_4} is related to Asian type options in mathematical finance and the observation process in filtering problems (see, chapter 6 in \cite{Ok98}).

\begin{Thm}\label{main_9}
	Suppose that the system of equation \eqref{SDE_4} has unique strong solution, the coefficients of $X$ are bounded and $\sigma$ is uniformly elliptic, and the coefficients of $Y$ satisfies the following conditions;
	there exist $K\geq 0$, $\beta, \gamma \in (0,1]$ and $g \in BV$ such that for any $t,s \in [0,T]$, $(x,y)^{\top}, (x',y')^{\top} \in \real^{2}$ and $h\in\{\mu,\rho_{1},\rho_{2}\}$,
	\begin{align*}
		|h(t,x,y)-h(s,x',y')|
		&\leq
		|g(x)-g(x')|^{\gamma}
		+
		K|y-y'|
		+
		K|t-s|^{\beta}\\
		|h(t,x,y)|
		&\leq
		K(1+|x|+|y|).
	\end{align*}
	Suppose $\sup_{0 \leq t \leq T} \e[|X(t)-X^{(n)}(t)|^{\widehat{p}}]^{1/\widehat{p}} \leq \mathrm{Err}_{\widehat{p}}(n)$ holds for some $\widehat{p} \in (0,\infty)$.
	Then for any $\alpha$-H\"older continuous function $f:\real \to \real$ and $p \geq 2/\alpha$, there exists a positive constant $C>0$ such that for any $n \in \n$,
	\begin{align}\label{int_0}
		\e\left[
			\sup_{0 \leq t \leq T}
			\left|
				f(Y(t))
				-
				f(Y^{(n)}(t))
			\right|^{p}
		\right]
		&\leq
		C
		\mathrm{Err}_{\widehat{p}}(n)^{\frac{\widehat{p}}{\widehat{p}+1}}
		+
		C
		n^{-p\alpha \gamma/2}
		+
		Cn^{-p\alpha \beta/2}.
	\end{align}
\end{Thm}

\begin{Rem}
	It is known (Theorem 3 and Theorem 4 in \cite{KoMa13}) that if the coefficients of $(X,Y)^{\top}$ are time independent, smooth, bounded with bounded derivatives, and $\sigma$ is uniformly elliptic, $\rho_{1}$ satisfies the uniform H\"ormander condition of order $k$, and $\rho_{2}=0$, then $Y(t)$, for $t>0$, admits the smooth density with respect to Lebesgue measure, and it satisfies some Gaussian type two sided bound (see also \cite{KoTa12, LeMe10}).
	Therefore, in this case we can use original Avikainen's estimate \eqref{Av_0}.
	And the estimate \eqref{int_0} can be applied to multilevel Monte Carlo methods (see, \cite{Gi08}).
\end{Rem}

\begin{proof}[Proof of Theorem \ref{main_9}]
	Let $Z=(X,Y)^{\top}$ and $Z^{(n)}=(X^{(n)},Y^{(n)})^{\top}$.
	By using Burkholder-Davis-Gundy's inequality, it holds that for any $p \geq 2/\alpha$,
	\begin{align*}
		&
		\e\left[
			\sup_{0 \leq u \leq t}
			\left|
				Y(u)
				-
				Y^{(n)}(u)
			\right|^{p\alpha}
		\right]\\
		&\leq
		3^{q\alpha-1}T^{p\alpha-1}
		\int_{0}^{t}
			\e\left[
				\left|
					\mu(s,Z(s))
					-
					\mu(\eta_{n}(s),Z^{(n)}(\eta_{n}(s)))
				\right|^{p\alpha}
			\right]
		\rd s\\
		&\quad+
		3^{q\alpha-1}c_{p\alpha}T^{p\alpha/2-1}
		\sum_{i=1}^{2}
			\int_{0}^{t}
				\e\left[
					\left|
						\rho_{i}(s,Z(s))
						-
						\rho_{i}(\eta_{n}(s),Z^{(n)}(\eta_{n}(s)))
					\right|^{p\alpha}
				\right]
			\rd s\\
		&\leq
		C
		\int_{0}^{T}
			\e\left[
				\left|
					g(X(s))
					-
					g(X^{(n)}(\eta_{n}(s)))
				\right|^{p\alpha \gamma}
			\right]
		\rd s
		+
		C
		\int_{0}^{t}
			\e
			\left[
				\sup_{0 \leq u \leq s}
				\left|
					Y(u)
					-
					Y^{(u)}(\eta_{n}(s))
				\right|^{p\alpha}
			\right]
		\rd s\\
		&\quad
		+
		C
		\int_{0}^{T}
			\e
			\left[
				\left|
					Y^{(n)}(s)
					-
					Y^{(n)}(\eta_{s}(s))
				\right|^{p\alpha}
			\right]
		\rd s
		+
		Cn^{-p\alpha \beta},
	\end{align*}
	for some $C$, and $c_{p \gamma}$ is the constant of Burkholder-Davis-Gundy's inequality.
	Since $b$ and $\sigma$ are bounded, and $\sigma$ is uniformly elliptic, by using Proposition \ref{main_2}, for any $\widehat{p} \in (0,\infty)$, there exists $C_{\widehat{p}}$ such that
	\begin{align*}
		\int_{0}^{T}
			\e\left[
				\left|
					g(X(s))
					-
					g(X^{(n)}(\eta_{n}(s)))
				\right|^{p\alpha \gamma}
			\right]
		\rd s
		\leq
		C_{\widehat{p}}
		\mathrm{Err}_{\widehat{p}}(n)^{\frac{\widehat{p}}{\widehat{p}+1}}
		+
		C_{\widehat{p}}
		n^{-p\alpha \gamma/2}.
	\end{align*}
	On the other hand, since $\mu$, $\rho_{1}$ and $\rho_{2}$ are of linear growth, thus it can be shown that
	\begin{align*}
		\sup_{0 \leq t \leq T}
		\e
		\left[
			\left|
				Y^{(n)}(t)
				-
				Y^{(n)}(\eta_{s}(t))
			\right|^{p\alpha}
		\right]
		\leq
		Cn^{-p\alpha/2},
	\end{align*}
	for some $C>0$.	
	Therefore, by using Gronwall's inequality and then using H\"older continuity of $f$, we conclude the proof.
\end{proof}

\subsection{SDEs driven by symmetric $\alpha$-stable with bounded $\alpha$-variation coefficient}\label{sec_3_5}

Let us consider the following one-dimensional SDEs of the form
\begin{align}\label{SDE_2}
	\rd X(t)
	=
	\sigma(X(t-))\rd Z(t),
	~X(0)=x_0 \in \real,
	~t \in [0,T],
\end{align}
where $Z=(Z(t))_{0\leq t \leq T}$ is a symmetric $\alpha$-stable process with $\alpha \in (1,2)$ defined on a probability space $(\Omega, \mathscr{F},\p)$ with a filtration $(\mathscr{F}_t)_{0\leq t \leq T}$ satisfying the usual conditions, that is, $Z$ is a L\'evy process with characteristic function of the form
\begin{align*}
	\e\left[
		\exp
		\left(
			\sqrt{-1}
			\xi
			Z(t)
		\right)
	\right]
	=
	\exp
	\left(
		-t|\xi|^{\alpha}
	\right),~
	\xi \in \real,~
	t \geq 0.
\end{align*}
We consider the Euler--Maruyama scheme for \eqref{SDE_2} which is given by 
\begin{align*}
	\rd X^{(n)}(t)
	=
	\sigma(X^{(n)}(\eta _n(t))) \rd Z(t),~
	X^{(n)}(0)=x_{0},~
	t \in [0,T].
\end{align*}

\begin{Thm} \label{main_5}
	Suppose that the diffusion coefficient $\sigma$ is measurable, bounded and uniformly positive.
	Moreover, assume that there exists bounded and strictly increasing function and $f_{\sigma}$ such that for any $x,y \in \real$,
	\begin{align*}
		|\sigma(x)-\sigma(y)|^{\alpha}
		\leq
		|f_{\sigma}(x)-f_{\sigma}(y)|.
	\end{align*}
	Then there exists a constant $C$ such that for all $n \geq 2$,
	\begin{align*}
		\sup_{0\leq t \leq T}
		\e[|X(t) - X^{(n)}(t)|^{\alpha-1}]
		\leq
		\frac{C}{(\log n)^{\alpha-1}}.
	\end{align*}
\end{Thm}

Before proving, we consider an analogue of Proposition \ref{main_1} for the following c\`adl\`ag process $Y$ defined by
\begin{align*}
	\rd Y(t)
	=
	\sigma(t-,\omega)
	\rd Z(t),~
	Y(0)=y_{0} \in \real,~
	t \in [0,T].
\end{align*}

\begin{Prop}\label{main_6}
	The diffusion coefficient $\sigma$ of $Y$ is uniformly bounded and $\sigma(Z(s)~;~s\leq t)$-adapted process, and $|\sigma|^{\alpha}$ is uniformly positive, that is, there exists $\underline{a}>0$ such that $|\sigma(t,\omega)|^{\alpha} \geq \underline{a}$ for all $t \geq 0$ almost surely.
	Let $\widehat{Y}=(\widehat{Y}(t))_{t \geq 0}$ be a one-dimensional progressively measurable process.
	Then for any $g \in BV$, $p \in (0,\infty)$ and $q \in [1,\infty)$, there exists $C=C(g,\sigma,p,q)>0$ such that
	\begin{align*}
		\int_{0}^{T}
			\e\left[
				\left|
					g(Y(s))
					-
					g(\widehat{Y}(s))
				\right|^{q}
			\right]
		\rd s
		\leq
		C
		\left(
		\int_{0}^{T}
			\e
			\left[
				\left|
					Y(s)
					-
					\widehat{Y}(s)
				\right|^p
			\right]
		\rd s
		\right)^{\frac{1}{p+1}}.
	\end{align*}
\end{Prop}
\begin{proof}
	It is suffices to estimate $\int_{0}^{T} \p(a<Y(s)\leq b)\rd s$, for $a,b \in \real$ with $a<b$.
	Note that $\sigma(s,\omega)$ is an $\sigma(Z(s)~;~s\leq t)$-adapted process with $\e[\int_{0}^{t}|\sigma(s,\omega)|^{\alpha} \rd s]<\infty$ since $\sigma$ is bounded, thus there exists an $\alpha$-stable process $\widetilde{Z}$ defined on an extended probability space such that $Y(t)=y_{0}+\widetilde{Z}(A(t))$ for $t \geq 0$ where $A(t):=\int_{0}^{t}|\sigma(s,\omega)|^{\alpha} \rd s$.
	Moreover, since $\sigma$ is bounded, it holds that $A(t) \leq \|\sigma\|_{\infty}t$.
	Therefore, by the change of variable $s=\tau_{u}:=\inf\{t \geq 0~;~A(t) > u\}$, we have
	\begin{align*}
		\int_{0}^{T}
			\1_{(a,b]}(Y(s))
		\rd s
		&=
		\int_{0}^{T}
			\1_{(a,b]}(Y(s))
			|\sigma(s,\omega)|^{-\alpha}
		\rd A(s)
		\\
		&\leq
		\underline{a}^{-1}
		\int_{0}^{T}
			\1_{(y_{0}+a,y_{0}+b]}(\widetilde{Z}(A(s)))
		\rd A(s)\\
		&=
		\underline{a}^{-1}
		\int_{0}^{A(T)}
			\1_{(y_{0}+a,y_{0}+b]}(\widetilde{Z}(u))
		\rd u\\
		&\leq
		\underline{a}^{-1}
		\int_{0}^{\|\sigma\|_{\infty}T}
				\1_{(y_{0}+a,y_{0}+b]}(\widetilde{Z}(u))
		\rd u.
	\end{align*}
	Note that the density function of stable process $\widetilde{Z}$, denoted by $p(t,\cdot)$, satisfies the upper bound $p(t,y) \leq C_{0}(t^{-1/\alpha} \wedge t |y|^{-(1+\alpha)})$ for some $C_{0}>0$ (see, e.g., Theorem 2.1 in \cite{BlGe60}), we have
	\begin{align*}
		\int_{0}^{T} \p(a<Y(s)\leq b)\rd s
		\leq
		\underline{a}^{-1}
		\int_{y_{0}+a}^{y_{0}+b} \rd y
			\int_{0}^{\|\sigma\|_{\infty}T}	\rd u
				p(u,y)
		\leq
		C_{1}(b-a),
	\end{align*}
	for some $C_{1}>0$, which concludes the proof.
\end{proof}

\begin{proof}[Proof of Theorem \ref{main_5}]
	In order to deal with the assumption on the diffusion coefficient $\sigma$, similar to the proof of Proposition \ref{main_1}, we use Komatsu's approximation technique (see \cite{Ko82} or \cite{TsHa13}).
	Let $\psi _{\delta, \varepsilon}: \real \to [0,\infty)$ be a continuous function defined on the proof of Theorem \ref{main_3} for $\delta \in (1,\infty)$ and $\varepsilon \in (0,1)$.
	Set $u(x):=|x|^{\alpha-1}$ and $u_{\delta,\varepsilon}:=u*\psi_{\delta, \varepsilon}$, where $*$ is the convolution.
	Since $u$ is $(\alpha-1)$-H\"older continuous, it holds that
	\begin{align}\label{stable_1}
		u(x)
		\leq
		\varepsilon^{\alpha-1}
		+u_{\delta,\varepsilon}(x).
	\end{align}
	Moreover, by the same way as in the proof of Theorem 1 in \cite{Ko82}, $u_{\delta,\varepsilon}$ satisfies the equation
	\begin{align}\label{stable_2}
		Lu_{\delta,\varepsilon}
		=c_{\alpha}\psi_{\delta,\varepsilon},~
		c_{\alpha}
		:=
		-2\pi \alpha^{-1} \cot (\alpha \pi /2),
	\end{align}
	where $L$ is the infinitesimal generator of $Z$ defined by
	\begin{align*}
		Lf(x)
		:=
		\int_{\real}
		\left\{
			f(x+y)
			-
			f(x)
			-
			\1_{|y| \leq 1}(y) f{'}(y)
		\right\}
		\frac{1}{|y|^{1+\alpha}}
		\rd y.
	\end{align*}
	Then by It\^o's formula, we have
	\begin{align*}
		u_{\delta,\varepsilon}(X(t)-X^{(n)}(t))
		=
		M^{n,\delta,\varepsilon}(t)
		+
		J^{n,\delta,\varepsilon}(t),
	\end{align*}
	where
	\begin{align*}
		M^{n,\delta,\varepsilon}(t)
		&:=
		\int_{0}^{t}
			\int_{\real \setminus\{0\}}
				\Big\{
					u_{\delta,\varepsilon}
					(
						X(s-)
						-
						X^{(n)}(s-)\\
						&
						\quad
						+
						\{\sigma(X(s-))-\sigma(X^{(n)}(\eta_{n}(s)-))\}z
					)
				\Big\}
		\widetilde{N}(\rd s, \rd z),\\
		J^{n,\delta,\varepsilon}(t)
		&:=
		\int_{0}^{t}
			|\sigma(X(s))-\sigma(X^{(n)}(\eta_{n}(s)))|^{\alpha}
			(Lu_{\delta,\varepsilon})(X(s)-X^{(n)}(s))
		\rd s,
	\end{align*}
	and $\widetilde{N}$ is the compensated Poisson random measure.
	
	Now we consider the upper bound for $J^{n,\delta,\varepsilon}(t)$.
	From the equation \eqref{stable_2} and upper bound of $\psi_{\delta, \varepsilon}$, we have
	\begin{align*}
		J^{n,\delta,\varepsilon}(t)
		&
		\leq
		2c_{\alpha}
		\int_{0}^{T}
			\frac{\1_{[\varepsilon/\delta,\varepsilon]}(|X(s)-X^{(n)}(s)|)}{|X(s)-X^{(n)}(s)| \log \delta}
				\left|
					f_{\sigma}(X(s))
					-
					f_{\sigma}(X^{(n)}(\eta_{n}(s)))
				\right|
		\rd s\\
		&\leq
		2c_{\alpha}
		\left\{
			J_{1}^{n,\delta,\varepsilon}(T)
			+
			J_{2}^{n,\delta,\varepsilon}(T)
		\right\},
	\end{align*}
	where
	\begin{align*}
		J_{1}^{n,\delta,\varepsilon}(t)
		&:=
		\int_{0}^{t}
			\frac{\1_{[\varepsilon/\delta,\varepsilon]}(|X(s)-X^{(n)}(s)|)}{|X(s)-X^{(n)}(s)| \log \delta}
			\left|
				f_{\sigma}(X(s))
				-
				f_{\sigma}(X^{(n)}(s))
			\right|
		\rd s,\\
		J_{2}^{n,\delta,\varepsilon}(t)
		&:=
		\int_{0}^{t}
			\frac{\1_{[\varepsilon/\delta,\varepsilon]}(|X(s)-X^{(n)}(s)|)}{|X(s)-X^{(n)}(s)| \log \delta}
			\left|
				f_{\sigma}(X^{(n)}(s))
				-
				f_{\sigma}(X^{(n)}(\eta_{n}(s)))
			\right|
		\rd s.
	\end{align*}
	
	We first consider $J_{1}^{n,\delta,\varepsilon}(T)$.
	This part is based on the argument in \cite{BeOu08}.
	We consider approximation $f_{\sigma,\ell} \in C^1(\real)$ of $f_{\sigma}$ which is also strictly increasing function and satisfies $\|f_{\sigma, \ell}\|_{\infty} \leq \|f_{\sigma}\|_{\infty}$ and $f_{\sigma,\ell} \uparrow f_{\sigma}$ as $\ell \to \infty$ on $\real$.
	Then by using Fatou's lemma and the mean value theorem, by using the same way as Proposition \ref{main_1}, we have
	\begin{align*}
		J_{1}^{n,\delta,\varepsilon}(T)
		&\leq
		\liminf_{\ell \to \infty}
		\frac{1}{\log \delta}	
		\int_{0}^{T}\rd s \int_{0}^{1} \rd \theta f{'}_{\sigma, \ell}(V_{s}^{(n)}(\theta)),
	\end{align*}
	where $V_t^{(n)}(\theta):=(1-\theta)X(t)+\theta X^{(n)}(t)=x_{0}+\int_{0}^{t} H_{s}(\theta) \rd Z(s)$.
	Here $H_{t}(\theta):=(1-\theta) \sigma(X(t-))+\theta \sigma (X^{(n)}(t-))$ is an $\sigma(Z(s)~;~s\leq t)$-adapted process with $\e[\int_{0}^{t}|H_{s}(\theta)|^{\alpha} \rd s]<\infty$.
	Hence there exists an $\alpha$-stable process $\widetilde{Z}$ defined on an extended probability space such that $V_t^{(n)}(\theta)=x_{0}+\widetilde{Z}(A_{t}(\theta))$ for $t \geq 0$ where $A_{t}(\theta):=\int_{0}^{t}|H_{s}(\theta)|^{\alpha} \rd s$.
	Moreover, since $\sigma$ is bounded and uniformly positive, it holds that $H_{t}(\theta) \geq \inf_{x \in \real} \sigma(x)$ and $A_{t}(\theta) \leq \|\sigma\|_{\infty}t$.
	Therefore, by the change of variable $s=\tau_{u}:=\inf\{t \geq 0~;~A_t(\theta) > u\}$, we have
	\begin{align*}
		\int_{0}^{T}\rd s \int_{0}^{1} \rd \theta f{'}_{\sigma, \ell}(V_{s}^{(n)}(\theta))
		&=
		\int_{0}^{1} \rd \theta
			\int_{0}^{T} \rd A_{s}(\theta)
				f{'}_{\sigma, \ell}(\widetilde{Z}(A_{s}(\theta)))
				H_{s}(\theta)^{-\alpha}\\
		&\leq
		\inf_{x \in \real} \sigma(x)^{-\alpha}
			\int_{0}^{1} \rd \theta
				\int_{0}^{T} \rd A_{s}(\theta)
				f{'}_{\sigma, \ell}(\widetilde{Z}(A_{s}(\theta)))\\
		&=
		\inf_{x \in \real} \sigma(x)^{-\alpha}
			\int_{0}^{1} \rd \theta
				\int_{0}^{A_{T}(\theta)} \rd u
				f{'}_{\sigma, \ell}(\widetilde{Z}(u))\\
		&\leq
		\inf_{x \in \real} \sigma(x)^{-\alpha}
			\int_{0}^{1} \rd \theta
				\int_{0}^{\|\sigma\|_{\infty}T} \rd u
				f{'}_{\sigma, \ell}(\widetilde{Z}(u)).
	\end{align*}
	Note that the density function of stable process $\widetilde{Z}$, denoted by $p(t,\cdot)$ satisfies the upper bound $p(t,y) \leq C_{0}(t^{-1/\alpha} \wedge t |y|^{-(1+\alpha)})$ for some $C_{0}>0$ (see, e.g., Theorem 2.1 in \cite{BlGe60}), thus by using the estimate $\|f{'}_{\sigma, \ell}\|_{L^1(\real)} \leq 2 \|f_{\sigma, \ell}\|_{\infty} \leq 2 \|f_{\sigma} \|_{\infty}$, we have
	\begin{align}\label{stable_3}
		\e\left[J_{1}^{n,\delta,\varepsilon}(T)\right]
		&\leq
		\frac{1}{\log \delta}
			\inf_{x \in \real} \sigma(x)^{-\alpha}
				\liminf_{\ell \to \infty}
				\int_{0}^{1} \rd \theta
					\int_{\real} \rd y
						\int_{0}^{\|\sigma\|_{\infty}T} \rd u
							p(u,y)
							f{'}_{\sigma, \ell}(y)  \notag\\
		&\leq
		\frac{C_{1}\|f_{\sigma}\|_{\infty}}{\log \delta},
	\end{align}
	for some $C_{1}>0$.
	
	Finally, we consider $J_{2}^{n,\delta,\varepsilon}(T)$.
	By using Proposition \ref{main_6} with $g=f_{\sigma}$ and $q=1$ and $p \geq 1$, we have
	\begin{align*}
		\e[J_{2}^{n,\delta,\varepsilon}(T)]
		&\leq
		\frac{\delta}{\varepsilon \log \delta}
		\int_{0}^{T}
			\e\left[
			\left|
				f_{\sigma}(X^{n}(s))
				-
				f_{\sigma}(X^{(n)}(\eta_n(s)))
			\right|
			\right]
		\rd s\notag\\
		&\leq
		\frac{\delta}{\varepsilon \log \delta}
		C(f_{\sigma},\sigma,p,q,T)
		\left(
			\int_{0}^{T}
				\e
				\left[
					\left|
						X^{(n)}(s)
						-
						X^{(n)}(\eta_{n}(s))
					\right|^{p}
				\right]
			\rd s
		\right)^{\frac{1}{p+1}}.
	\end{align*}
	Note that since $\sigma$ is bounded, we have for any $p \in (0,\alpha)$,
	\begin{align*}
		\e
		\left[
			\left|
				X^{(n)}(s)
				-
				X^{(n)}(\eta_{n}(s))
			\right|^{p}
		\right]
		\leq
		\|\sigma\|_{\infty}^{p}
		\e\left[
			|Z(s)-Z(\eta_{n}(s))|^{p}
		\right]
		\leq
		\frac{C_{2}}{n^{p/\alpha}},
	\end{align*}
	for some $C_{2}>0$.
	Thus for any $p \in (0,\alpha)$, there exists $C_{3}>0$ such that
	\begin{align}\label{stable_4}
		\e[J_{3}^{n,\delta,\varepsilon}(T)]
		\leq
		C_{3}
		\frac{\delta}{\varepsilon \log \delta }
		\frac{1}{n^{\frac{p}{\alpha(p+1)}}}.
	\end{align}
	
	Let $\tau_{N}:=\inf\{t >0~:~|X(t)-X^{(n)}(t)| \geq N\}$ for $N \in \n$.
	Then $(M^{n,\delta,\varepsilon}(t \wedge \tau_{N}))_{t \in [0,T]}$ is a martingale, thus 
	it follows from \eqref{stable_1}, \eqref{stable_3} and \eqref{stable_4} that there exists a positive constant $C_{4}$ which do not depend on $n$ such that for any $t \in [0,T]$
	\begin{align*}
		\e[|X(t\wedge \tau_{N})-X^{(n)}(t\wedge \tau_{N})|^{\alpha}]
		\leq
		C_{4}
		\left\{
			\varepsilon^{\alpha-1}
			+
			\frac{1}{\log \delta}
			+
			\frac{\delta}{\varepsilon \log \delta }
			\frac{1}{n^{\frac{p}{\alpha(p+1)}}}.
		\right\}
	\end{align*}
	By choosing $\varepsilon = 1/\log n$ and $\delta = n^{\frac{p}{2\alpha(p+1)}}$ and taking the limit $N \to \infty$, we conclude the proof.
\end{proof}

\subsection{Fractional Brownian motions with irregular drift}\label{sec_3_6}

In this subsection, inspired by \cite{NuOu02}, we consider the one-dimensional SDEs driven by fractional Brownian motion on a probability space $(\Omega, \mathscr{F},\p)$ with a filtration $(\mathscr{F}_t)_{0\leq t \leq T}$ satisfying the usual conditions.

We first recall the definitions of fractional Brownian motion and its basic properties.
A one-dimensional continuous centered Gaussian process $B^{H}=(B^{H}(t))_{0 \leq t \leq T}$ is a fractional Brownian motion with Hurst parameter $H \in (0,1)$ on the probability space $(\Omega,\mathscr{F},\p)$ if it holds that
\begin{align*}
	\e[B^{H}(t) B^{H}(s)]
	=\frac{1}{2}
	\left\{
		|t|^{2H}
		+
		|s|^{2H}
		-
		|t-s|^{2H}
	\right\},~
	\forall s,t \in [0,T].
\end{align*}
If $H=1/2$, $B^H$ is a standard Brownian motion.
If $ H \neq 1/2$, $B^{H}$ is neither a Markov process nor a semimartingale.
A fractional Brownian motion satisfies the following properties; (i) self-similar property: for any $a >0$ and $t \in [0,T]$, $B^{H}(t)$ and $a^{-H}B^{H}(at)$ has the same distribution; (ii) stationary increments: for all $0\le s<t\le T$, $B^{H}(t)-B^{H}(s)$ is equal in distribution to $B^{H}(t-s)$.
It is known that (see, e.g. Corollary 3.1 in \cite{DeUs99}), a fractional Brownian motion $B^{H}$ has an integral representation of the form
\begin{align*}
	B^{H}(t)
	=
	\int_{0}^{t}
		K_{H}(t,s)
	\rd W(s),
\end{align*}
where $W=(W(t))_{0 \leq t \leq T}$ is a standard Brownian motion and the kernel $K_{H}:[0,T] \times [0,T] \to \real$ is given by
\begin{align*}
	K_{H}(t,s)
	:=
	\frac{(t-s)^{H-1/2}}{\Gamma(H+1/2)}
	F\left(
		H-\frac{1}{2},
		\frac{1}{2}-H,
		H+\frac{1}{2},
		1-\frac{t}{s}
	\right),
\end{align*}
where $F(a,b,c,z)$ is the Gauss hypergeometric function, defined by
\begin{align*}
	F(a,b,c,z)
	&:=
	\sum_{k=0}^{\infty}
		\frac{a_{(k)}b_{(k)}}{c_{(k)}} z^{k},~
	a,b \in \real,~
	c \neq 0, -1, -2, \ldots,~
	z \in (-1,1),\\
	a_{(0)}&:=1,~
	a_{(k)}:=a(a+1)\cdots (a+k-1).
\end{align*}
For $f \in L^{2}([0,T])$, define a linear operator $K_{H}: L^{2}([0,T]) \to L^{2}([0,T])$ by
\begin{align*}
	K_{H}f(t)
	:=
	\int_{0}^{t}
		K_{H}(t,s)
		f(s)
	\rd s.
\end{align*}
Then the operator $K_{H}$ has the following representation (see, e.g. Theorem 2.1 in \cite{DeUs99})
\begin{align*}
	K_{H}f
	=\left\{ \begin{array}{ll}
	\displaystyle
	I^{2H}_{0+}
	x^{1/2-H}
	I^{1/2-H}_{0+}
	x^{H-1/2}
	f
	&\text{ if } H < 1/2,  \\
	\displaystyle
	I^{1}_{0+}
	x^{H-1/2}
	I^{H-1/2}_{0+}
	x^{1/2_H}
	f
	&\text{ if } H>1/2,
	\end{array}\right.
\end{align*}
where $I^{\alpha}_{a+}$ is the left Riemann-Liouville fractional integral of $f \in L^{1}([a,b])$ of order $\alpha>0$ on $(a,b)$ defined by
\begin{align*}
	I^{\alpha}_{a^{+}}f(x)
	:=
	\frac{1}{\Gamma(\alpha)}
	\int_{a}^{x}
	(x-y)^{\alpha-1}
	f(y)
	\rd y,
\end{align*}
(for more details of fractional calculus, see, e.g. \cite{SaKiMa}).
Thus the operator $K_{H}$ is an isomorphism from $L^{2}([0,T])$ onto $I^{H+1/2}_{0+}(L^{2}([0,T]))$.

We consider the following one-dimensional SDEs driven by fractional Brownian motion of the form 
\begin{align}\label{SDE_3}
	\rd X(t)
	=
	b(X(t)) \rd t
	+
	\rd B^{H}(t),
	~X(0)=x_0 \in \real,
	~t \in [0,T]
\end{align}
and its Euler--Maruyama scheme given by
\begin{align*}
	\rd X^{(n)}(t)
	=
	b(X^{(n)}(\eta _n(t))) \rd t
	+
	\rd B^{H}(t),~
	X^{(n)}(0)=x_{0},~
	t \in [0,T].
\end{align*}

\begin{Rem}
	Note that if $H<1/2$ and $b$ is of linear growth or if $H>1/2$ and $b$ is $\gamma$-H\"older continuous with $\gamma \in (1-1/(2H),1)$, then there exists a unique strong solution for SDE \eqref{SDE_3} (see, Theorem 3, 5, 8 in \cite{NuOu02}).
\end{Rem}

\begin{Thm}\label{main_7}
	Suppose that $H<1/2$ and the drift coefficient $b$ is bounded and is one-sided Lipschitz, that is, there exists $K \geq 0$ such that for any $x,y \in \real$, $(x-y)(b(x)-b(y)) \leq K|x-y|^{2}$.
	Moreover, we assume that there exist $f_{b} \in BV$ and $\gamma \in (0,1]$ such that for any $x,y \in \real$,
	\begin{align*}
		|b(x)-b(y)|
		\leq
		|f_{b}(x)-f_{b}(y)|^{\gamma},
	\end{align*}
	Then for any $p \geq 2$ and $\varepsilon>0$, there exists a constant $C$ such that for any $n \geq 1$,
	\begin{align}\label{main_7_1}
		\e\left[
			\sup_{0\leq t \leq T}
			\left|
				X(t) - X^{(n)}(t)
			\right|^{p}
		\right]^{1/p}
		\leq
		\left\{ \begin{array}{ll}
		\displaystyle
		Cn^{-\frac{(1-\varepsilon)H}{p(H+1)}},
		&\text{ if } p\gamma \geq 1,\\
		\displaystyle
		Cn^{-\frac{(1-\varepsilon)\gamma H}{H+1}},
		&\text{ if } p\gamma < 1,
		\end{array}\right.
	\end{align}
	and
	\begin{align}\label{main_7_2}
	\e\left[
		\sup_{0\leq t \leq T}
			\left|
				X(t) - X^{(n)}(t)
			\right|
	\right]
	\leq
	Cn^{-\frac{(1-\varepsilon)\gamma H}{H+1}},~
	\text{ if } H<1/2.
	\end{align}
\end{Thm}

\begin{Rem}
	It is worth noting that if $H \in (0,1)$ and the drift coefficient $b$ is $\gamma$-H\"older continuous (that is, $f_{b}(x)=cx$, for some $c>0$) with $\gamma>2-1/H$, then it is shown in \cite{BuDaGe19} that for each $\varepsilon>0$ and $p \geq 2$,
	\begin{align*}
		\e\left[
			\sup_{0\leq t \leq 1}
			\left|
				X(t) - X^{(n)}(t)
			\right|^{p}
		\right]^{1/p}
		\leq
		Cn^{-(\frac{1}{2}+\gamma(H \wedge \frac{1}{2}))+\varepsilon},
	\end{align*}
	by using a stochastic sewing approach (see, Theorem 1.1 and Remark 1.1 in \cite{BuDaGe19}).
	On the other hand, in Theorem \ref{main_7} we need one-sided Lipschitz condition, but we can consider discontinuous function for the drift coefficient.
	For example, define $b(x):=\theta_{1} \1_{(-\infty,\delta]}(x)+\theta_{0} \1_{(\delta,\infty)}(x)$ with $\theta_{1}>\theta_{0}$,  then $b$ satisfies the assumption of Theorem \ref{main_7} with $\gamma=1$.
\end{Rem}



Before proving Theorem \ref{main_7}, we consider an analogue of Proposition \ref{main_1} for the Euler--Maruyama scheme with $H<1/2$.

\begin{Prop}\label{main_8}
	Suppose that $b$ is bounded and $H<1/2$.
	Let $\widehat{X}=(\widehat{X}(t))_{t \geq 0}$ be a one-dimensional progressively measurable process.
	Then for any $g \in BV$ and $p \in (0,\infty)$, $q \in [1,\infty)$ and $\beta \in (0,1/(H+1))$, there exists $C=C(g,b,p,q,\beta)>0$ such that
	\begin{align*}
		\int_{0}^{T}
			\e\left[
				\left|
					g(X^{(n)}(s))
					-
					g(\widehat{X}(s))
				\right|^{q}
			\right]
		\rd s
		\leq
		C
		\left(
		\int_{0}^{T}
			\e
			\left[
				\left|
					X^{(n)}(s)
					-
					\widehat{X}(s)
				\right|^p
			\right]
		\rd s
		\right)^{\frac{\beta}{p+\beta}}.
	\end{align*}
\end{Prop}
\begin{proof}
	From Theorem \ref{main_0}, it is suffices to estimate $\int_{0}^{T} \p(a<X^{(n)}(s)\leq b)\rd s$, for $a,b \in \real$ with $a<b$.
	The idea of the proof is based on the Krylov type estimate for $X^{(n)}$.
	
	(Step 1).
	Let $\widetilde{B}^{H}(t)=B^{H}(t)+\int_{0}^{t} b(X^{(n)}(\eta_{n}(s))) \rd s$.
	In order to use Girsanov thoemre, we first prove that $u:=b(X^{(n)}(\eta_{n}(\cdot)))$ satisfies conditions (i) and (ii) of Theorem 2 in \cite{NuOu02}.
	Note that the condition (i) is equivalent to $v:=K_{H}^{-1}(\int_{0}^{\cdot} u(s) \rd s) \in L^{2}([0,T])$, (see, page 107 of \cite{NuOu02}).
	If $h$ is absolutely continuous, then $K_{H}^{-1}h=s^{H-1/2}I^{1/2-H}_{0+}s^{1/2-H}h'$ (see, e.g. equation (13) in \cite{NuOu02}), thus it holds that
	\begin{align}\label{pr_8_0}
		\left|
			v(s)
		\right|
		&=
		\left|
			s^{H-1/2}I^{1/2-H}_{0+}s^{1/2-H}
			b(X^{(n)}(\eta_{n}(s)))
		\right| \notag\\
		&=
		\frac{s^{H-1/2}}{\Gamma(\frac{1}{2}-H)}
		\left|
			\int_{0}^{s}
				(s-r)^{-1/2-H}
				r^{\frac{1}{2}-H}
				b(X^{(n)}(\eta_{n}(r)))
			\rd r
		\right| \notag\\
		&\leq
		\frac{\|b\|_{\infty}\Gamma(\frac{3}{2}-H)T^{\frac{1}{2}-H}}{\Gamma(2-2H)},
	\end{align}
	and thus $u(s)$ satisfies the condition (i) of Theorem 2 in \cite{NuOu02}.
	Let $p \in \real$ and $Z(p,\cdot)=(Z(p,t))_{t \in [0,T]}$ be a exponential local martingale defined by
	\begin{align*}
		Z(p,t)
		:=
		\exp
		\left(
			p
			\int_{0}^{t}
				v(s)
			\rd W(s)
			-
			\frac{p^{2}}{2}
			\int_{0}^{t}
				v(s)^{2}
			\rd s
		\right).
	\end{align*}
	Then, from \eqref{pr_8_0} and Novikov condition, $Z(p,\cdot)$ is a martingale, thus we conclude that $u$ satisfies the condition (ii) of Theorem 2 in \cite{NuOu02}.
	Therefore, we obtain that $\widetilde{B}^{H}$ is a fractional Brownian motion with Hurst parameter $H$ under the probability measure $\rd \widetilde{\p}:=Z(-1,T) \rd \p$.
	
	(Step 2).
	Now we prove a Krylov type estimate for $X^{(n)}$, that is, for a measurable and non-negative function $g:[0,T] \times \real$, $p>H+1$ and $q>1$, there exists $C(p,q,H,T)>0$ such that
	\begin{align}\label{pr_8_1}
		\e\left[
			\int_{0}^{T}
				g(t,X^{(n)}(t))
			\rd t
		\right]
	\leq
	C(p,q,H,T)
	\left(
		\int_{0}^{T}
		\rd t
			\int_{\real}
			\rd y
				g(t,y)^{pq}
	\right)^{1/p}.
	\end{align}
	The proof is based on the proof of Proposition 6 in \cite{NuOu02}.
	By using Girsanov theorem (Step 1), H\"older's inequality and Jensen's inequality, for any $q,q'>1$ with $1/q+1/q'=1$, we have
	\begin{align*}
		\e
		\left[
			\int_{0}^{T}
				g(s,X^{(n)}(s))
			\rd s
		\right]
		&=
		\widetilde{\e}
		\left[
			Z(-1,T)
			\int_{0}^{T}
				g(s,X^{(n)}(s))
			\rd s
		\right]\\
		&\leq
		T^{q-1}
		\widetilde{\e}[Z(-1,T)^{q'}]^{1/q'}
		\widetilde{\e}
		\left[
			\int_{0}^{T}
				g(s,X^{(n)}(s))^{q}
			\rd s
		\right]^{1/q},
	\end{align*}
	where $\widetilde{\e}$ is the expectation with respect to the probability measure $\widetilde{\p}$.
	Since $Z(-q',\cdot)$ is a martingale and $b$ is bounded, we have
	\begin{align}\label{moment_Z}
		\widehat{\e}[Z(-1,T)^{q'}]
		&=
		\widehat{\e}\left[
			Z(-q',T)
			\exp
				\left(
					\frac{q'(q'-1)}{2}
					\int_{0}^{T}
						v(s)^{2}
					\rd s
				\right)
		\right]\\
		&\leq
		\exp\left(
			\frac{q'(q'-1)\|b\|_{\infty}^{2}T}{2}
		\right)
		<\infty \notag.
	\end{align}
	Since $X^{(n)}$ is a fractional Brownian motion with Hurst parameter $H$ starting from $x_{0}$ under the probability measure $\widetilde{\p}$, by using H\"older's inequality with $1/p+1/p'=1$ and $p>H+1$, we have
	\begin{align*}
		&\widetilde{\e}
		\left[
			\int_{0}^{T}
				g(s,X^{(n)}(s))^{q}
			\rd s
		\right]
		=
		\int_{0}^{T}
			\rd s
			\int_{\real}
				\rd y
				g(s,y)^{q}
				\frac{
					\exp\left(
						-\frac{|y-x_{0}|^{2}}{2s^{2H}}
					\right)
				}
				{
					\sqrt{2 \pi s^{2H}}
				}\\
		&\leq
		\left(
			\int_{0}^{T}
				\rd s
				\int_{\real}
					\rd y
						g(s,y)^{pq}
		\right)^{1/p}
		\left(
			\int_{0}^{T}
				\rd s
				\int_{\real}
					\rd y
						\frac{\sqrt{p'}}{(2\pi)^{\frac{p'-1}{2}}}
						\frac{1}{s^{H(p'-1)}}
						\frac{
							\exp\left(
								-\frac{p' |y-x_{0}|^{2}}{2s^{2H}}
							\right)
						}
						{
							\sqrt{2 \pi (1/p') s^{2H}}
						}
		\right)^{1/p'}\\
		&\leq
		C_{0}
		\left(
			\int_{0}^{T}
				\rd s
				\int_{\real}
					\rd y
					g(s,y)^{pq}
		\right)^{1/p},
	\end{align*}
	for some $C_{0}>0$.
	This concludes \eqref{pr_8_1}.
	
	(Step 3).
	By using Krylov type estimate for $X^{(n)}$ (Step 2), for any $p=1/\beta>H+1$, we have
	\begin{align*}
	\int_{0}^{T}
		\p(a<X^{(n)}(s)\leq b)
	\rd s
	\leq
	C(p,q,H,T)
	\left(
		\int_{0}^{T} \rd s
		\int_{\real}
			\rd x
			\1_{(a,b]}(x)
	\right)^{\beta}
	\leq
	C(p,q,H,T)
	T^{\beta}
	(b-a)^{\beta}.
	\end{align*}
	Therefore, the map $\real \ni x \mapsto \int_{0}^{T} \p(X^{(n)}(s) \leq x) \rd s$ is $\beta$-H\"older continuous, and thus we conclude the proof.
\end{proof}


\begin{proof}[Proof of Theorem \ref{main_7}]
	We first note that since $b$ is bounded, for any $p \geq 1$, 
	$
		\e[
			|X^{(n)}(s)
			-
			X^{(n)}(\eta_{n}(s))|^{p}
		]
		\leq
		C_{0}n^{-pH}
	$ for some $C_{0}>0$.
	By using chain rule, one-sided Lipschitz condition for $b$ and Young's inequality, we have
	\begin{align}\label{pr_7_0}
		&\left|
			X(t)
			-
			X^{(n)}(t)
		\right|^{p}\notag\\
		&=
		p
		\int_{0}^{t}
			\left\{
				X(s)
				-
				X^{(n)}(s)
			\right\}^{p-1}
			\left\{
				b(X(s))
				-
				b(X^{(n)}(\eta_{n}(s)))
			\right\}
		\rd s \notag\\
		&\leq
		(2p-1)K
		\int_{0}^{t}
			\left|
				X(s)
				-
				X^{(n)}(s)
			\right|^{p}
		\rd s
		+
		\int_{0}^{t}
			\left|
				b(X^{(n)}(s))
				-
				b(X^{(n)}(\eta_{n}(s)))
			\right|^{p}
		\rd s.
	\end{align}
	Now we consider the expectation of the second term of \eqref{pr_7_0}.
	Since $b$ is of bounded $1/\gamma$-variation, by using Proposition \ref{main_8} with $q=p\gamma$ or $q=1$, $g=f_{b}$ and $\widehat{X}=X^{(n)}(\eta_{n}(\cdot))$, for any $\widehat{p} \geq \beta$ and $\beta \in (0,1/(H+1))$, we have
	\begin{align*}
		&\e
		\left[
			\int_{0}^{t}
				\left|
					b(X(s))
					-
					b(X^{(n)}(\eta_{n}(s)))
				\right|^{p}
			\rd s
		\right]\\
		&\leq
		\left\{ \begin{array}{ll}
		\displaystyle
			\int_{0}^{T}
				\e\left[
					\left|
						f_{b}(X(s))
						-
						f_{b}(X^{(n)}(\eta_{n}(s)))
					\right|^{p \gamma}
				\right]
			\rd s
		&\text{ if } p\gamma \geq 1,  \\
		\displaystyle
			T^{1-p\gamma}
			\left(
				\int_{0}^{T}
					\e\left[
						\left|
							f_{b}(X(s))
							-
							f_{b}(X^{(n)}(\eta_{n}(s)))
						\right|
					\right]
				\rd s
			\right)^{p \gamma}
		&\text{ if } p\gamma < 1,
		\end{array}\right.\\
		&\leq
		\left\{ \begin{array}{ll}
		\displaystyle
			C(f_{b}, b,\widehat{p},p\gamma,T)
			\left(
				\int_{0}^{T}
					\e\left[
						\left|
							X^{(n)}(s)
							-
							X^{(n)}(\eta_{n}(s))
						\right|^{\widehat{p}}
					\right]
				\rd s
			\right)^{\frac{\beta}{\widehat{p}+\beta}}
			&\text{ if } p\gamma \geq 1,  \\
		\displaystyle
			T^{1-p\gamma}
			C(f_{b}, b,\widehat{p},1,T)
			\left(
				\int_{0}^{T}
				\e\left[
					\left|
						X^{(n)}(s)
						-
						X^{(n)}(\eta_{n}(s))
					\right|^{\widehat{p}}
				\right]
				\rd s
			\right)^{\frac{p\gamma \beta}{\widehat{p}+\beta}}
			&\text{ if } p\gamma < 1,
		\end{array}\right.\\
		&\leq
		\left\{ \begin{array}{ll}
		\displaystyle
			C_{2}n^{-\frac{\widehat{p} H \beta}{\widehat{p}+\beta}},
			&\text{ if } p\gamma \geq 1,  \\
		\displaystyle
			C_{2}n^{-\frac{\widehat{p} H  p \gamma \beta}{\widehat{p}+\beta}},
			&\text{ if } p\gamma < 1,
		\end{array}\right.
	\end{align*}
	for some $C_{2}>0$.
	By using Gronwall's inequality, since $\widehat{p} \geq \beta$ and $\beta \in (1,1/(H+1))$ are arbitrarily and $\frac{\widehat{p}\beta}{\widehat{p}+\beta} \to 1/(H+1)$ as $\widehat{p} \to \infty$ and $\beta \to 1/(H+1)$, we obtain \eqref{main_7_1}.
	
	Now we prove \eqref{main_7_2} by using Yamada and Watanabe approximation technique which is used in the proof of Theorem \ref{main_3}.
	By using chain rule, $\phi_{\delta,\varepsilon}(X(t)-X^{(n)}(t))$ can be decomposed by the following two terms
	\begin{align*}
		\phi_{\delta,\varepsilon}(X(t)-X^{(n)}(t))
		=
		I_{1}^{n,\delta,\varepsilon}(t)
		+
		I_{2}^{n,\delta,\varepsilon}(t),
	\end{align*}
	where
	\begin{align*}
		I_{1}^{n,\delta,\varepsilon}(t)
		&:=
		\int_{0}^{t}
			\phi{'}_{\delta,\varepsilon}(X(s)-X^{(n)}(s))
			\left\{
				b(X(s))
				-
				b(X^{(n)}(s))
			\right\}
		\rd s,\\
		I_{2}^{n,\delta,\varepsilon}(t)
		&:=
		\int_{0}^{t}
			\phi{'}_{\delta,\varepsilon}(X(s)-X^{(n)}(s))
			\left\{
				b(X^{(n)}(s))
				-
				b(X^{(n)}(\eta_n(s)))
			\right\}
		\rd s.
	\end{align*}
	Note that since $\phi{'}_{\delta,\varepsilon}$ satisfies $\phi{'}_{\delta,\varepsilon}(x)/x>0$, $x \neq 0$, by using one-sided Lipschitz condition on $b$, we have $\phi{'}_{\delta,\varepsilon}(x-y)(b(x)-b(y)) \leq K|x-y|$, for any $x,y \in \real$.
	Therefore, it holds that
	\begin{align*}
		I_{1}^{n,\delta,\varepsilon}(t)
		\leq
		K
		\int_{0}^{t}
			\left|
				X(s)
				-
				X^{(n)}(s)
			\right|
		\rd s.
	\end{align*}
	On the other hand, since $b$ is of bounded $1/\gamma$-variation, by using Proposition \ref{main_8} with $q=1$, $g=f_{b}$ and $\widehat{X}=X^{(n)}(\eta_{n}(\cdot))$, for any $\widehat{p} \geq \beta$, we have
	\begin{align*}
		\e\left[
			I_{2}^{n,\delta,\varepsilon}(t)
		\right]
		&\leq
		T^{1-\gamma}
		\left(
			\int_{0}^{T}
				\e\left[
					\left|
						f_{b}(X^{(n)}(s))
						-
						f_{b}(X^{(n)}(\eta_n(s)))
					\right|
				\right]
			\rd s
		\right)^{\gamma}\\
		&
		\leq
		T^{1-\gamma}
		C(f_{b}, b,\widehat{p},\gamma,T)
		\left(
			\int_{0}^{T}
				\e\left[
					\left|
						X^{(n)}(s)
						-
						X^{(n)}(\eta_{n}(s))
					\right|^{\widehat{p}}
				\right]
			\rd s
		\right)^{\frac{\gamma \beta}{\widehat{p}+\beta}}\\
		&\leq
		C n^{-\frac{\widehat{p} H \gamma \beta}{\widehat{p}+\beta}},
	\end{align*}
	for some $C>0$.
	By choosing $\varepsilon=n^{-\frac{\widehat{p} H \gamma \beta}{\widehat{p}+\beta}}$, $\delta=2$ and by usign Gronwall's inequality, since $\widehat{p} \geq \beta$ and $\beta \in (1,1/(H+1))$ are arbitrarily and $\frac{\widehat{p}\beta}{\widehat{p}+\beta} \to 1/(H+1)$ as $\widehat{p} \to \infty$ and $\beta \to 1/(H+1)$, we obtain \eqref{main_7_2}.
\end{proof}

\subsection{Stochastic heat equations with irregular drift}\label{sec_3_7}

Let us consider the following one-dimensional stochastic heat equations (SHEs) of the form
\begin{align}\label{SPDE_1}
	\frac{\partial}{\partial t}u(t,x)
	&=
	\frac{\partial^{2}}{\partial x^{2}}u(t,x)
	+
	b(t,x,u(t,x))
	+
	\sigma(t,x,u(t,x))
	\frac{\partial^{2}}{\partial t \partial x}W(t,x),
	~(t,x) \in [0,T] \times [0,1],
\end{align}
with Dirichlet boundary conditions
\begin{align}\label{SPDE_2}
	u(t,0)=u(t,1)=0,~t \in [0,T],
\end{align}
and with initial condition
\begin{align}\label{SPDE_3}
	u(0,x)=u_{0}(x),~x \in [0,1],
\end{align}
where the coefficients $b,\sigma:[0,T] \times [0,1] \times \real \to \real$ are measurable functions, $u_{0}$ is a continuous function on $[0,1]$ and $W=\{W(t,x)~;~(t,x) \in [0,T] \times [0,1]\}$ is a Brownian sheet, that is, $W$ is a zero means Gaussian random field with covariance $\e[W(t,x)W(s,y)]=(t \wedge s) (x \wedge y)$, on a probability space $(\Omega, \mathscr{F},\p)$ with a filtration $(\mathscr{F}_t)_{0\leq t \leq T}$ satisfying the usual conditions.
Let $\mathscr{P}$ be a progressively measurable subsets of $[0,T] \times \Omega$.
We say that a $\mathscr{P} \otimes \mathscr{B}([0,1])$-measurable and continuous random field $u=\{u(t,x)~;~(t,x) \in [0,T] \times [0,1]\}$ is a solution of equations \eqref{SPDE_1}, \eqref{SPDE_2} and \eqref{SPDE_3}, if for any $\varphi \in C^{2}([0,1];\real)$ with $\varphi(0)=\varphi(1)=1$, it holds that
\begin{align*}
	\int_{0}^{1}
		u(t,x)
		\varphi(x)
	\rd x
	&=
	\int_{0}^{1}
		u_{0}(x)
		\varphi(x)
	\rd x\\
	&\quad
	+
	\int_{0}^{t}
		\int_{0}^{1}
			u(t,x)
			\varphi^{''}(x)
			+
			f(s,x,u(t,x))
			\varphi(x)
		\rd x
	\rd s\\
	&\quad
	+
	\int_{0}^{t}
		\int_{0}^{1}
			\sigma(s,x,u(t,x))\varphi(x)
	W(\rd s, \rd x).
\end{align*}
and if the coefficients are locally bounded, then this formulation is equivalent to the following the form (see, e.g. \cite{Sh94}, \cite{Wa86})
\begin{align*}
	u(t,x)
	&=
	\int_{0}^{1}
		G(t,x,y)u_{0}(y)
	\rd y
	+
	\int_{0}^{t}
		\int_{0}^{1}
			G(t-s,x,y)f(s,y,u(t,y))
		\rd y
	\rd s\\
	&\quad
	+
	\int_{0}^{t}
		\int_{0}^{1}
			G(t-s,x,y)
			\sigma(s,y,u(t,y))
	W(\rd s, \rd y),
\end{align*}
where $G(t,x,y)$ is the fundamental solution of the heat equation on $[0,T] \times [0,1]$ with Dirichlet boundary condition, that is,
\begin{align*}
	G(t,x,y)
	=
	\frac{1}{\sqrt{2 \pi t}}
	\sum_{n \in \z}
		\left\{
			\exp
				\left(
					-\frac{|y-x+2n|^{2}}{2t}
				\right)
			-
			\exp
				\left(
					-\frac{|y+x+2n|^{2}}{2t}
				\right)
		\right\}.
\end{align*}
Moreover, the existence and uniqueness of the above SHEs are shown in \cite{Wa86} with Lipschitz continuous coefficients.
On the other hand, as mentioned in introduction, Bally, Gy\"ongy and Pardoux \cite{BaGyPa94} extend this result in case the drift coefficient is measurable and satisfies a one-sided linear growth condition, that is, $ub(t,x,u) \leq K(1+|u|^{2})$ for some $K>0$, and the diffusion coefficient is non-generate, linear growth and has a locally Lipschitz derivative, by using a Krylov type estimate based on Girsanov transform and some density estimate as applications of Malliavin calculus (see, Proposition 4.1 in \cite{BaGyPa94}).

Now we consider the following numerical scheme introduced by Gy\"ongy \cite{Gy99};
\begin{align*}
	u_{m}^{n}(t_{i+1},x_{j})
	&=
	\left(
		I
		+
		\frac{T}{m}
		\Delta_{n}
	\right)
	u_{m}^{n}(t_{i},\cdot)(x_{j})
	+
	\frac{T}{m}
	b(t_{i},x_{j},u_{m}^{n}(t_{i},x_{j}))\\
	&\quad+
	\frac{T}{m}
	\sigma(t_{i},x_{j},u_{m}^{n}(t_{i},x_{j})
	\square_{mn}W(t_{i},x_{j}),
	\notag \notag\\
	u_{m}^{n}(t_{i},0)
	&
	=
	u_{m}^{n}(t_{i},1)
	=0,~
	i=1.\ldots,m,\notag\\
	u_{m}^{n}(0,x_{j})
	&=
	u_{0}(x_{j}),~j=1.\ldots,n-1\notag,
\end{align*}
where $t_{i}\equiv t_{i}^{(m)}:=iT/m$, $x_{j}\equiv x_{j}^{(n)}:=j/n$, $I$ is the identity operator and
\begin{align*}
	\Delta_{n} \varphi(x_{j})
	&:=
	n^{2}
	\left\{
		\varphi(x_{j+1})-2\varphi(x_{j})+\varphi(x_{j-1})
	\right\},\\
	\square_{mn}\psi(t_{i}, x_{j})
	&:=
	nm
	\left\{
		\psi(t_{i+1},x_{j+1})
		-
		\psi(t_{i},x_{j+1})
		-
		\psi(t_{i+1},x_{j})
		+
		\psi(t_{i},x_{j})
	\right\},
\end{align*}
for functions $\varphi$ and $\psi$ defined on $\{x_{j}~;~j=0,\ldots, n\}$ and on the lattice $\mathcal{L}:=\{(t_{i},x_{j})~;~i=0,\ldots, m,~j=0,\ldots, n\}$, respectively.
Then we extend $u_{m}^{n}$ from $\mathcal{L}$ by polygonal interpolation as
\begin{align*}
	u_{m}^{n}(t,x)
	&:=
	u_{m}^{n}(t,x_{j})
	+
	n(x-x_{j})
	\{
		u_{m}^{n}(t,x_{j})
		-
		u_{m}^{n}(t,x_{j+1})
	\},\\
	u_{m}^{n}(t,x_{j})
	&:=
	u_{m}^{n}(t_{i},x_{j})
	+
	\frac{m}{T}(t-t_{i})
	\{
		u_{m}^{n}(t_{i+1},x_{j})
		-
		u_{m}^{n}(t_{i},x_{j+1})
	\},
\end{align*}
for $(t,x) \in (t_{i},t_{i+1}) \times (x_{j}, x_{j+1})$, $i=0,\ldots, m-1$, $j=0,\ldots, n-1$.
Then it is shown in \cite{Gy99} that $u_{m}^{n}(t,x)$ satisfies the following stochastic integral equation of the form
\begin{align}\label{EM_SPDE_0}
	u_{m}^{n}(t,x)
	&=
	\int_{0}^{1}
		G_{m}^{n}(t,x,y)u_{0}(\kappa_{n}(y))
	\rd y \notag\\
	&\quad
	+
	\int_{0}^{t}
		\int_{0}^{1}
			G_{m}^{n}(t-s+T/m,x,y)
			b(\eta_{m}(s),\kappa_{n}(y),u(\eta_{m}(s),\kappa_{n}(y)))
		\rd y
	\rd s\notag\\
	&\quad
	+
	\int_{0}^{t}
		\int_{0}^{1}
			G_{m}^{n}(t-s+T/m,x,y)
			\sigma(\eta_{m}(s),\kappa_{n}(y),u(\eta_{m}(s),\kappa_{n}(y)))
	W(\rd s, \rd y),
\end{align}
for any $t=iT/m$, $i=0,\ldots,m$, $x \in [0,1]$, where $\kappa_{n}(x):=[nx]/n$, $\eta_{n}(s):=T\kappa_{m}(s/T)$ and
\begin{align*}
	G_{m}^{n}(t,x,y)
	:=
	\sum_{j=1}^{n-1}
	(1+\lambda_{j}^{n}T/m)^{[mt/T]}
	\varphi_{j}^{n}(x)
	\varphi_{j}^{n}(\kappa_{n}(y)),~t \in [0,T], x,y \in [0,1],
\end{align*}
and
\begin{align*}
	\varphi_{j}(x)
	&:=\sqrt{2} \sin(jx \pi),~x \in [0,1],\\
	\varphi_{j}^{n}(x)
	&:=
	\varphi_{j}(x_{k})+n(x-x_{k})(\varphi_{j}(x_{k+1})-\varphi_{j}(x_{k})),~x \in [x_{k}, x_{k+1}],\\
	\lambda_{j}^{n}
	&:=
	-4n^{2} \sin^{2}(j\pi/2n).
\end{align*}
Here $[x]$ stands for the integer part of $x$.
Then we have the following strong rate of convergence for $u_{m}^{n}$ (see, Theorem 3.2 (iii) in \cite{Gy99});
Suppose that $u_{0} \in C^{3}([0,1];\real)$ and there exists $K>0$ such that for all $t,s\in [0,T]$, $x,y \in [0,1]$ and $u,v \in \real$,
\begin{align*}
	&|b(t,x,u)-b(s,y,v)|
	+
	|\sigma(t,x,u)-\sigma(s,y,v)|\\
	&\leq
	K
	\left\{
		|t-s|^{1/4}
		+
		|x-y|^{1/2}
		+
		|u-v|
	\right\}.
\end{align*}
Then for each $p \geq 2$, there exists $C>0$ such that
\begin{align}\label{Thm_Gy}
	\sup_{(t,x) \in [0,T] \times [0,1]}
	\e\left[
		\left|
			u(t,x)
			-
			u_{m}^{n}(t,x)
		\right|^{p}
	\right]^{1/p}
	\leq
	C
	\{
		m^{-1/4}
		+
		n^{-1/2}
	\}.
\end{align}

In this subsection, we apply a generalized Avikainen's estimate, in order to obtain the following weak rate of convergence for $u_{m}^{n}$ with irregular drift coefficient $b$.

\begin{Thm}\label{main_11}
	Suppose that $u_{0} \in C^{3}([0,1];\real)$, $b,\sigma:[0,T] \times [0,1] \times \real \to \real$ are bounded, measurable, and $\sigma$ is uniformly elliptic and $\sigma(t,x,\cdot) \in C^{2}_{b}(\real;\real)$ for all $(t,x) \in [0,T] \times [0,1]$.
	Moreover, there exists $f_{b} \in BV$ and $\gamma \in (0,1]$ and $K>0$ such that for all $t,s\in [0,T]$, $x,y \in [0,1]$ and $u,v \in \real$,
	\begin{align*}
		|b(t,x,u)-b(s,y,v)|
		&\leq
		K
		\left\{
			|t-s|^{1/4}
			+
			|x-y|^{1/2}
		\right\}
		+
		|f_{b}(u)-f_{b}(v)|^{\gamma},\\
		|\sigma(t,x,u)-\sigma(s,y,v)|
		&\leq
		K
		\left\{
			|t-s|^{1/4}
			+
			|x-y|^{1/2}
			+
			|u-v|
		\right\}.
	\end{align*}
	Then for any bounded and $\rho$-H\"older continuous function $f$ with $\rho \in (0,1]$ and $\varepsilon \in (0,1)$, there exists $C>0$ such that
	\begin{align*}
		\sup_{(t,x) \in [0,T] \times [0,1]}
		\left|
			\e[f(u(t,x))]
			-
			\e[f(u_{m}^{n}(t,x))]
		\right|
		&\leq
		C\left\{
			m^{-\frac{1}{4} \left(\rho \wedge \frac{(1-\varepsilon)\gamma}{4}\right)}
			+
			n^{-\frac{1}{2} \left(\rho \wedge \frac{(1-\varepsilon)\gamma}{4}\right)}
		\right\}.
	\end{align*}
	In particular, if $f_{b}(u)=Ku$, that is, the map $\real \ni u \mapsto b(t,x,u)$ is $\gamma$-H\"older continuous, then the factor $\frac{(1-\varepsilon)\gamma}{4}$ can be replaced by $\gamma$.
\end{Thm}

For proving Theorem \ref{main_11}, we prove the following estimate based on a Krylov type estimate for a solution of SHEs proved in \cite{BaGyPa94}.

\begin{Prop}\label{Key_1}
	Suppose assumptions on the coefficients $b$ and $\sigma$ in Theorem \ref{main_11} hold.
	Then for any $g \in BV$, $p \in (0,\infty)$, $q \in [1,\infty)$ and $\alpha \in (0,1/2)$, there exists a positive constant $C=C(g,b,\sigma,p,q)>0$ such that for any $\mathscr{P} \otimes \mathscr{B}([0,1])$-measurable random field $\widehat{u}=\{\widehat{u}(t,x)~;~(t,x) \in [0,T] \times [0,1]\}$, we have
	\begin{align*}
		&\int_{0}^{T}
			\int_{0}^{1}
				\e\left[
					\left|
						g(u(s,x))
						-
						g(\widehat{u}(s,x))
					\right|^{q}
				\right]
			\rd x
		\rd s
		\leq
		C
		\left(
			\int_{0}^{T}
				\int_{0}^{1}
					\e
					\left[
						\left|
							u(s,x)
							-
							\widehat{u}(s,x)
						\right|^p
					\right]
				\rd x
			\rd s
		\right)^{\frac{\alpha}{p+\alpha}}.
	\end{align*}
\end{Prop}
\begin{proof}
	From Theorem \ref{main_0}, it is suffices to estimate $\int_{0}^{T} \int_{0}^{1} \p(a < u(s,x) \leq b) \rd x \rd s$.
	Let $\rho>2$.
	Then from Proposition 4.1 in \cite{BaGyPa94}, there exists $C>0$ such that
	\begin{align*}
		\int_{0}^{T}
			\int_{0}^{1}
				\p(a < u(s,x) \leq b)
			\rd x
		\rd s
		&\leq
		C
		\left(
			\int_{0}^{T}
				\int_{0}^{1}
					\int_{\real}
						\1_{(a,b]}(u)
					\rd u
				\rd x
			\rd s
		\right)^{1/\rho}
		\leq
		CT^{1/\rho}(b-a)^{1/\rho}.
	\end{align*}
	This concludes the proof.
\end{proof}

\begin{proof}[Proof of Theorem \ref{main_11}]
	We apply Maruyama--Girsanov transform in order to remove the drift coefficient from $u$ and $u_{m}^{n}$.
	Let $\mu:=b/\sigma$, $p \in \real$.
	We define $Z(p,\cdot,u)=(Z(p,t,u))_{0 \leq t \leq T}$ and $Z_{m}^{n}(p,\cdot,u_{m}^{n})=(Z_{m}^{n}(p,t,u_{m}^{n}))_{0 \leq t \leq T}$ by
	\begin{align*}
		Z(p,t,u)
		:=
		\exp\Bigg(
			p
			\int_{0}^{t}
				\int_{0}^{1}
					\mu(s,x,u(s,y))
			W(\rd s, \rd y)
			-
			\frac{p^{2}}{2}
			\int_{0}^{t}
				\int_{0}^{1}
					\mu(s,x,u(s,y))^{2}
				\rd y
			\rd s
		\Bigg)
	\end{align*}
	and
	\begin{align*}
		Z_{m}^{n}(p,t,u_{m}^{n})
		:=
		\exp\Bigg(
			p
			\int_{0}^{t}
				\int_{0}^{1}
					\mu(\eta_{m}(s),\kappa_{n}(y),u_{m}^{n}(\eta_{m}(s),\kappa_{n}(y)))
			W(\rd s, \rd y)\\
			-
			\frac{p^{2}}{2}
			\int_{0}^{t}
				\int_{0}^{1}
					\mu(\eta_{m}(s),\kappa_{n}(y),u_{m}^{n}(\eta_{m}(s),\kappa_{n}(y)))^{2}
				\rd y
			\rd s
		\Bigg)
	\end{align*}
	Then since $\mu=b/\sigma$ is bounded, thus, $Z(p,\cdot,u)$ and $Z_{m}^{n}(p,\cdot,u_{m}^{n})$ are martingale, and by the same way as \eqref{moment_Z}, $Z(-1,t,u)$ and $Z_{m}^{n}(-1,t,u_{m}^{n})$ have any both positive and negative moments, uniformly in $t \in [0,T]$ and $n,m \in \n$.
	Moreover, by Maruyama-Girsanov theorem, for any bounded measurable function $f:\real \to \real$, it holds that
	\begin{align*}
		\e[f(u(x,t))]
		&=
		\e[f(v(x,t))Z(-1,t,v)]
		~\text{and}~
		\e[f(u_{m}^{n}(x,t))]
		=
		\e[f(v_{m}^{n}(x,t))Z_{m}^{n}(-1,t,v_{m}^{n})],
	\end{align*}
	where $v=\{v(t,x)~;~(t,x) \in [0,T] \times [0,1]\}$ is solution of SHE \eqref{SPDE_1}, \eqref{SPDE_2} and \eqref{SPDE_3} with $b=0$, and $v_{m}^{n}=\{v_{m}^{n}(t,x)~;~(t,x) \in [0,T] \times [0,1]\}$ is its approximation defined in \eqref{EM_SPDE_0} with $b=0$.
	By using H\"older continuity of $f$ and  H\"older's inequality with $1/p+1/p'=1$, $p\geq 2/\rho$, we have
	\begin{align}\label{pr_11_1}
		\left|
			\e[f(u(x,t))]
			-
			\e[f(u_{m}^{n}(x,t))]
		\right|
		&\leq
			\e\left[
				\left|
					f(v(x,t))
					-
					f(v_{m}^{n}(x,t))
				\right|
				Z(-1,t,v)
			\right]\notag\\
			&\quad+
			\e\left[
				\left|
					f(v_{m}^{n}(x,t))
				\right|
				\left|
					Z(-1,t,v)
					-
					Z_{m}^{n}(-1,t,v_{m}^{n})
				\right|
			\right] \notag\\
		&\leq
		\|f\|_{\rho}
		\e\left[
		\left|
			v(x,t)
			-
			v_{m}^{n}(x,t)
		\right|^{\rho p}
		\right]^{1/p}
		\e\left[
			Z(-1,t,v)^{p'}
		\right]^{1/p'}\notag\\
		&\quad+
		\|f\|_{\infty}
		\e\left[
			\left|
				Z(-1,t,v)
				-
				Z_{m}^{n}(-1,t,v_{m}^{n})
			\right|
		\right].
	\end{align}
	Thus it is sufficient to estimate the second part of \eqref{pr_11_1}.
	By using the elementary estimate $|e^{x}-e^{y}|\leq (e^{x}+e^{y})|x-y|$ and Schwarz inequality, we have
	\begin{align*}
		&\e\left[
			\left|
				Z(-1,t,v)
				-
				Z_{m}^{n}(-1,t,v_{m}^{n})
			\right|
		\right]\\
		&\leq
		\sqrt{2}
		\e\left[
			Z(-1,t,v)^{2}
			+
			Z_{m}^{n}(-1,t,v_{m}^{n})^{2}
		\right]^{1/2}
		\e\left[
			A(t)^{2}
			+
			B(t)^{2}
		\right]^{1/2},
	\end{align*}
	where
	\begin{align*}
		A(t)
		&:=
		-
		\int_{0}^{t}
			\int_{0}^{1}
				\{
					\mu(s,x,v(s,y))
					-
					\mu(\eta_{m}(s),\kappa_{n}(y),v_{m}^{n}(\eta_{m}(s),\kappa_{n}(y))
				\}
		W(\rd s, \rd y)
		\\
		B(t)
		&:=
		\frac{1}{2}
		\int_{0}^{t}
			\int_{0}^{1}
				\{
					\mu(s,x,v(s,y))^{2}
					-
					\mu(\eta_{m}(s),\kappa_{n}(y),v_{m}^{n}(\eta_{m}(s),\kappa_{n}(y)))^{2}
				\}
			\rd y
		\rd s.
	\end{align*}
	By the assumptions on $b$ and $\sigma$, we have 
	for all $t,s\in [0,T]$, $x,y \in [0,T]$ and $u,v \in \real$,
	\begin{align*}
	|\mu(t,x,u)-\mu(s,y,v)|
	&\leq
	C_{0}
	\left\{
		|t-s|^{1/4}
		+
		|x-y|^{1/2}
		+
		|u-v|
	\right\}
	+
	|f_{b}(u)-f_{b}(v)|^{\gamma},
	\end{align*}
	for some $C_{0}>0$ and thus by It\^o's isometly, it holds that
	\begin{align*}
		&\e\left[
			A(t)^{2}
			+
			B(t)^{2}
		\right]^{1/2}
		\leq
		C_{1}
		\{
			m^{-1/4}
			+
			n^{-1/2}
		\}\\
		&+
		C_{1}
		\left(
			\int_{0}^{T}
				\int_{0}^{1}
					\e\left[
						\left|
							v(s,y)
							-
							v_{m}^{n}(s,y)
						\right|^{2}
						+
						\left|
							f_{b}(v(s,y))
							-
							f_{b}(v_{m}^{n}(s,y))
						\right|^{2}
					\right]
				\rd y
			\rd s
		\right)^{\gamma/2},
	\end{align*}
	for some $C_{1}>0$.
	Hence, if $f_{b}(u)=Ku$, then by uniform estimate \eqref{Thm_Gy} with $b\equiv 0$, we conclude the proof.
	For general case on $f_{b}$, by using Proposition \ref{Key_1} with $g=f_{b}$, $q=2$ and $\widehat{u}=u_{m}^{n}$ and uniform estimate \eqref{Thm_Gy} with $b\equiv 0$, we have for any $p \in (0,\infty)$ and $\alpha \in (0,1/2)$,
	\begin{align*}
		&\int_{0}^{T}
			\int_{0}^{1}
				\e\left[
					\left|
						v(s,y)
						-
						v_{m}^{n}(s,y)
					\right|^{2}
					+
					\left|
						f_{b}(v(s,y))
						-
						f_{b}(v_{m}^{n}(s,y))
					\right|^{2}
				\right]
			\rd y
		\rd s\\
		&
		\leq
		C_{2}
		\left\{
			m^{-\frac{p\alpha}{4(p+\alpha)}}
			+
			n^{-\frac{p\alpha}{2(p+\alpha)}}
		\right\},
	\end{align*}
	for some $C_{2}>0$.
	Since $\alpha \in (0,1/2)$ and $p\in (0,\infty)$ are arbitrary, we conclude the proof.
\end{proof}


\subsection*{Acknowledgements}
The author would like to thank an anonymous referee for his/her careful readings and comments.
The author was supported by JSPS KAKENHI Grant Number 19K14552.
The author would like to thank Akihiro Tanaka for his valuable comments and discussions.


\begin{thebibliography}{99}
		
	\bibitem{AkIm14}
	{Akahori, J. and Imamura, Y.}
	{\it On a symmetrization of diffusion processes.}
	Quant. Finance
	{\bf 14}(7)
	1211--1216 (2014).
	
	
	\bibitem{Ar67}
	{Aronson, D.~G.}
	{\it Bounds for the fundamental solution of a parabolic equation.}
	Bull. Amer. Math. Soc.
	{\bf 73}
	890--896
	(1967).
	
	
	\bibitem{Av09}
	{Avikainen, R.}
	{\it On irregular functionals of SDEs and the Euler scheme.}
	Finance Stoch.
	{\bf 13}(3)
	381--401
	(2009).
	
	\bibitem{BaGyPa94}
	{Bally, V.}, {Gy\"ongy, I.} and {Pardoux, E.}
	{\it White noise driven parabolic SPDEs with measurable drift.}
	{J. Funct. Anal.}
	{\bf 120}
	484--510
	(1994).
	
	
	
	\bibitem{BaHuYu19}
	{Bao, J.}, {Huang, X.} and {Yuan, C.}
	{\it Convergence rate of Euler--Maruyama scheme for SDEs with H\"older-Dini continuous drifts.}
	J. Theoret. Probab.
	{\bf 32}(2)
	848--871
	(2019).
	
	\bibitem{Ba82}
	{Barlow, M.~T.}
	{\it One dimensional stochastic differential equations with no strong solution.}
	Lond. Math. Soc.
	{\bf 2}(2)
	335--347
	(1982).
	
	\bibitem{Bach01}
	{Bass, R.~F.} and {Chen, Z.~Q.}
	{\it Stochastic differential equations for Dirichlet processes.}
	Probab. Theory Relat. Fields
	{\bf 121}(3)
	422--446
	(2001).
	
	\bibitem{BeOu08}
	{Belfadli, R.} and {Ouknine, Y.}
	{\it On the pathwise uniqueness of solutions of stochastic differential equations driven by symmetric stable L\'evy processes.}
	Stochastics
	{\bf 80}(6)
	519--524
	(2008).
	
	\bibitem{BSW}
	{Bene\v{s}, V.~E., Shepp, L.~A. and Witsenhausen, H.~ S.}
	{\it Some solvable stochastic control problems.}
	Stochastics
	{\bf 4}
	39--83
	(1980).
	
	\bibitem{BlGe60}
	{Blumenthal, R.~M.} and {Getoor, R.~K.}
	{\it Some theorems on stable processes.}
	Trans. Amer. Math. Soc.
	{\bf 95}(2)
	263--273
	(1960).
	
	\bibitem{BuDaGe19}
	{Butkovsky, O.}, {Dareiotis, K.} and {Gerencs\'er, M.}
	{\it Approximation of SDEs - a stochastic sewing approach.}
	arXiv:1909.07961v1
	(2019).
	
	\bibitem{CaFeNu98}
	{Caballero, M.~E.}, {Fern\'andez, B.} and {Nualart, D.}
	{\it Estimation of densities and applications.}
	J. Theoret. Probab.
	{\bf 11}(3)
	831--851
	(1998).
	
	
	
	\bibitem{ChGa98}
	{Chistyakov, V.~V.} and {Galkin, O.~E.}
	{\it On maps of bounded $p$-variation with $p > 1$.}
	Positivity
	{\bf 2}
	19--45
	(1998).
	
	\bibitem{CoSa07}
	{Corns, T.~R.~A.} and {Satchell, S.~E.}
	{\it Skew Brownian motion and pricing European options.}
	Eur. J. Finance
	{\bf 13}(6)
	523--544
	(2007).
	
	\bibitem{DeGoSc06}
	{Decamps, M.}, {Goovaerts, M.} and {Schoutens, W.}
	{\it Self exciting threshold interest rates model.}
	Int. J. Theor. Appl. Finance
	{\bf 9}(7)
	1093--1122
	(2006).
	
	\bibitem{DeScGo04}
	{Decamps, M.}, {De Schepper, A.} and {Goovaerts, M.}
	{\it Applications of $\delta$-function perturbation to the pricing of derivative securities.} 
	Phys. A
	{\bf 342}
	677--692
	(2004).
	
	\bibitem{DeUs99}
	{Decreusefond, L.} and {\"Ust\"unel, A.~S.}
	{\it Stochastic analysis of the fractional Brownian motion.}
	Potential Anal.
	{\bf 10}(2)
	177--214
	(1999).
	
	
	\bibitem{EiKoKrLa19}
	{Eisenmann, M.}, {Kov\'acs, M.}, {Kruse, R.} and {Larsson, S.}
	{\it Error estimates of the backward Euler-Maruyama method for multi-valued stochastic differential equations.}
	arXiv:1906.11538
	(2019).
	
	\bibitem{Et06}
	{\'Etor\'e, P.}
	{\it On random walk simulation of one-dimensional diffusion processes with discontinuous coefficients.}
	Electron. J. Probab.
	{\bf 11}
	249--275
	(2006).
	
	\bibitem{EtMa13}
	{\'Etor\'e, P.} and {Martinez, M.}
	{\it Exact simulation of one-dimensional stochastic differential equations involving the local time at zero the unknown process.}
	Monte Carlo Methods Appl.
	{\bf 19}
	41--71
	(2013).
	
	\bibitem{FaKe81}
	{Fabes, E.~B.} and {Kenig, C.}
	{\it Examples of singular parabolic measures and singular transition probability densities.}
	Duke Math. J.
	{\bf 48}(4)
	845--855
	(1981).
	
	\bibitem{Fr64}
	{Friedman, A.}
	{\it Partial Differential Equations of Parabolic Type.}
	Dover Publications Inc.,
	(1964).
	
	\bibitem{Fr75}
	{Friedman, A.}
	{\it Stochastic differential equations and applications. Courier Corporation.}
	Academic Press
	New York
	(1975).
	
	\bibitem{Fri16}
	{Frikha, N.}
	{\it On the weak approximation of a skew diffusion by an Euler-type scheme.}
	Bernoulli
	{\bf 24}(3)
	1653--1691
	(2018).
	
	\bibitem{GaSh17}
	{Gairat, A.} and {Shcherbakov V.}
	{\it Density of skew Brownian motion and its functionals with application in finance.}
	Math. Finance
	{\bf 27}(4)
	1069--1088
	(2017).
	
	\bibitem{Gi08}
	{Giles, M.~B.}
	{\it Multilevel monte carlo path simulation.}
	Oper. Res.
	{\bf 56}(3)
	607--617
	(2008).
	
	\bibitem{Gi62}
	{Girsanov, I.~V.}
	{\it An example of non-uniqueness of the solution of the stochastic equation of K. Ito.}
	Theory Probab. Appl.
	{\bf 7}(3)
	325--331
	(1962).
	
	\bibitem{Gi84}
	{Giusti, E.}
	{\it Minimal surfaces and functions of bounded variation.}
	Monographs in Mathematics
	80
	Basel–Boston–Stuttgart
	Birkh\"auser Verlag
	(1984).
	
	
	\bibitem{Gy98}
	{Gy\"ongy, I.}
	{\it Lattice approximations for stochastic quasi-linear parabolic partial differential equations driven by space time white noise I.}
	Potential Anal.
	{\bf 9}
	1--25
	(1998).
	
	\bibitem{Gy99}
	{Gy\"ongy, I.}
	{\it Lattice approximations for stochastic quasi-linear parabolic partial differential equations driven by space time white noise II.}
	Potential Anal.
	{\bf 11}
	1--37
	(1999).
	
	\bibitem{GyKr96}
	{Gy\"ongy, I.} and {Krylov, N.}
	{\it Existence of strong solutions for It\^o's stochastic equations via approximations.}
	Probab. Theory Relat. Fields
	{\bf 105}
	143--158
	(1996).
	
	\bibitem{GyPa93a}
	{Gy\"ongy, I.} and {Pardoux, E.}
	{\it On quasi-linear stochastic partial differential equations.}
	Probab. Theory Relat. Fields
	{\bf 94}
	413--425
	(1993).
	
	\bibitem{GyPa93b}
	{Gy\"ongy, I.} and {Pardoux, E.}
	{\it On the regularization effect of space--time white noise on quasi-linear parabolic partial differential equations.}
	Probab. Theory Relat. Fields
	{\bf 97}
	211--229
	(1993).
	
	\bibitem{GyRa11}
	{Gy\"ongy, I. and R\'asonyi, M.}
	{\it A note on Euler approximations for SDEs with H\"older continuous diffusion coefficients.}
	Stochastic. Process. Appl. 
	{\bf 121}
	2189--2200
	(2011).
	
	\bibitem{HaHuJe15}
	{Hairer, M.}, {Hutzenthaler, M.} and {Jentzen, A.}
	{\it Loss of regularity for Kolmogorov equation.}
	Ann. Probab.
	{\bf 43}(2)
	468--527
	(2015).
	
	\bibitem{Ha13}
	{Halmos, P.~R.}
	{\it Measure theory, Second printing.}
	Springer.
	(1950).
	
	\bibitem{HaSh}
	{Harrison, J.~M.} and {Shepp, L.~A.}
	{\it On skew Brownian motion.}
	Ann. Probab.
	{\bf 9}
	309--313
	(1981).
	
	\bibitem{HeHeMu19}
	{Hefter, M.}, {Herzwurm, A.} and {M\'uller-Gronbach, T.}
	{\it Lower error bounds for strong approximation of scalar SDEs with non-Lipschitzian coefficients.}
	Ann. Appl. Probab.
	{\bf 29}(1)
	178--216
	(2019).
	
	\bibitem{HuJeKl11}
	{Hutzenthaler M.}, {Jentzen, A.} and {Kloeden, P.~E.}
	{\it Strong and weak divergence in finite time of Euler's method for stochastic differential equations with non-globally Lipschitz continuous coefficients.}
	Proc. R. Soc. A
	{\bf 467}
	1563--1576
	(2011).
	
	\bibitem{HuJeKl12}
	{Hutzenthaler, M.}, {Jentzen,  A.} and {Kloeden, P.~E.}
	{\it Strong convergence of an explicit numerical method for SDEs with nonglobally Lipschitz continuous coefficients.}
	Ann. Appl. Probab.
	{\bf 22}(4)
	1611--1641
	(2012).
	
	\bibitem{IkWa}
	{Ikeda, N.} and {Watanabe, S.}
	{\it Stochastic differential equations and diffusion processes, second ed.}
	volume~24 of North-Holland Mathematical Library,
	North-Holland Publishing Co.,
	Amsterdam-New York; Kodansha, Ltd.,
	Tokyo,
	(1981).
	
	\bibitem{JeMuYa16} 
	{Jentzen, A.}, {M\"uller-Gronbach, T.} and {Yaroslavtseva, L.}
	{\it On stochastic differential equations with arbitrary slow convergence rates for strong approximation.}
	Commun. Math. Sci.
	{\bf 14}(6)
	1477--1500
	(2016). 
	
	
	\bibitem{KS}
	{Karatzas, I.}  and  {Shreve, S.~E.}
	{\it Brownian motion and stochastic calculus. Second edition.}
	Springer
	(1991).
	
	\bibitem{KP}
	{Kloeden, P.~E.} and  {Platen, E.}
	{\it Numerical solution of stochastic differential equations.}
	Springer
	(1995).
	
	\bibitem{KoMa13}
	{Kohatsu-Higa, A.} and {Makhlouf, A.}
	{\it Estimates for the density of functionals of SDEs with irregular drift.}
	Stochastic Process. Appl.
	{\bf 123}(5)
	1716--1728
	(2013).
	
	\bibitem{KoMaNo14}
	{Kohatsu-Higa, A.}, {Makhlouf, A.} and {Ngo, H-L.}
	{\it Approximations of non-smooth integral type functionals of one dimensional diffusion processes.}
	Stochastic. Process. Appl.
	{\bf 124}(5)
	1881--1909
	(2014).
	
	\bibitem{KoTaZh16}
	{Kohatsu-Higa, A.}, {Taguchi, D.} and {Zhong, J.}
	{\it The parametrix method for skew diffusions.}
	Potential Anal.
	{\bf 45}
	299–-329
	(2016).
	
	\bibitem{KoTa12}
	{Kohatsu-Higa, A.} and {Tanaka, A.}
	{\it A Malliavin calculus method to study densities of additive functionals of SDE's with irregular drifts.}
	Ann. Inst. Henri Poincar\'e
	{\bf 48}(3),
	871--883,
	(2012).
	
	\bibitem{Ko82}
	{Komatsu, T.}
	{\it On the pathwise uniqueness of solutions of one-dimensional stochastic differential equations of jump type.}
	Proc. Japan Acad.,
	{\bf 58} Ser. A
	353--356
	(1982).
	
	
	\bibitem{Kr80}
	{Krylov, N.~V.}
	{\it Controlled diffusion processes.}
	Springer-Verlag Berlin
	(1980).
	
	
	\bibitem{Ku17}
	{Kusuoka, S.}
	{\it Continuity and Gaussian two-sided bounds of the density functions of the solutions to path-dependent stochastic differential equations via perturbation.}
	Stochastic Process. Appl.
	{\bf 127}
	359--384
	(2017).
	
	\bibitem{LeGall}
	{Le Gall, JF.}
	{\it One-dimensional stochastic differential equations involving the local times of the unknown process.}
	Stochastic Analysis and Applications, Springer, Berlin, Heidelberg
	51--82
	(1984). 
	
	\bibitem{Lee94}
	{Lee, S.}
	{\it Optimal drift on $[0,1]$.}
	Trans. Amer. Math. Soc.
	{\bf 346}(1)
	159--175
	(1994).
	
	\bibitem{Lejay07}
	{Lejay, A.}
	{\it On the constructions of the skew Brownian motion.}
	Probab. Surveys
	{\bf 3}
	413--466
	(2006).
	
	\bibitem{LeMa}
	{Lejay, A.} and {Martinez, M.}
	{\it A scheme for simulating one-dimensional diffusion processes with discontinuous coefficients.}
	Ann. Appl. Probab.
	{\bf 16}(1)
	107--139
	(2006).
	
	\bibitem{LeMe10}
	{Lemaire, V.} and {Menozzi, S.}
	{\it On some non-asymptotic bounds for the Euler scheme.}
	Electron J. Probab.
	{\bf 15}
	1645--1681
	(2010).
	
	\bibitem{LeSz}
	{Leobacher, G.} and {Sz\"olgyenyi, M.}
	{\it A numerical method for SDEs with discontinuous drift.}
	BIT Numer. Math.
	{\bf 56}(1)
	151--162
	(2015).
	
	\bibitem{LeSz17}
	{Leobacher, G.} and {Sz\"olgyenyi, M.}
	{\it A strong order $1/2$ method for multidimensional SDEs with discontinuous drift.}
	Ann. Appl. Probab.
	{\bf 27}(4)
	2383-2418
	(2017).
	
	\bibitem{LeSz17b}
	{Leobacher, G.} and  {Sz\"olgyenyi, M.}
	{\it Convergence of the Euler--Maruyama method for multidimensional SDEs with discontinuous drift and degenerate diffusion coefficient.}
	Numerische Mathematik.
	{\bf 138}(1)
	219--239
	(2018).
	
	
	
	\bibitem{MaTa12}
	{Martinez, M.} and {Talay, D.}
	{\it One-dimensional parabolic diffraction equations: pointwise estimates and discretization of related stochastic differential equations with weighted local times.} Electron. J. Probab.
	{\bf 17}(27)
	1--32
	(2012).
	
	
	\bibitem{MeTa}
	{Menoukeu Pamen, O.} and {Taguchi, D.}
	{\it Strong rate of convergence for the Euler--Maruyama approximation of SDEs with H\"older continuous drift coefficient.}
	Stochastic Process. Appl.
	{\bf 127}(8)
	2542--2559
	(2017).
	
	
	\bibitem{MuYa19}
	{M\"uller-Gronbach, T.} and {Yaroslavtseva, L.}
	{\it A strong order $3/4$ method for SDEs with discontinuous drift coefficient.} arXiv:1904.09178
	(2019).
	
	\bibitem{Nakao}
	{Nakao, S.}
	{\it On the pathwise uniqueness of solutions of one-dimensional stochastic differential equations.}
	Osaka J. Math.
	{\bf 9}
	513--518
	(1972).
	
	\bibitem{Nak16}
	{Nakatsu, T.}
	{\it Integration by parts formulas concerning maxima of some SDEs with applications to study on density functions.}
	Stoch. Anal. Appl.
	{\bf 34}(2)
	293--317
	(2016).
	
	\bibitem{NgoLu17}
	{Ngo, H-L.} and {Luong, D-T.}
	{\it Strong rate of tamed Euler–Maruyama approximation for stochastic differential equations with H\"older continuous diffusion coefficient.}
	Braz. J. Probab. Stat.
	{\bf 31}(1)
	24--40
	(2017).
	
	\bibitem{NgoOga11}
	{Ngo, H-L}. and {Ogawa, S.}
	{\it On the discrete approximation of occupation time of diffusion processes.} 
	Electron. J. Stat.
	5
	1374--1393
	(2011).
	
	\bibitem{NgTa16}
	{Ngo, H-L.} and {Taguchi, D.}
	{\it Strong rate of convergence for the Euler--Maruyama approximation of stochastic differential equations with irregular coefficients.}
	Math. Comp.
	{\bf 85}(300)
	1793--1819
	(2016).
	
	\bibitem{NgTa1701}
	{Ngo, H-L.} and {Taguchi, D.}
	{\it On the Euler--Maruyama approximation for one-dimensional stochastic differential equations with irregular coefficients.}
	IMA J. Numer. Anal.
	{\bf 37}
	1864--1883
	(2017).
	
	\bibitem{NgTa1702}
	{Ngo, H-L.} and {Taguchi, D.}
	{\it Strong convergence for the Euler--Maruyama approximation of stochastic differential equations with discontinuous coefficients.}
	Statist. Probab. Lett.
	{\bf 125}
	55--63
	(2017).
	
	\bibitem{NT2}
	{Ngo, H-L.} and {Taguchi, D.}
	{\it Approximation for non-smooth functionals of stochastic differential equations  with irregular drift.}
	J. Math. Anal. Appl.
	{\bf 457}(1),
	361--388,
	(2018).
	
	\bibitem{NoSi06}
	{Nourdin, I.} and {Simon, T.}
	{\it On the absolute continuity of L\'evy processes with drift.}
	Ann. Probab.
	{\bf 34}(3)
	1035--1051
	(2006).
	
	\bibitem{Nu06}
	{Nualart, D.}
	{\it The Mallivain calculus and related topics.}
	Springer
	(2006).
	
	\bibitem{NuOu02}
	{Nualart, D.} and {Ouknine, Y.}
	{\it Regularization of differential equations by fractional noise.}
	Stochastic Process. Appl.
	{\bf 102}
	103–116
	(2002).
	
	\bibitem{Ok98}
	{\O ksendal, B.}
	{\it Stochastic differential equations: an introduction with applications, sixth edn.}
	Springer
	(1998).
	
	\bibitem{Ra11}
	{Ramirez, J.}
	{\it Multi-skewed Brownian motion and diffusion in layered media.}
	Proc. Amer. Math. Soc.
	{\bf 139}(10)
	3739--3752
	(2011).
	
	\bibitem{Ru74}
	{Rudin, W.}
	{\it Real and complex analysis, 2nd edn.}
	McGraw-Hill, New York
	(1974).
	
	\bibitem{Sa13}
	{Sabanis, S.}
	{\it A note on tamed Euler approximations.}
	Electron. Commun. Probab.
	{\bf 18}(47)
	1--10
	(2013).
	
	\bibitem{Sa16}
	{Sabanis, S.}
	{\it Euler approximations with varying coefficients: the case of superlinearly growing diffusion coefficients.}
	Ann. Appl. Probab.,
	{\bf 25}(4)
	2083--2105
	(2016).
	
	\bibitem{SaKiMa}
	{Samko, G.~S.}, {Kilbas, A.~A.} and {Marichev, O.~I.}
	{\it Fractional integrals and derivatives-Theory and Applications.}
	Gordon and Breach Science Publishers,
	New York
	(1993).
	
	\bibitem{Sato}
	{Sato, K.}
	{\it L\'evy processes and infinitely divisible distributions.}
	Cambridge University Press
	(2011).
	
	\bibitem{Sh94}
	{Shiga, T.}
	{\it Two contrasting properties of solutions for one-dimensional stochastic partial differential equations.}
	Can. J. Math.
	{\bf 46}(2)
	415--437.
	(1994).
	
	
	\bibitem{TaTa18}
	{Taguchi D.} and {Tanaka, A.}
	{\it Probability density function of SDEs with unbounded and path--dependent drift coefficient.}
	arXiv:1811.07101v2.
	
	\bibitem{TsHa13}
	{Tsuchiya, T.} and {Hashimoto, H.}
	{\it On the convergent rates of Euler--Maruyama schemes for SDEs driven by rotation invariant α-stable processes.}
	RIMS Kokyuroku
	{\bf 1855}
	229--235
	(2013) in Japanese.
	
	\bibitem{Wa86}
	{Walsh, J.~B.}
	{\it An introduction to stochastic partial differential equations.}
	In Hennequin, P.~L. (ed) Ecole d'\'et'e de Probabilit\'es de St. Flour XIV.
	(Lect. Notes Math., vol 1180, 265--437)
	Berlin, Heidelberg New York,
	Springer
	(1986).
	
	\bibitem{Wi91}
	{Williams, D.}
	{\it Probability with Martingales.}
	Cambridge University Press
	(1991).
	
	\bibitem{YaWa}
	{Yamada, T.} and {Watanabe, S.}
	{\it On the uniqueness of solutions of stochastic differential equations.}
	J. Math. Kyoto Univ.
	{\bf 11}
	155--167
	(1971).
	
	
	\bibitem{Ya02}
	{Yan, B. L.}
	{\it The Euler scheme with irregular coefficients.}
	Ann. Probab.
	{\bf 30}(3)
	1172--1194
	(2002).
	
	\bibitem{Yar17}
	{Yaroslavtseva, L.}
	{\it On non-polynomial lower error bounds for adaptive strong approximation of SDEs.}
	J. Complexity
	{\bf 42}
	1--18
	(2017).
	
	\bibitem{Zh19}
	{Zhang, X.}
	{\it A discretized version of Krylov's estimate and its applications.}
	arXiv 1909.09976
	(2019).
	
\end{thebibliography}
\end{document}